%% file: surfaces_covers.tex
\documentclass[11pt]{amsart}

\usepackage[T1]{fontenc}
\usepackage[utf8]{inputenc}
\usepackage{lmodern}
\usepackage{adjustbox}
\usepackage[american]{babel}
\usepackage[babel, final]{microtype}
\usepackage{amssymb}
\usepackage[all]{xy}

\usepackage{mathtools}

\usepackage{ifpdf}
\usepackage{comment}
\usepackage{multirow}

\usepackage{csquotes}

\usepackage{booktabs}

\usepackage[pdftex, dvipsnames]{xcolor}
\usepackage{caption}
\usepackage{subcaption}
\usepackage{pinlabel}
\usepackage[pdftex,%
    a4paper,%
    includehead,%
    includefoot,%
    nomarginpar,%
    lmargin=1in,%
    rmargin=1in,%
    tmargin=1in,%
    bmargin=1in,%
]{geometry}
\usepackage[pdftex,%
	final,%
	colorlinks=true,%
	linkcolor={red!45!black},%
	citecolor=NavyBlue,%
	filecolor=NavyBlue,%
	menucolor=NavyBlue,%
	urlcolor={blue!45!black},%
	bookmarks=true,%
	bookmarksdepth=3,%
	bookmarksnumbered=true,%
	bookmarksopen=true,%
	bookmarksopenlevel=2,%
]{hyperref}
\usepackage{cleveref}
\hypersetup{
    pdftitle={Algorithms for Computing Invariants of Trisected Branched Covers},
    pdfauthor={Cahn, Matic, Pinzon-Caicedo, Ruppik},
    pdfsubject={Algorithms for Computing Invariants of Trisected Branched Covers},
    pdfkeywords={Knotted surfaces, branched covers, group bridge trisections, Reidemeister-Schreier algorithm}
}

\usepackage[
	backend=biber,
	style=alphabetic,
	sorting=nyt,
	giveninits=true,
	maxnames=10,
	url=false,
	backref=true
]{biblatex}
\DefineBibliographyStrings{english}{
	backrefpage={$\uparrow$},
	backrefpages={$\uparrow$}
}

\DeclareFieldFormat[article,incollection,inproceedings]{title}{#1}
\DeclareFieldFormat[unpublished]{title}{\textit{#1}}

\newtheorem{algorithm}{Algorithm}

\newenvironment{nalign}{
    \begin{equation}
    \begin{aligned}
}{
    \end{aligned}
    \end{equation}
    \ignorespacesafterend
}

\newcommand{\pathelement}[1]{\widehat{#1}}
\newcommand{\loopelement}[1]{{#1}} 

\makeatletter
\def\@tocline#1#2#3#4#5#6#7{\relax
	\ifnum #1>\c@tocdepth 
	\else
	\par \addpenalty\@secpenalty\addvspace{#2}%
	\begingroup \hyphenpenalty\@M
	\@ifempty{#4}{%
		\@tempdima\csname r@tocindent\number#1\endcsname\relax
	}{%
		\@tempdima#4\relax
	}%
	\parindent\z@ \leftskip#3\relax \advance\leftskip\@tempdima\relax
	\rightskip\@pnumwidth plus4em \parfillskip-\@pnumwidth
	#5\leavevmode\hskip-\@tempdima
	\ifcase #1
	\or\or \hskip 1em \or \hskip 2em \else \hskip 3em \fi%
	#6\nobreak\relax
	\dotfill\hbox to\@pnumwidth{\@tocpagenum{#7}}\par
	\nobreak
	\endgroup
	\fi}
\makeatother

\usepackage{tikz,tikz-cd}
\usetikzlibrary{patterns,knots,arrows.meta,decorations.markings}
\usetikzlibrary{tikzmark}
\usetikzlibrary{decorations.pathmorphing}

\tikzset{>={Straight Barb[scale=0.85]}}
\tikzcdset{
  cells={font=\everymath\expandafter{\the\everymath\displaystyle}},
  arrow style=tikz,
  diagrams={>={Straight Barb[scale=0.85]}},
  every label/.append style = {font = \small}
}

\usepackage{tikzsymbols}

\usepackage{overpic}

\input{commands}

\usepackage{todonotes}

\title{Algorithms for Computing Invariants of Trisected Branched Covers}

\author[Cahn]{Patricia Cahn}
\address{Smith College, Department of Mathematical Sciences, Northampton, MA 01063}
\email{\href{mailto:pcahn@smith.edu}{pcahn@smith.edu}}
\urladdr{\url{https://www.science.smith.edu/~pcahn}}

\author[Matic]{Gordana Matic}
\address{University of Georgia, Department of Mathematics, Athens, GA 30602}
\email{\href{gordanam@math.uga.edu}{gordanam@math.uga.edu}}

\author[Ruppik]{Benjamin Ruppik}
\address{Heinrich-Heine-University, Düsseldorf, Germany}
\email{\href{mailto:ruppik@hhu.de}{benjamin.ruppik@hhu.de}}
\urladdr{\url{https://bruppik.de}}

\def\subjclassname{\textup{2020} Mathematics Subject Classification}
\expandafter\let\csname subjclassname@1991\endcsname=\subjclassname
\expandafter\let\csname subjclassname@2000\endcsname=\subjclassname
\subjclass{
    57K40, 
    57M05. 
    \hfill
    Date: \today
}

\begin{document}

\begin{abstract}
    We give diagrammatic algorithms for computing the group trisection, homology groups, and intersection form of a closed, orientable, smooth 4-manifold, presented as a branched cover of a bridge-trisected surface in $\sphere{4}$.
    The algorithm takes as input a tri-plane diagram, labelled with permutations according to the Wirtinger relations.
    We apply our algorithm to several examples, including dihedral and cyclic covers of spun knots, cyclic covers of Suciu's ribbon knots with the trefoil knot group, and an infinite family of irregular covers of the Stevedore disk double.
    As an application, we give a fully automated algorithm for computing Kjuchukova's homotopy-ribbon obstruction for a $p$-colorable knot, given an extension of that coloring over a ribbon surface in the 4-ball.
\end{abstract}

\maketitle

\section{Introduction}
\label{sec:intro}

Branched covers are a rich source of invariants of classical knots in dimension 3, and knotted surfaces in dimension 4.
For example, the dihedral linking invariant \cite{perko1964invariant} played a key role in the creation of early knot tables; signature invariants derived from cyclic and dihedral covers give obstructions to a knot being slice or ribbon \cite{kauffman1976signature, trotter1962homology, murasugi1965certain, casson1986cobordism, kjuchukova2018dihedral}; branched covers were recently shown to distinguish Suciu's family of ribbon knots with the same knot group \cite{suciu1985ribbon,kanenobu2020suciu}; and branched covers were also recently used show that push-ins of two minimal-genus Seifert surfaces for the same classical knot are not topologically isotopic in the 4-ball rel boundary, answering a 40-year-old question of Livingston \cite{hayden2022seifert}.  
In addition, just as every closed, connected, oriented 3-manifold is a 3-fold simple cover of $\sphere{3}$ branched along a knot \cite{hilden1974every,montesinos1974representation}, every such 4-manifold is a 4-fold simple cover of $\sphere{4}$ along an immersed surface \cite{piergallini1995four}, or 5-fold cover of $\sphere{4}$ along an embedded surface \cite{iori20024}.

This motivates the problem of finding an efficient, automated, and completely general method to compute invariants of a branched cover of $\sphere{4}$ or $\disk{4}$, including $\pi_1$, homology groups, and the intersection form, along a (typically smoothly embedded) surface, for not just cyclic covers, but arbitrary irregular covers.  We carry this out for smoothly embedded surfaces in $\sphere{4}$, and as an immediate byproduct of our construction, we also compute invariants of branched covers of ribbon surfaces in $\disk{4}$.  A major goal of the project is to automate the computation of these invariants, so we implement our algorithms; the code is available at \cite{Ruppik_Knot_groups_2021}.

Existing methods for computing such invariants require first constructing a Kirby diagram of the cover, and have been carried out for cyclic covers of certain surfaces in $\sphere{4}$ and ribbon surfaces $\disk{4}$ in \cite{akbulutkirby1980branchedcoverssurfaces} (see also \cite[Sec. 6.8]{gompfstipsicz1994fourmanifolds}), and for arbitrary coverings of 4-manifolds in \cite[Section 11.1]{akbulut2016fourmanifoldsbook}.
To the best of our knowledge, these techniques have not been implemented, and would be difficult to carry out on large families of examples.

We take advantage of the fact that branched covers of surfaces are particularly easy to describe diagrammatically for {\it trisected} 4-manifolds along {\it bridge-trisected} surfaces.  Trisections of 4-manifolds were introduced by Gay and Kirby in \cite{gay2016trisecting} as a 4-dimensional analogue of Heegaard splittings of 3-manifolds.  Bridge trisections of surfaces were introduced by Meier and Zupan in \cite{meier2017bridgeS4,meier2018bridge4manifolds} as a 4-dimensional analogue of a bridge-position knot or link.  
Just as a branched cover of a bridge-position link in $\sphere{3}$ admits a natural choice of Heegaard splitting, a branched cover of $\sphere{4}$ along a bridge-trisected surface admits a natural trisection.

Every branched cover of a surface $F\subset\sphere{4}$ is determined by a permutation representation $\rho \colon \pi_1(\sphere{4}-F)\rightarrow S_n$. Bridge-trisected surfaces are described diagrammatically by tri-plane diagrams \cite{meier2017bridgeS4}. Given a tri-plane diagram for $F$, every such permutation representation, and therefore, every branched cover of $F$, is described combinatorially by a permutation-labelling of the arcs of a tri-plane diagram, such that the Wirtinger relations are satisfied at each crossing, and such that the permutations along the endpoints of each tangle agree.

In this paper, we give algorithms for computing the algebraic topology of a branched cover of a bridge-trisected surface from this tri-plane data: three braid words that describe a tri-plane diagram, and a list of permutations along the endpoints of the tangles. Implementations of our algorithms can be found at \cite{Ruppik_Knot_groups_2021}.  The data we compute includes the group trisection of the fundamental group the branched cover $X$ (\Cref{RS.alg}), which determines a corresponding trisection of $X$, and hence determines $X$ up to diffeomorphism \cite{abrams2018group}.  From the group trisection of $\pi_1(X)$, we compute the homology groups of $X$ (\Cref{homology.alg}) and the intersection form of $X$ (\Cref{intersectionform.alg}).  The latter computations can be viewed as bridge-trisected analogues of the algorithms in \cite{feller2018homology,florens2022torsions}  for computing the homology and intersection form of a trisected 4-manifold from its trisection diagram.

We then carry out these algorithms on families of permutation-labelled bridge-trisected surfaces, yielding both invariants of these surfaces and new examples of trisected 4-manifolds realized as branched covers of the standard genus 0 trisection of $\sphere{4}$.  These include cyclic and dihedral covers of spun torus knots, whose covers can in some cases be viewed as spun manifolds; an infinite family of simply-connected irregular covers of the double of the standard Stevedore ribbon disk, corresponding to non-co-Hopfian subgroups of its group BS(1,2); Suciu's family of ribbon knots with the same exterior, which can be distinguished by the fundamental groups of their cyclic branched covers \cite{suciu1985ribbon, kanenobu2020suciu}; and an infinite family of dihedral covers of singular surfaces whose covers have unbounded signature.  The latter family illustrates another application of our work, as it yields a new algorithm for computing Kjuchukova's invariant $\Xi_p(K,\rho)$ of a knot $K$ equipped with a $p$-coloring $\rho$ \cite{kjuchukova2018dihedral}, which give an obstruction to a knot being homotopy ribbon \cite{geske2021signatures,cahn2017singular,cahn2018computing}. The algorithm takes as input a tri-plane presentation for a ribbon surface with boundary $K$, over which the coloring $\rho$ extends.

\subsection*{Outline}

In \Cref{sec:trisection} we review trisections, bridge trisections, and tri-plane diagrams.
In \Cref{sec:group_trisection}, we review group trisections and permutation-labelled tri-plane diagrams, our method for diagrammatically presenting an arbitrary branched cover of $\sphere{4}$.
Our algorithm for computing the group trisection of the branched cover of $\sphere{4}$ along a bridge-trsected surface, \Cref{RS.alg}, is given in \Cref{sec:trisected_RS}.
\Cref{homology:sec} gives \Cref{homology.alg} and \Cref{intersectionform.alg} for computing the homology groups and intersection form of the branched cover, respectively. 
We use the 3-colored spun trefoil, whose branched cover is $\sphere{2}\times \sphere{2}$, as a running example throughout \Cref{sec:trisected_RS} and \Cref{homology:sec}.  Examples and applications are given in \Cref{sec:twist_spun}, \Cref{stevedore.sec}, \Cref{sec:suciu}, and \Cref{xi.sec}.

\subsection*{Acknowledgments}

This project started at the 2020 Virtual Winter Trisectors Meeting, which was financially supported by the NSF.
 We are particularly grateful to Akram Alishahi and Juanita Pinz\'on Caicedo, for many thoughtful conversations throughout the early stages of the project. We are also grateful to the workshop organizers David Gay, Jeffrey Meier, and Alex Zupan.  Special credit should also go to Jason Joseph and Arunima Ray. PC was supported by NSF DMS-2145384.  BR was supported by the Max Planck Institute for Mathematics in Bonn.

\section{Trisections and Bridge Trisections}
\label{sec:trisection}

\subsection{Trisections} A $(g;k_1,k_2,k_3)-$\emph{trisection} of a smooth, compact, orientable, closed 4-manifold $X$ is a decomposition of $X$ into three standard 4-dimensional 1-handlebodies
\begin{equation*}
	X = X_{1} \cup X_{2} \cup X_{3}
\end{equation*}
with conditions on the pairwise and triple intersections as follows. The 4-dimensional 1-handlebodies are diffeomorphic to the boundary connected sums $X_{i} \cong \bcs^{k_{i}} \sphere{1} \times \disk{3}$, with pairwise intersections $X_{i} \cap X_{j} \cong \bcs^{g} \sphere{1} \times \disk{2}$ being  3-dimensional handlebodies $H_{\alpha}, H_{\beta}$ and $H_{\gamma}$, and triple intersection $X_{1} \cap X_{2} \cap X_{3} \cong \Sigma_{g}$ the \emph{central surface} of the trisection a genus $g$ surface $\Sigma_g$.
For each $i$, the pairwise union of the 3-dimensional handlebodies adjacent to $X_{i}$ determine a Heegaard splitting of $\partial X_{i} \cong \cs^{k_{i}} \sphere{1} \times \sphere{2}$. The pieces of a trisection, together with the gluing information, are illustrated in the commutative cube of topological spaces in \Cref{fig:trisection_diagram} (left). 
\index{trisection}
\index{central surface}

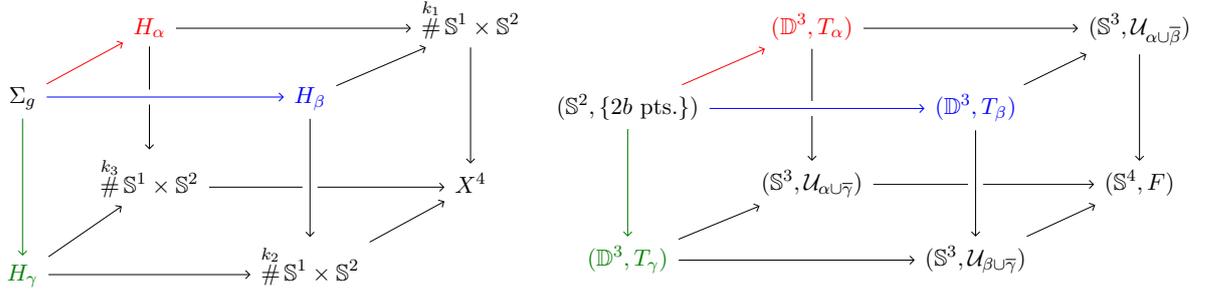
\begin{figure}
\begin{subfigure}[b]{.5\textwidth}
\adjustbox{scale=.8,center}{
    \begin{tikzcd}[row sep=scriptsize, column sep=scriptsize]
        & 
        \color{red}{H_{\alpha}}
        \arrow[rr] \arrow[dd] & &
        {
            \cs^{k_{1}} \sphere{1} \times \sphere{2} 
        }
        \arrow[dd] \\
        \color{black}{\Sigma_{g}}
        \arrow[ur, red] \arrow[rr, crossing over, blue] \arrow[dd, myGreen]
        & &
        \color{blue}{H_{\beta}}
        \ar[ur]
        \\
        & 
        {
            \cs^{k_{3}} \sphere{1} \times \sphere{2} 
        }
        \arrow[rr] & &
        X^{4} \\
        \color{myGreen}{H_{\gamma} }
        \arrow[rr] \arrow[ur]
        & &
        {
            \cs^{k_{2}} \sphere{1} \times \sphere{2} 
        }
        \arrow[from=uu, crossing over] \arrow[ur]
    \end{tikzcd}}
    \end{subfigure}%
    ~
    \begin{subfigure}[b]{.5\textwidth}
    \adjustbox{scale=.8,center}{\begin{tikzcd}[row sep=scriptsize, column sep=scriptsize]
		& 
		\color{red}{(\disk{3}, T_{\alpha})}
		\arrow[rr] \arrow[dd] & &
		{
			(\sphere{3}, \mathcal{U}_{\alpha \cup \overline{\beta}}) 
		}
		\arrow[dd] \\
		\color{black}{(\sphere{2}, \{2b \text{ pts.}\})}
		\arrow[ur, red] \arrow[rr, crossing over, blue] \arrow[dd, myGreen]
		& &
		\color{blue}{(\disk{3}, T_{\beta})}
		\ar[ur]
		\\
		& 
		{
			(\sphere{3}, \mathcal{U}_{\alpha \cup \overline{\gamma}}) 
		}
		\arrow[rr] & &
		(\sphere{4}, F) \\
		\color{myGreen}{(\disk{3}, T_{\gamma})}
		\arrow[rr] \arrow[ur]
		& &
            {
			(\sphere{3}, \mathcal{U}_{\beta \cup \overline{\gamma}}) 
		}
		\arrow[from=uu, crossing over] \arrow[ur]
	\end{tikzcd}
    }
    \end{subfigure}
    
    \caption{
        The pieces in the definition of a (parameterized) trisection of $X^{4}$ (left) and pieces of a (parameterized) $b$-bridge trisection of a smoothly embedded closed surface $F$ in the 4-sphere (right), as a commutative cube of (pairs of) topological spaces.
        \label{fig:trisection_diagram}
    }
\end{figure}

Gay and Kirby \cite{gay2016trisecting} show that every closed, connected, oriented 4-manifold has a trisection, which is unique up to a stabilization operation that increases the genus of the central surface and of the handlebodies.

The most basic example is the unique genus 0 trisection of the 4-sphere, which we can describe using the standard open book decomposition of $\sphere{4}$ with binding an unknotted $\sphere{2}$ and whose pages are 3-balls, as follows:
Cut the 4-sphere into the three sectors consisting of the union of pages 
\begin{equation*}
	X_{1} = \disk{3}_{[0, \frac{2\pi}{3}]},
	X_{2} = \disk{3}_{[\frac{2\pi}{3}, \frac{4\pi}{3}]}
	\textup{ and } 
	X_{3} = \disk{3}_{[\frac{4\pi}{3}, 0]}.
\end{equation*}
Their pairwise intersections are in the extremal 3-ball pages, the unions of which give genus 0 Heegaard splittings of the $3$-sphere.

\subsection{Bridge trisections} Bridge trisections, introduced by Meier and Zupan \cite{meier2017bridgeS4}, are the 4-dimensional analogue of bridge splittings of links in 3-manifolds.
Recall that a \emph{bridge splitting} of a link $L$ in a 3-manifold $M$ is a decomposition of the manifold-link  pair $(M,L)$ into a pair consisting of a Heegaard decomposition $M=H_1 \cup \overline{H}_2$ and  a pair of trivial tangles $T_i \subset H_i$ so that $\partial T_1 = \partial T_2 \subset \Sigma = \partial H_1=\partial H_2$ and $L=T_1 \cup \overline{T}_2 $.

 Recall that a $b$-bridge \emph{trivial tangle} $T$ in a handlebody $H$ is a collection of $b$ properly embedded arcs in $H$ such that all the arcs can be simultaneously isotoped into the boundary of $H$, while fixing $\partial T$, and a $b$-bridge splitting consists of two $b$-bridge trivial tangles.

\index{handlebody}
\index{trivial!tangle}

A \emph{bridge trisection} of a smoothly embedded closed surface $F \subset \sphere{4}$ is a decomposition into a union of three trivial 2-disk systems,
\begin{equation*}
	(\sphere{4}, F) = (X_{1}, D_{1}) \cup (X_{2}, D_{2}) \cup (X_{3}, D_{3}),
\end{equation*}
where the pairwise and triple intersections are required to be of a special form, as we now explain.
The $X_{i} \cong \disk{4}$ are the 4-balls of the previously discussed standard genus $0$ trisection of $\sphere{4}$, and each 2-disk system $D_{i} \colon \coprod^{c_{i}} \disk{2} \hookrightarrow \disk{4}$ is a properly embedded \emph{trivial disk system}.
Here trivial means that the disks are simultaneously isotopic into the 3-dimensional boundary $\partial X_{i}$, fixing the boundary $\partial D_{i}$ pointwise.
The boundary of the disks is further assumed to be an endpoint-union of $b$-bridge trivial tangles in 3-balls, each tangle living in one half of the boundary $\partial X_{i} \cong \disk{3} \cup_{\sphere{2}} \disk{3}$.
We denote the three tangles by $T_\alpha$, $T_\beta$, and $T_\gamma$. Each pairwise endpoint-union $T_{\mu} \cup_{\partial} \overline{T}_{\nu}$ forms an unlink $\mathcal{U}_{\mu \cup \overline{\nu}}$ in one of the 3-spheres $\partial X_{i} \cong \sphere{3}$.
The bridge trisection has four parameters $(b; c_{1}, c_{2}, c_{3})$, where $b$ is the \emph{bridge number}, and $c_{i}$ is the number of trivial disks in the disk system of that sector.

The pieces of a bridge trisection of a surface in $\sphere{4}$ and their interactions are illustrated in the commutative cube of pairs of topological spaces in \Cref{fig:trisection_diagram} (right).

Bridge trisections exist more generally for every smoothly knotted, not necessarily orientable surface embedded in a trisected connected orientable closed smooth 4-manifold \cite{meier2018bridge4manifolds}.
The pieces of a bridge trisected surface in a general trisected 4-manifold X and their interactions are illustrated in the commutative cube of pairs of topological spaces in \Cref{fig:genus_g_bridge_trisection_diagram}. 

\begin{figure}[!htbp]
    \begin{tikzcd}[row sep=scriptsize, column sep=scriptsize]
        & 
        \color{red}{(H_{\alpha}, \mathcal{T}_{\alpha})}
        \arrow[rr] \arrow[dd] & &
        \substack{
            (\cs^{k_{1}} \sphere{1} \times \sphere{2}, \; \mathcal{U}_{\alpha \cup \overline{\beta}}) \\
            \cong \partial (X_1, \; D_1=D_{\alpha \cup \overline{\beta}})
        }
        \arrow[dd] \\
        \color{black}{(\Sigma_g, \{2b \text{ pts.}\})}
        \arrow[ur, red] \arrow[rr, crossing over, blue] \arrow[dd, myGreen]
        & &
        \color{blue}{( H_{\beta}, \mathcal{T}_{\beta})}
        \ar[ur]
        \\
        & 
        \substack{
            (\cs^{k_{3}} \sphere{1} \times \sphere{2}, \; \mathcal{U}_{\alpha \cup \overline{\gamma}}) \\
            \cong \partial (X_3, \; D_3=D_{\alpha \cup \overline{\gamma}})
        }
        \arrow[rr] & &
        (X, \mathcal{S}) \\
        \color{myGreen}{(H_{\gamma}, \mathcal{T}_{\gamma})}
        \arrow[rr] \arrow[ur]
        & &
        \substack{
            (\cs^{k_{2}} \sphere{1} \times \sphere{2}, \; \mathcal{U}_{\beta \cup \overline{\gamma}}) \\
            \cong \partial (X_2, \; D_2=D_{\beta \cup \overline{\gamma}})
        }
        \arrow[from=uu, crossing over] \arrow[ur]
    \end{tikzcd}
    \caption[Pieces of a (parameterized) $b$-bridge trisection]{
        The pieces of a (parameterized) $b$-bridge trisection of a smoothly embedded closed surface $F$ in the trisected 4-manifold X, as a commutative cube of pairs of topological spaces.
        \label{fig:genus_g_bridge_trisection_diagram}
    }
\end{figure}
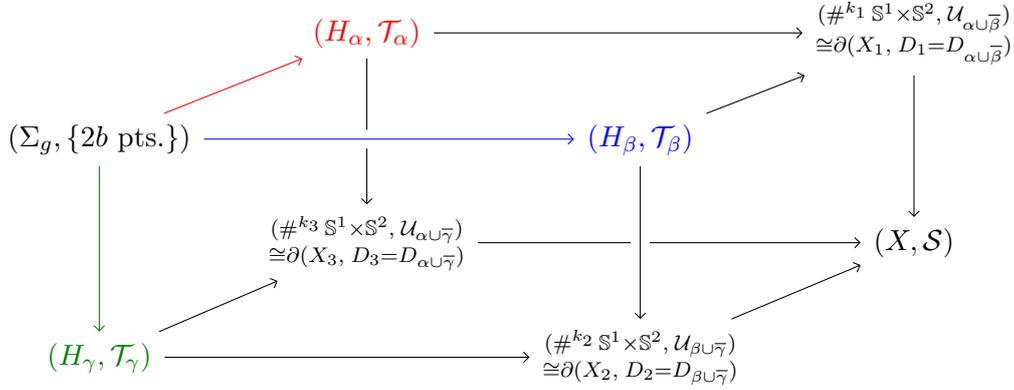

\tikzcdset{
  two tails/.code={\pgfsetarrows{tikzcd double to[reversed]-tikzcd to}},
}

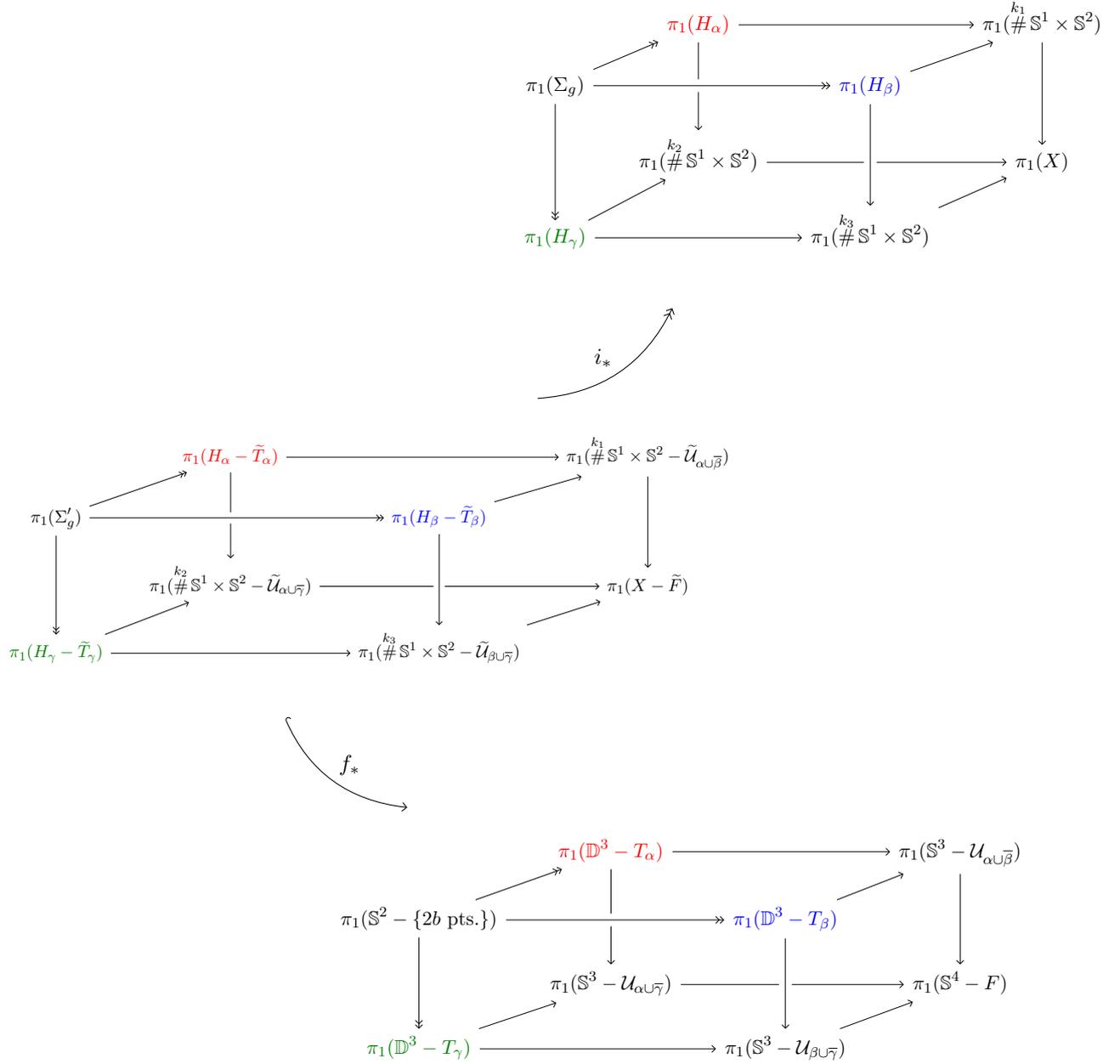
\begin{figure}[!htbp]
\begin{flushright}
\begin{subfigure}[b]{.5\textwidth}
	\adjustbox{scale=.8,left}{\begin{tikzcd}[row sep=scriptsize, column sep=scriptsize]
		& 
		\color{red}{\pi_1(H_\alpha)}
		\arrow[rr] \arrow[dd] & &
		{
			\pi_{1}(\cs^{k_{1}} \sphere{1} \times \sphere{2} ) 
		}
		\arrow[dd] \\
		\color{black}{\pi_1(\Sigma_g )}
		\arrow[ur, two heads, black] \arrow[rr, two heads, crossing over, black] \arrow[dd, two heads, black]
		& &
		\color{blue}{\pi_1(H_\beta)}
		\ar[ur]
		\\
		& 
		{
			\pi_{1}(\cs^{k_{2}} \sphere{1} \times \sphere{2}) 
		}
		\arrow[rr] & &
		\pi_{1}(X ) \\
		\color{myGreen}{\pi_1(H_\gamma)}
		\arrow[rr] \arrow[ur]
		& &
		{
			\pi_{1}(\cs^{k_{3}} \sphere{1} \times \sphere{2}) 
		}
		\arrow[from=uu, crossing over] \arrow[ur] \\
	\end{tikzcd}}
 \end{subfigure}
 \end{flushright}
 \begin{flushright}
    \begin{subfigure}[t]{.5\textwidth}
     \begin{tikzcd}
        && \quad \\
        &&\\
        \arrow[rruu, bend right,two heads,"i_*"]&&
     \end{tikzcd}
 \end{subfigure}
 \end{flushright}
 ~
\begin{flushleft}
\begin{subfigure}[b]{.5\textwidth}
	\adjustbox{scale=.7,left}{\begin{tikzcd}[row sep=scriptsize, column sep=scriptsize]
		& 
		\color{red}{\pi_1(H_\alpha - \widetilde{T}_{\alpha})}
		\arrow[rr] \arrow[dd] & &
		{
			\pi_{1}(\cs^{k_{1}} \sphere{1} \times \sphere{2} - \widetilde{\mathcal{U}}_{\alpha \cup \overline{\beta}}) 
		}
		\arrow[dd] \\
		\color{black}{\pi_1(\Sigma_g')}
		\arrow[ur, two heads, black] \arrow[rr, two heads, crossing over, black] \arrow[dd, two heads, black]
		& &
		\color{blue}{\pi_1(H_\beta - \widetilde{T}_{\beta})}
		\ar[ur]
		\\
		& 
		{
			\pi_{1}(\cs^{k_{2}} \sphere{1} \times \sphere{2}-\widetilde{\mathcal{U}}_{\alpha \cup \overline{\gamma}}) 
		}
		\arrow[rr] & &
		\pi_{1}(X - \widetilde{F}) \\
		\color{myGreen}{\pi_1(H_\gamma - \widetilde{T}_{\gamma})}
		\arrow[rr] \arrow[ur]
		& &
		{
			\pi_{1}(\cs^{k_{3}} \sphere{1} \times \sphere{2}-\widetilde{\mathcal{U}}_{\beta \cup \overline{\gamma}}) 
		}
		\arrow[from=uu, crossing over] \arrow[ur] \\
	\end{tikzcd}}
 \end{subfigure}
 \end{flushleft}
 ~
 \begin{subfigure}[t]{.5\textwidth}
     \begin{tikzcd}
        \arrow[rrdd, bend right, hook, "f_*"]&&\\
        &&\\
        && \quad
     \end{tikzcd}
 \end{subfigure}%
 ~\begin{flushright} 
\begin{subfigure}[t]{.5\textwidth}
	\adjustbox{scale=.8,right}{\begin{tikzcd}[row sep=scriptsize, column sep=scriptsize]
		& 
		\color{red}{\pi_1(\disk{3} - T_{\alpha})}
		\arrow[rr] \arrow[dd] & &
		{
			\pi_{1}(\sphere{3} - \mathcal{U}_{\alpha \cup \overline{\beta}}) 
		}
		\arrow[dd] \\
		\color{black}{\pi_1(\sphere{2} - \{2b \text{ pts.}\})}
		\arrow[ur, two heads, black] \arrow[rr, two heads, crossing over, black] \arrow[dd, two heads, black]
		& &
		\color{blue}{\pi_1(\disk{3} - T_{\beta})}
		\ar[ur]
		\\
		& 
		{
			\pi_{1}(\sphere{3}-\mathcal{U}_{\alpha \cup \overline{\gamma}}) 
		}
		\arrow[rr] & &
		\pi_{1}(\sphere{4} - F) \\
		\color{myGreen}{\pi_1(\disk{3} - T_{\gamma})}
		\arrow[rr] \arrow[ur]
		& &
		{
			\pi_{1}(\sphere{3}-\mathcal{U}_{\beta \cup \overline{\gamma}}) 
		}
		\arrow[from=uu, crossing over] \arrow[ur] \\
	\end{tikzcd}}
 \end{subfigure}%
  \end{flushright}
	\caption{
            Applying the Seifert-van Kampen theorem to obtain three compatible group trisections for a branched cover $f \colon X\rightarrow \sphere{4}$ over a bridge trisected surface.
            In the bottom row is the group trisection of the complement of the branching set.
            Its preimage under the covering map is in the middle.
            The map $f$ induces maps on corresponding vertices to give a commutative hypercube.
            Similarly, the inclusion $i$ of the unbranched cover into the branched cover $X$ induces maps from vertices of middle cube to corresponding vertices of the group trisection of $X$ (top).
            \label{fig:Seifert_van_Kampen_cube}
        }
\end{figure}

\subsection{Tri-plane diagrams} 

Our algorithm for computing invariants of the branched cover of $\sphere{4}$ along a surface $F$ uses a representation of $F$ as a tri-plane diagram.

\begin{definition}[Tri-plane diagram \cite{meier2017bridgeS4}]
    A $(b;c_1,c_2,c_3)$-{\it tri-plane diagram} is a triple of $b$-bridge trivial tangle diagrams $(A, B, C)$, such that the pairwise unions $A\cup\overline{B}$, $B\cup\overline{C}$, and $C\cup \overline{A}$ are $c_1$, $c_2$, and $c_3$-component unlinks, respectively.
\end{definition}

More generally, one can consider {\it singular tri-plane diagrams} \cite{cahn2017singular}, in which the pairwise unions of tangles are allowed to be non-trivial links.  Singular tri-plane diagrams represent surfaces with cone singularities, and appear only in \Cref{xi.sec}.

Meier and Zupan \cite{meier2017bridgeS4} show that any bridge-trisected surface $F$ in $\sphere{4}$ can be represented by a tri-plane diagram, and give a set of moves sufficient for relating any two such diagrams for $F$.  Their proof gives an explicit algorithm for converting a representation of $F$ as a \emph{banded unlink diagram} to a tri-plane diagram. We will not describe the process explicitly, but for tri-plane diagrams obtained by this method (e.g., \Cref{fig:stevedore_triplane,fig:Suciu_R_3_bridge_trisection}), we mark the bands in the tri-plane diagram, and refer the reader to \cite{meier2017bridgeS4} for details.
Additional methods for obtaining tri-plane representations of knotted ribbon surfaces in $\sphere{4}$ are given in \cite{joseph2021bridgetrisectionclassical}.

\subsection{Braid notation for trivial tangles} 

To aid the reader in using the implementation of our algorithms, we discuss the notation used in the input for our program \cite{Ruppik_Knot_groups_2021}.

Each tangle diagram is drawn in bridge position and is stored as a list of \emph{braid crossings} together with a list of \emph{strand matchings}.
For an illustration of the process, see \Cref{fig:spun_trefoil_triplane_diagram_3_coloring}.
To produce each list, first label arcs of the tangle diagrams from $0$ to $2b-1$ from left to right, where $b$ is the bridge number.
Continue this labelling from bottom to top along level sets, omitting pairs of labels as their corresponding strands are joined at a local maximum.
The list of braid crossings is obtained from reading from bottom to top; the notation $(i,j,\epsilon)$ denotes a crossing of strands $i$ and $j$ of sign $\epsilon$.
Whenever two strands $i$ and $j$ join at a local maximum, we record $(i,j)$ in the list of strand matchings.

The braid crossing lists for the red and blue tangles in the tri-plane diagram for the spun trefoil in \Cref{fig:spun_trefoil_triplane_diagram_3_coloring} are $\text{[(1,2,+),(1,2,+),(1,2,+)]}$,
and the list for the green tangle is
$\text{[(1,4,+),(1,4,+),(1,4,+)]}$.
The strand matching lists are:
$$\text{Red tangle: [(0,1),(2,7),(3,6),(4,5)] }$$
$$\text{Blue tangle: [(0,1),(2,5),(3,4),(6,7)] }$$
$$\text{Green tangle: [(0,1),(2,3),(4,7),(5,6)] }$$

The only other input needed for the program is the list of permutations along the endpoints of the tangles, in this case $$[ (1,2), (1,2),(1,3),(1,3),(1,3),(1,3),(1,3),(1,3)].$$
\begin{figure}
    \centering
    \includegraphics[width=\textwidth]{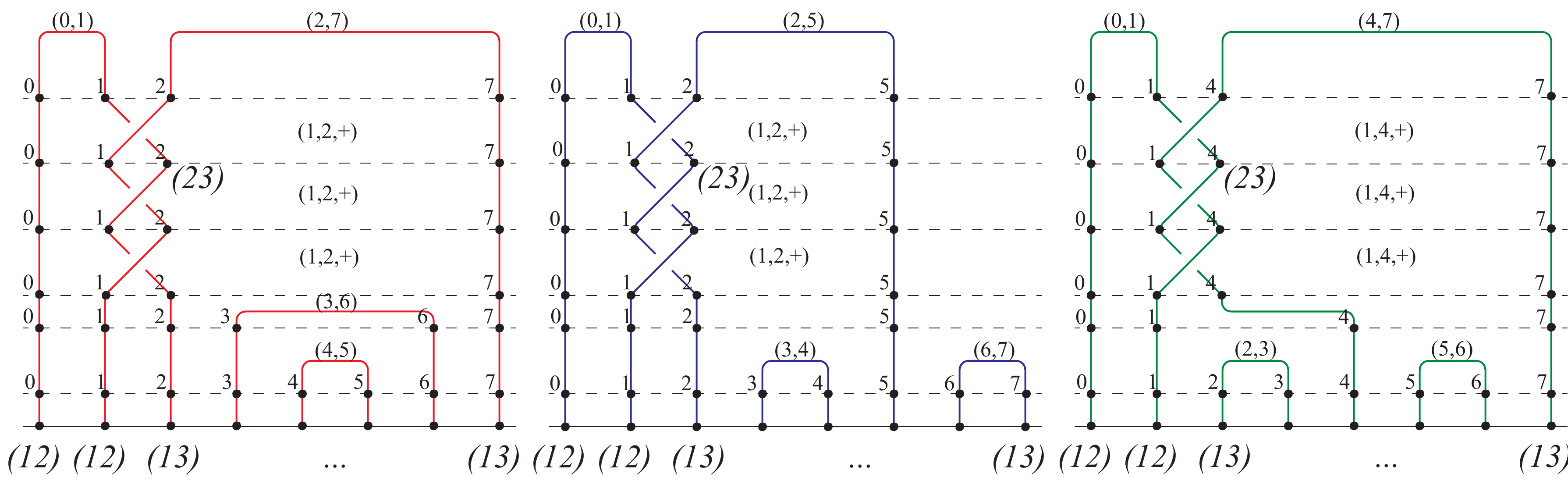}
    \caption{
        Example of a 3-colored tri-plane diagram for a bridge position of the spun trefoil, together with braid crossing and matching data for each tangle.
        \label{fig:spun_trefoil_triplane_diagram_3_coloring}
    }
	
\end{figure}

\section{Group trisections of branched covers}
\label{sec:group_trisection}

\subsection{Group trisections} 

By applying the Seifert-van Kampen Theorem to the trisection decomposition of a 4-manifold $X$, one obtains a commutative cube such that every face is a push-out. This is known as the \emph{group trisection} of $X$, pictured in \Cref{fig:Seifert_van_Kampen_cube} (top).
Group trisections were introduced by Abrams, Gay, and Kirby \cite{abrams2018group}, who showed they determine the manifold $X$ up to diffeomorphism.

The analogous commutative cube for the bridge trisection of the surface $F$ is known as a {\it group trisection of knotted surface type} \cite{blackwell2023group,ruppik2022thesis}, and pictured in \Cref{fig:Seifert_van_Kampen_cube} (bottom).

Note that because every face of the cube is a pushout, the entire group trisection (of either type) is determined by the three inclusion-induced maps out of the upper-left vertices. We refer to this sub-diagram as the {\it tripod.}

\subsection{Computing \texorpdfstring{$\pi_1$}{pi 1} from a tri-plane diagram}
\label{Wirtinger.sec}

Several methods for obtaining a presentation of $\pi_1(\sphere{4}-F)$ from a tri-plane diagram for $F$ are given in \cite{joseph2021bridgetrisectionclassical}; we will use a method similar to \cite[Theorem 4.1 (a)]{joseph2021bridgetrisectionclassical}, but with a different orientation convention at the meridians of the bridge points.  Our presentation is obtained as follows:

\begin{enumerate}
    \item Let $S$ be the image of the arcs (also called shadows) of the tangles $(T_\alpha,T_\beta,T_\gamma)$ after a choice of isotopy into the bridge sphere $\sphere{2}$. Let $B \subset \sphere{2}$ be a closed ball containing $S$. Let $P\in \sphere{2}-S$ be a basepoint on $\sphere{2}-\{p_0,p_1,\dots, p_{2b-1}\}$. Intuitively, we think of $P$ as lying on the bridge sphere ``far in front'' of each tangle.
    \item Let $x_i\in \pi_1(\sphere{2}-\{p_0,p_1,\dots, p_{2b-1}\},P)$ be a meridian of $p_i$, oriented clockwise when viewed from above, as shown in \Cref{puctured_sphere_gens.fig}.  
    This orientation convention is chosen so that a presentation for the punctured surface is
    \[
        \pi_1(\sphere{2}-\{p_0,p_1,\dots, p_{2b-1}\},P)=\langle x_0,x_1,x_2,\dots x_{2b-1}|x_0x_1x_2\dots x_{2b-1}\rangle.
    \]
    In a tri-plane diagram, we use the convention that the $x_i$ are labelled $x_0,\dots,x_{2b-1}$ from left to right. 
    \item Let $D_\mu$ be a tangle diagram for $T_\mu$, $\mu\in\{\alpha, \beta,\gamma\}$. Push $x_i$ into $\disk{4}-T_\mu$.  Note that the choice of $P$ in step (1) allows us to assume that the loop $x_i$ in the diagram $x_i\cup D_\mu$ crosses over arcs of $D_\mu$ with one exception, when passing under the $i^{\text{th}}$ bridge arc.
    \item Starting on the bridge sphere and working upwards along each tangle, use the Wirtinger relation at each crossing to associate a word in the $x_i$ to the outgoing under-strand.
    \item At each maximum where arcs $i$ and $j$ meet, add the relation $r_ir_j=1$, where $r_i$ and $r_j$ are the words at the top of the arcs resulting from successive applications of the Wirtinger relations. Let $w_1^\mu=1,w_2^\mu=1,\dots, w_b^\mu=1$ be the $b$ relations obtained in this manner from $D_\mu$.
\end{enumerate}

The tripod whose pushout produces the corresponding group trisection is therefore 

\begin{tikzcd}
    & \langle x_0,x_1,\dots, x_{2b-1} \mid w_1^\alpha,\dots w_b^\alpha \rangle \\
    \langle x_0,x_1,\dots, x_{2b-1} \mid x_0x_1\dots x_{2b-1}\rangle \arrow[ru, two heads, "\iota_\alpha"] \arrow[r, two heads, "\iota_\beta"] \arrow[rd, two heads, "\iota_\gamma"] &
    \langle x_0,x_1,\dots, x_{2b-1} \mid w_1^\beta,\dots w_b^\beta \rangle  \\
    &\langle x_0,x_1,\dots, x_{2b-1} \mid w_1^\gamma,\dots w_b^\gamma \rangle, \\
\end{tikzcd}
  
using $x_i$ to denote both an element of $\pi_1(\sphere{2}-\{2b\text{ pts.}\})$ and its images under the three inclusion maps $\iota_\mu$, for $\mu\in\{\alpha,\beta,\gamma \}$. 
  
We note that the surface relation $x_0x_1\dots x_{2b-1}=1$ can be derived from each of the three sets of relations $\{w_1^\mu,\dots, w_b^\mu\}$, $\mu\in\{\alpha,\beta,\gamma \}$.

\subsection{Monodromy homomorphisms and colorings}

An n-fold branched cover of a 4-manifold $X$ branched over an embedded surface $F$ is determined by a group homomorphism $\rho \colon \pi_{1}(X- F) \to S_n$ called its {\it monodromy homomorphism}, or informally, an \emph{$n$-coloring}.

The branched cover determined by $\rho$ is connected if the image of $\pi_{1}(X -F)$ in $S_n$ acts transitively on $\{1, \dots, n\}$. 
We will call such a homomorphism \emph{transitive}.

When $S_n$ is the dihedral group of order $n=2p$, with $\rho$ mapping meridians of $F$ to reflections, $\rho$ can be viewed as a classical Fox $p$-coloring of $F$.

\begin{definition}[Colored bridge trisection]
    \label{coloredgptrisection.def}
    A \emph{n-colored bridge trisection} is a commutative cube of transitive group epimorphisms as shown below, such that all the faces are pushouts, together with transitive group homomorphisms from each vertex group to the symmetric group $S_{n}$ making the diagram commute.

    \begin{tikzcd}
        & \pi_{1}(\disk{3} - T_{\alpha}) \arrow[rddd, dashed, bend left] \arrow[rr]
        &                         
        & \pi_1(\sphere{3} - \mathcal{U}_{\alpha \cup \overline{\beta}}) \arrow[rd] \arrow[lddd, dotted, bend right] & \\
        \pi_{1}(\sphere{2} - \{2b \textrm{ pts.}\}) \arrow[ru, two heads] \arrow[r, two heads] \arrow[rd, two heads] \arrow[rrdd, bend right] 
        & \pi_{1}(\disk{3} - T_{\beta}) \arrow[rdd, dashed, bend left] \arrow[rru, crossing over]
        &                         
        & \pi_1(\sphere{3} - \mathcal{U}_{\alpha \cup \overline{\gamma}}) \arrow[r] \arrow[ldd, dotted, bend right] \arrow[from=ull, crossing over]
        & \pi_1(\sphere{4} - F) \arrow[lldd, dotted, bend left] \\
        & \pi_{1}(\disk{3} - T_{\gamma}) \arrow[rd, dashed, bend left] \arrow[rru, crossing over]
        &                         
        & \pi_1(\sphere{3} - \mathcal{U}_{\beta \cup \overline{\gamma}}) \arrow[ru] \arrow[ld, dotted, bend right] \arrow[from=ull, crossing over] \arrow[from=ll, crossing over]
        & \\
        & & \operatorname{S}_{n} & &                                                      
    \end{tikzcd}
	
\end{definition}

\begin{proposition}
    \label{extend_coloring.prop}
    A transitive homomorphism $\rho \colon \pi_{1}(\sphere{2} - \{2b \textrm{ pts.}\}) \to \operatorname{S}_{n}$ determines a colored bridge trisection if and only if it extends over the tripod, i.e., for each $\mu\in\{\alpha, \beta,\gamma\}$, there exists a map $\rho_\mu$
    \begin{center}
        \begin{tikzcd}
            \pi_{1}(\sphere{2} - \{2b \textrm{ pts.}\} )\arrow[r,two heads,"\iota_\mu"]\arrow[rd,"\rho"]&\pi_1(\disk{3}-T_\mu)\arrow[d,dashed, "\rho_\mu"] \\
            & S_n
        \end{tikzcd}
    \end{center}
    making the diagram commute.
    Moreover, such an extension, when it exists, is unique.
    
\end{proposition}

\begin{proof}
    The extensions $\rho_\mu$, if they exist, are unique, by the surjectivity of $\iota_\mu$. Next, given extensions $\rho_\mu$ of $\rho$ over the tripod, from the universal property of the push-out, we can recover all the other groups by building the corresponding amalgamated products.
    The maps to $S_n$ exist and are unique by the universal property of the pushout.
    Conversely, a colored bridge trisection clearly restricts to a unique map on the tripod.
\end{proof}

The diagrammatic version of a colored bridge trisection is a colored tri-plane diagram.

\begin{definition}[Colored tri-plane diagram, c.f. \cite{cahn2017singular}] \label{colored_tri_plane.def}
    A \emph{colored tri-plane diagram} is a labelling of the arcs of each tangle diagram with permutations in $S_n$ such that at each crossing, the labels satisfy the Wirtinger relations, and such that the labels agree along the endpoints of the tangles.
\end{definition}

We note that a colored tri-plane diagram determines a colored group trisection, and conversely, a colored group trisection together with a choice of tri-plane diagram (and hence a choice of Wirtinger presentation), determines a colored tri-plane diagram.

\subsection{Branched covers of bridge trisections}
A colored bridge trisection determines two related commutative cubes of covering spaces, and two related group trisections --
one for the unbranched cover $X'$ of the surface complement $\sphere{4} - F$ determined by $\rho$, and one for the branched cover ${X}$ of $\sphere{4}$ branched along $F$ determined by the same coloring.
See \Cref{fig:Seifert_van_Kampen_cube}. To see that the branched cover of a bridge-trisected surface carries the structure of a trisection, one needs only to check that the branched cover of a ball along a trivial disk tangle in dimension $n$ is an $n$-dimensional handlebody (specifically when $n=3$ and $n=4$) \cite{cahn2017singular, lambert2021symplectic}.  The unbranched cover is a manifold with boundary, and carries a trisection-like structure, but it does not correspond to a  relative trisection diagram of a manifold with boundary in the standard sense of \cite{castro2018trisections}. However, after applying the $\pi_1$ functor to the commutative cube, one gets a group trisection of knotted surface type as in \cite{ruppik2022thesis,blackwell2023group}.

Consider the inclusion $i:X'\hookrightarrow X$ of the unbranched cover into the branched cover, and its restrictions to the components of the trisections for $X'$ and $X$.
Note that one can recover the inclusion-induced maps between all corresponding groups in the group trisections for $X'$ and $X$ from the inclusion-induced maps on their tripods, namely
\[
    i_* \colon \pi_1(\Sigma_g') \twoheadrightarrow \pi_1(\Sigma_g),
\]
where $\Sigma_g'=\Sigma_g - f^{-1} \{2b \text{ pts.}\}$ is the unbranched cover of the punctured bridge sphere $\sphere{2}-\{2b \text{ pts.}\}$,
and
\[
    i_* \colon \pi_1(H_\mu - \widetilde{T}_\mu) \twoheadrightarrow \pi_1(H_\mu),
\]
for $\mu \in \{\alpha, \beta, \gamma\}$, by repeatedly applying the universal mapping property of the push-out.
This results in a commutative hypercube between the group trisections for $X'$ and $X$.
See \Cref{fig:Seifert_van_Kampen_cube}.

Because the branching set, and hence its preimage in $X$, is codimension two, the groups in the group trisection of $\pi_1(X)$ are obtained from the corresponding groups in the group trisection of $\pi_1(X')$ by adding relations that ``kill'' the meridians of the components of the preimage of the branching set.  
 
In the next section, we will describe an algorithm for computing the groups and maps in these group trisection diagrams.

\section{A Trisected Reidemeister-Schreier Algorithm}
\label{sec:trisected_RS}

Fox in \cite{fox1962quicktrip} describes a version of the \emph{Reidemeister-Schreier algorithm} which, given a coloring of a knot $K$ in $\sphere{3}$, derives a presentation of the fundamental group of the associated unbranched cover.
After lifting powers of meridians of the branching set to obtain meridians of its preimage, one obtains a presentation of the fundamental group of the branched cover of $\sphere{3}$ along $K$ by setting these lifts equal to the identity.

We will show that Fox's algorithm can be applied to a colored bridge-trisected surface in  $\sphere{4}$ to compute a group trisection of the associated unbranched and branched covers $X'$ and $X$ respectively. 

We will first use the algorithm to find the presentations for the fundamental groups of unbranched  covers of  $\sphere{2} - \{ p_0,...,p_{2b-1} \} $ and the three tangle complements, $\mathbb{D}^3 - T_{\alpha}$, $\mathbb{D} ^3- T_{\beta}$ and  $\mathbb{D}^3 -T_{\gamma}$.
This yields the tripod for the group trisection of the unbranched cover $X'$, from which we recover the entire group trisection for $X'$ by taking push-outs.
We then add in the relations corresponding to filling in the preimage of the branching set to obtain a group trisection for the branched cover $X$.
Here is a guide through the algorithm, with details and instructive examples worked out below.

\begin{algorithm}[Trisected Reidemeister-Schreier]
    \label{RS.alg}
    Let $\rho \colon \pi_{1}(\sphere{2} - \{2b \textrm{ pts.}\}) \to S_n$ be a group homomorphism that factors to a colored group trisection corresponding to a colored bridge trisected surface $F \subset \sphere{4}$ as in \Cref{coloredgptrisection.def}.
	
    Then a group trisection of the trisected branched cover $X$ (a closed 4-manifold) of $\sphere{4}$ along $F$ can be computed as follows:

    \begin{enumerate}
        \item Choose generators $x_0$, $x_1$, ..., $x_{2b-1}$ for $\pi_{1}(\sphere{2} - \{2b \textrm{ pts.}\},P)$ as shown in \Cref{puctured_sphere_gens.fig}, and consider a collection of path-lifts $\pathelement{x}_0^j$, $\pathelement{x}_1^j$, ..., $\pathelement{x}_{2b-1}^j$ of the above generators starting at the lifts $P_j$ of base point $P$, for each $j\in\{1,\dots,n\}$.
        \item Use Fox's ``claw'' attachment to associate to each $\pathelement{x}_i^j$ a closed loop $x_i^j$ based at $P_1$, for $i\in \{0,\dots,2b-1\}$ and $j \in \{1, \dots, n \}$, to obtain a total of $2bn$ generators for the unbranched cover.  
        \item Apply the relative Reidemeister-Schreier algorithm, detailed below,
        to a presentation for the group $\pi_1(\sphere{2}-\{p_0,\dots,p_{2b-1}\})$ , and the groups $\pi_1(\disk{3}-T_\mu)$ for each of the colored trivial tangles.
        Here ``relative'' means that in addition to computing a presentation of the group of the branched cover, we also remember how the branched cover of the punctured surface includes into the branched cover of the trivial tangle complement.
        This results in the tripod for the group trisection of the \emph{unbranched} cover, from which we can recover the full group trisection of the unbranched cover by taking push-outs.
        \item To obtain a group trisection for the branched cover, fill in the preimage of the branching locus by adding one relation $\loopelement{x}_i^{n_1}\loopelement{x}_i^{n_2}\dots\loopelement{x}_i^{n_k}=1$ for each $k$-cycle $(n_1,\dots,n_k)$ in the permutation $\rho(x_i)$, for all $i\in\{0,\dots,2b-1\}$.
    \end{enumerate}
\end{algorithm}

\vskip.02in
We will now give more details of the construction and the algorithm, and follow the steps in the algorithm to work out a couple of examples.

(1) We use the Wirtinger algorithm on the given tri-plane diagram for $F$ to obtain the tripod for its group trisection
\begin{equation}
    \label{downstairs_tripod.eq}
    \begin{tikzcd}
        & G_\alpha =
        \langle x_0,x_1,\dots,x_{2b-1} \mid w_1^\alpha, \dots, w_b^\alpha \rangle \\
        G =
        \langle x_0,x_1,\dots,x_{2b-1} \mid x_0x_1\dots x_{2b-1} \rangle 
        \arrow[ru, two heads, "\iota_\alpha"] \arrow[r, two heads, "\iota_\beta"] \arrow[rd, two heads, "\iota_\gamma"]
        & G_\beta = 
        \langle w = x_0, x_1, \dots, x_{2b-1} \mid w_1^\beta, \dots, w_b^\beta \rangle  \\
        & G_\gamma = 
        \langle x_0, x_1, \dots, x_{2b-1} \mid w_1^\gamma, \dots, w_b^\gamma \rangle, \\
    \end{tikzcd}
\end{equation}
as described in \Cref{Wirtinger.sec}, from which one can recover the full group trisection for the surface complement $\sphere{4}-F$ by taking push-outs.  Recall that the generators $x_i$ are elements of $G=\pi_1(\sphere{2}-\{p_0,\dots,p_{2b-1}\},P)$, with $x_i$ oriented clockwise around $p_i$ as shown in \Cref{puctured_sphere_gens.fig}, and that we also use $x_i$ to denote its image in each of the three groups $G_\gamma=\pi_1(\disk{3}-T_\mu,P)$ for $\mu\in\{\alpha, \beta,\gamma\}$.

\begin{figure}
    \resizebox{0.5\textwidth}{!}{
        \begin{tikzpicture}
            \node[anchor=south west,inner sep=0] at (0,0){\includegraphics[width=10cm]{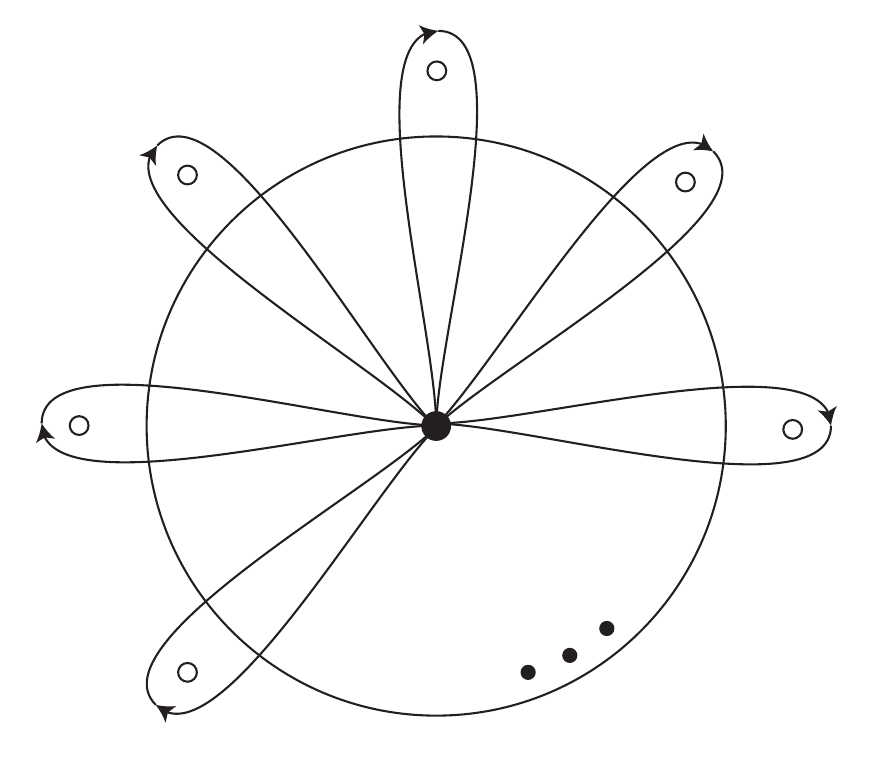}};
            \node at (5.1, 3.5) {\textcolor{black}{$P$}};
            \node at (1.5, 0.2) {\textcolor{violet}{$x_{2b-3}$}};
            \node at (-0.3, 4) {\textcolor{violet}{$x_{2b-2}$}};
            \node at (1.4, 7.5) {\textcolor{violet}{$x_{2b-1}$}};
            \node at (5, 9) {\textcolor{violet}{$x_{0}$}};
            \node at (8.5, 7.5) {\textcolor{violet}{$x_{1}$}};  
            \node at (10, 4) {\textcolor{violet}{$x_{2}$}};
        \end{tikzpicture}
    }
    \caption{
        Generators $x_0, x_1, \ldots, x_{2b-1}$ of $\pi_1(\sphere{2}- \{p_0, \ldots, p_{2b-1}\},P )$.
        \label{puctured_sphere_gens.fig}
    }
\end{figure}

The Reidemeister-Schreier algorithm for writing down presentations for the fundamental groups of covering spaces, as roughly described by Fox in \cite{fox1962quicktrip},  applied to the four groups in the tripod above, works as follows.

A connected $n$-fold unbranched cover for the complement of the knotted surface is determined by a transitive homomorphism  $\rho \colon G \to S_n$ of  $G=\pi_1(\sphere{2} - \{p_0,...,p_{2b-1}\},P )$  to the symmetric group $S_n$ that extends to the fundamental groups  of the three tangle complements  $G_{\alpha}=\pi_1(\mathbb{D} ^3 - T_{\alpha},P)$, $G_{\beta}=\pi_1(\mathbb{D} ^3 - T_{\beta},P)$ and $G_{\gamma}=\pi_1(\mathbb{D} ^3 - T_{\gamma},P)$ to $S_{n}$ as in \Cref{coloredgptrisection.def} and \Cref{extend_coloring.prop}.

The transitive homomorphism $\rho : \pi_1(\sphere{2} - \{p_0,...,p_{2b-1}\} ) \to S_n$ determines a connected  n-fold cover of the punctured sphere, which is a punctured surface $\Sigma_g'$ of genus $g$. 
We label the $n$ points in the fiber above the base point $P$  by $P_1,\dots,P_n $, and consider the lifts  of the generators $x_i$ of $\pi_1(\sphere{2}- \{p_0,...,p_{2b-1}\}$ starting at the different $P_j$.
We denote by $\pathelement{x}_i^j$, the path-lift of $x_i$ to $\Sigma_g'$ that starts at $P_j$.
Note that this $\pathelement{x}_i^j$ ends at $P_{\rho(x_i)(j)}$.

 \definecolor{green1}{cmyk}{.60,0,1.00,0}
    \definecolor{blue1}{cmyk}{.70,.20,0,0}
    \definecolor{purple1}{cmyk}{.55,.70,0,0}
    \definecolor{green2}{cmyk}{1.00,.30,1.00,.20}
    \definecolor{burgundy}{cmyk}{0,1,1,.5}
    \definecolor{blue2}{cmyk}{1,1,0,0}
    \definecolor{green2}{cmyk}{.5,0,.5,.5}
    \begin{figure}
        \centering
         \begin{tikzpicture}
            \node[anchor=south west,inner sep=0] at (6,0){ \includegraphics[width=6.8cm]{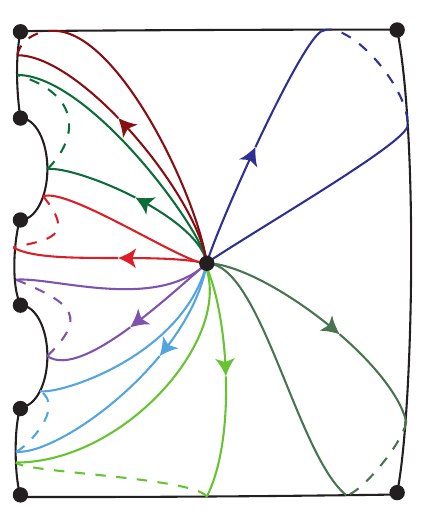}};
            \node at (10, 4.25) {\textcolor{black}{$P$}};
            \node at (12.6, 0) {\textcolor{violet}{$p_1$}};
            \node at (12.6, 8.2) {\textcolor{violet}{$p_0$}};
            \node at (6, 0) {\textcolor{violet}{$p_2$}};
            \node at (6, 2) {\textcolor{violet}{$p_3$}};
            \node at (6, 3.6) {\textcolor{violet}{$p_4$}};
            \node at (6, 5) {\textcolor{violet}{$p_5$}};
            \node at (6, 6.5) {\textcolor{violet}{$p_6$}};
            \node at (6, 8.2) {\textcolor{violet}{$p_7$}};
            \node at (7, 0.75) {\textcolor{green1}{${x}_2$}};
            \node at (7, 1.8) {\textcolor{blue1}{${x}_3$}};
            \node at (7, 3.6) {\textcolor{purple1}{${x}_4$}};
            \node at (7, 5.3) {\textcolor{red}{${x}_5$}};
            \node at (7, 6.5) {\textcolor{green2}{${x}_6$}};
            \node at (7, 7.4) {\textcolor{burgundy}{${x}_7$}};
            \node at (11.2, 7.4) {\textcolor{blue2}{${x}_0$}};
            \node at (11.2, 1.5) {\textcolor{green2}{${x}_1$}};
        \end{tikzpicture}
        \caption{Generators $x_0,\dots,x_7$ of $\pi_1(\sphere{2}-\{p_0,\dots,p_7\})$. }
        \label{fig:spun_trefoil_downstairs}
    \end{figure}
    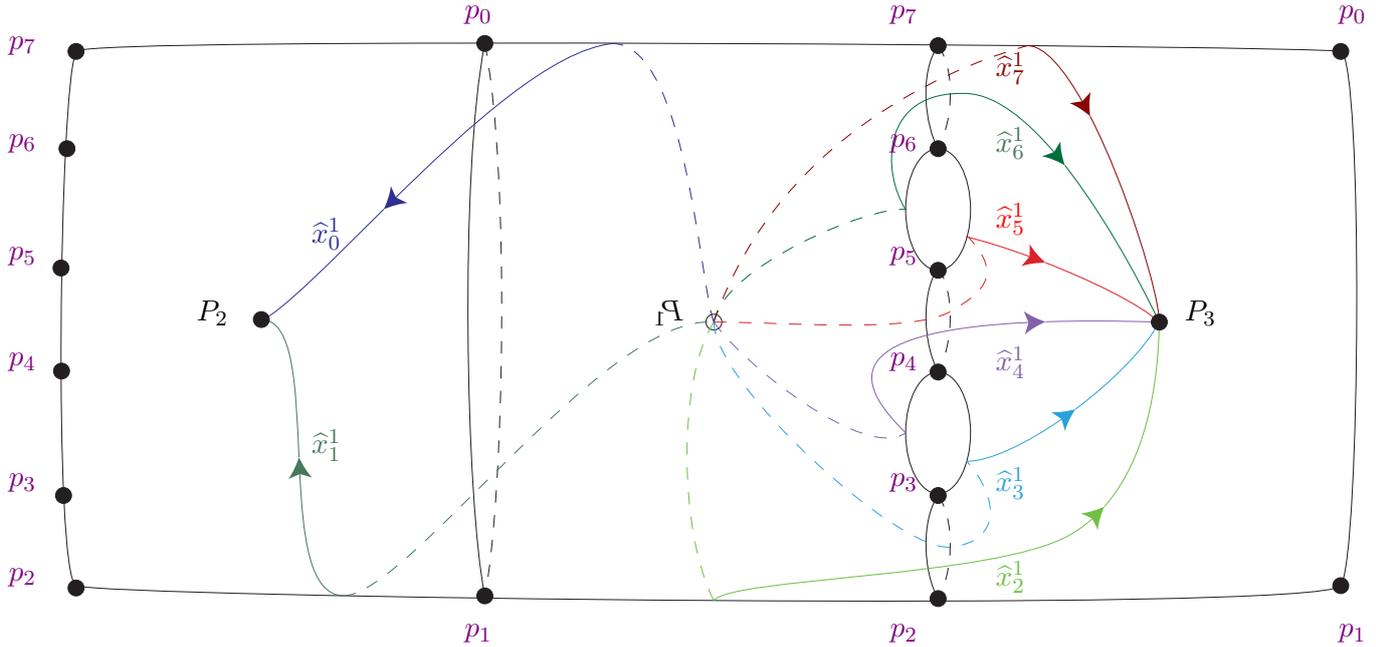
\begin{figure}
        \centering
         \begin{tikzpicture}
            \node[anchor=south west,inner sep=0] at (0,0){ \includegraphics[width=18cm]{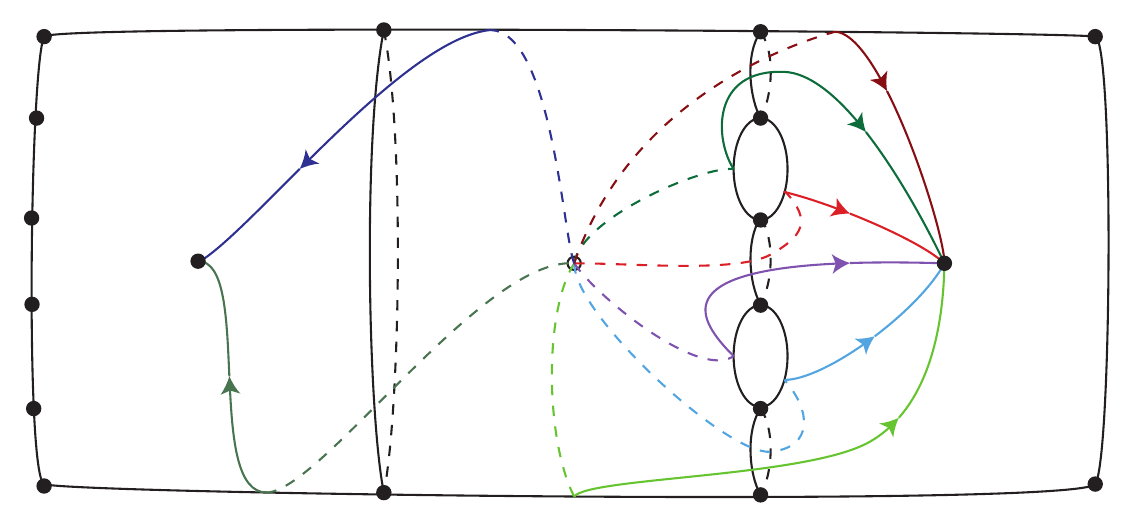}};
            \node at (2.5, 4.25) {\textcolor{black}{$P_{2}$}};
            \node at (8.5, 4.25) {\textcolor{black}{\reflectbox{$P_{1}$}}};
            \node at (15.5, 4.25) {\textcolor{black}{$P_{3}$}};
            \node at (0, 0.75) {\textcolor{violet}{$p_2$}};
            \node at (0, 2) {\textcolor{violet}{$p_3$}};
            \node at (0, 3.6) {\textcolor{violet}{$p_4$}};
            \node at (0, 5) {\textcolor{violet}{$p_5$}};
            \node at (0, 6.5) {\textcolor{violet}{$p_6$}};
            \node at (0, 7.8) {\textcolor{violet}{$p_7$}};
            \node at (6, 0) {\textcolor{violet}{$p_1$}};
            \node at (6, 8.2) {\textcolor{violet}{$p_0$}};
            \node at (11.6, 0) {\textcolor{violet}{$p_2$}};
            \node at (11.6, 2) {\textcolor{violet}{$p_3$}};
            \node at (11.6, 3.6) {\textcolor{violet}{$p_4$}};
            \node at (11.6, 5) {\textcolor{violet}{$p_5$}};
            \node at (11.6, 6.5) {\textcolor{violet}{$p_6$}};
            \node at (11.6, 8.2) {\textcolor{violet}{$p_7$}};
            \node at (17.5, 0) {\textcolor{violet}{$p_1$}};
            \node at (17.5, 8.2) {\textcolor{violet}{$p_0$}};
            \node at (13, 0.75) {\textcolor{green1}{$\pathelement{x}_2^1$}};
            \node at (13, 2) {\textcolor{blue1}{$\pathelement{x}_3^1$}};
            \node at (13, 3.6) {\textcolor{purple1}{$\pathelement{x}_4^1$}};
            \node at (13, 5.5) {\textcolor{red}{$\pathelement{x}_5^1$}};
            \node at (13, 6.5) {\textcolor{green2}{$\pathelement{x}_6^1$}};
            \node at (13, 7.5) {\textcolor{burgundy}{$\pathelement{x}_7^1$}};
            \node at (4, 2.5) {\textcolor{green2}{$\pathelement{x}_1^1$}};
            \node at (4, 5.3) {\textcolor{blue2}{$\pathelement{x}_0^1$}};
        \end{tikzpicture}
        \caption{
            Lifts $\pathelement{x}_i^1$ of generators $x_i$ of $\pi_1(\sphere{2}-\{p_0,\dots,p_7\})$ starting at $P_1$ on the 3-fold cover of $\sphere{2}-\{p_0,\dots,p_7$\}, corresponding to the coloring $\rho(x_0)=\rho(x_1)=(12)$, and $\rho(x_2),\dots,\rho(x_7)=(13)$. 
            \label{spun_trefoil_lifts.fig}
        }
    \end{figure}

(2) We next explain how to use Fox's ``claw attachment'' to associate a closed loop $x_i^j$ based at $P_1$ to each path $\pathelement{x}_i^j$, to obtain a set of generators for the fundamental group of the unbranched cover $X'$ of $\sphere{4}-F$.

We attach a tree with one central vertex $V$ and n edges $e_1=[V,Q_1], \dots, e_n=[V,Q_n] $, the ``claw'', to our trisected unbranched cover $X'$ by identifying the points  $Q_i$ with  $P_i$ for each $i\in\{1,\dots,n\}$. This produces a space whose fundamental group can be obtained from the fundamental group of $X'$ by adding $n-1$ extra free generators.  Since the claw is attached along the core surface $\Sigma_g'$ of $X'$, the same is true for each group in the group trisection for $X'$, shown in the center of \Cref{fig:Seifert_van_Kampen_cube}. 

We will then ``collapse'' this claw to get a space homotopy equivalent to the original $X'$, by choosing paths  $\gamma_j$ in $\Sigma_g'$ from $P_1$ to each  $P_ j$ for each $j\in\{1,\dots,n\}$ such that (1) $\gamma_j$ is a lift of a word $\tau_j$ in the generators $x_i$ and (2) the union of the paths $\gamma_j$ is a tree $T$ with vertices $P_i$ and with edges a subset of the $\pathelement{x}_i^j$. Note that the paths $\pathelement{x}_i^j$ are disjoint except at their endpoints. Then we add a two cell that fills the triangle $\langle P_i, V, P_j \rangle$ with edges $e_i^{-1}, e_j, e^{-1}$ for each edge $e$ of the tree $T$. See \Cref{claw.fig}.  

\begin{remark}
    The fact that the union of all the paths $\gamma_j$ is a tree clearly implies that for any word $\tau_j$ all of its initial sub-words are words $\tau_k$ for the points $P_k$ that lie on the lift of $\tau_j$.
    Note that this corresponds to the ``Schreier condition'' in the algebraic formulation of the Reidemeister-Schreier algorithm, and the extra relations that are induced by adding these two cells correspond to adding in the Schreier relations in the standard Reidemeister-Schreier construction of the presentation of a subgroup of a finitely presented group (here the fundamental group of the cover).
\end{remark}

\begin{figure}
        \begin{tikzpicture}
            \node[anchor=south west,inner sep=0] at (0,0){\includegraphics[width=10cm]{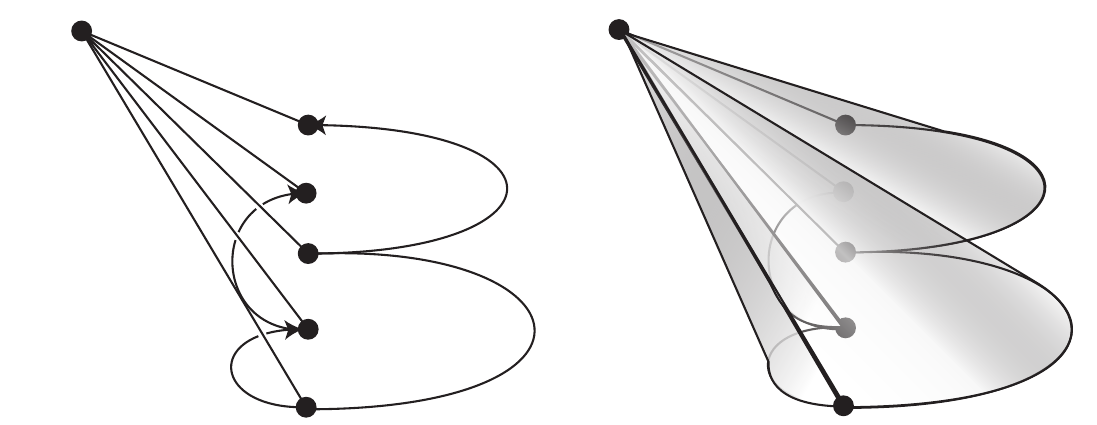}};
            \node at (3.2, 0) {\textcolor{black}{$P_1$}};
            \node at (3.2, 0.8) {\textcolor{black}{$P_2$}};
            \node at (2.5, 1.6) {\textcolor{black}{$P_3$}};
            \node at (3.2, 2.1) {\textcolor{black}{$P_4$}};
            \node at (3.2, 3) {\textcolor{black}{$P_5$}};
            \node at (1.8, .6) {\textcolor{black}{$\pathelement{x}_1^1$}};
            \node at (1.6, 1.2) {\textcolor{black}{$\pathelement{x}_1^1\pathelement{x}_1^2$}};
            \node at (5.1, .6) {\textcolor{black}{$\pathelement{x}_0^1$}};
            \node at (5.1, 2) {\textcolor{black}{$\pathelement{x}_0^1\pathelement{x}_3^3$}};
             \node at (0,4.4) {\textcolor{black}{$V$}};
        \end{tikzpicture}
    \caption{
       \label{claw.fig}
       Attaching 2-cells to collapse the ``claw'', producing the claw relations
       $\{c_2=x_1^1=1, c_3=x_0^1=1,c_4=x_1^1x_1^1x_1^2=1,  c_5=x_0^1x_0^1x_3^3=1\}.$ The corresponding 5-coloring is given in~\Cref{stevedore_RS_example.ex}.
    }
\end{figure}

We next determine a set of generators for the unbranched cover, with claw attached, based at the root $V$ of the claw.
The advantage of this perspective is that the $P_i$ are treated symmetrically.
The final step will be to change the basepoint to be one of the $P_i$, namely $P_1$.

Recall that the path $\pathelement{x}_i^j$ is the lift of $x_i$ that starts at $P_j$ and ends at $P_{\rho(i)(j)}$. To make the paths into loops based at V we start and end  with appropriate edges  of the ``claw'' to get the loop $\loopelement{x}_i^j = e_j*\pathelement{x}_{i}^j*{e}_{\rho(i)(j)}^{-1}$. If we use the attached triangles to homotope the paths $e_j$ and $e_k^{-1}$ across them to paths $e_1 * \gamma_j$ and $\gamma_k^{-1}*e_1^{-1}$, we get the loop  $\widetilde{x}_i^j = e_1*\gamma_j*\pathelement{x}_{i}^j* \gamma_{\rho(i)(j)}^{-1}*{e}_1^{-1}$ based at $V$. The corresponding closed loop based at $P_1$ is obtained by dropping the $e_1$ terms at the beginning and the end.  We will suppress this last drop in the notation, and define the generator $x_i^j$ based at $P_1$  as follows: $$\loopelement{x}_{i}^j=\gamma_j*\pathelement{x}_{i}^j* \gamma_{\rho(i)(j)}^{-1}\in\pi_1(\Sigma_g',P_1)\leq \pi_1(X',P_1).$$

The 2-cells attached produce relations in $\pi_1(\Sigma',P_1)\leq \pi_1(X,P_1)$ as follows:
For each path $\gamma_j$ from $P_1$ to $P_j$, $j\in\{2,\dots,n\}$, write
$\gamma_j=(\pathelement{x}_{i_1}^{j_1})^{\epsilon_1}\dots(\pathelement{x}_{i_k}^{j_k})^{\epsilon_k}$, $\epsilon_j\in\{\pm 1\}$.
Attaching the 2-cell that fills the triangle $\langle P_i, V,P_j\rangle$, and then collapsing the edges of the claw, corresponds to adding the relation $(\loopelement{x}_{i_1}^{j_1})^{\epsilon_1}\dots(\loopelement{x}_{i_k}^{j_k})^{\epsilon_k}=1$.
We refer to these $n-1$ relations as the \emph{claw relations} and denote the relation corresponding to $\gamma_j$ by $c_j=1$.
See \Cref{claw.fig}.

(3) To obtain presentations for the groups in the tripod of the group trisection of the unbranched cover $X'$, we must lift the relations in the corresponding groups in the tripod for the trisection of $\sphere{4}-F$, shown in \Cref{downstairs_tripod.eq}, and then add in the claw relations.
It suffices to explain the process of lifting a single relation according to the coloring $\rho$.

Each relation is given by a word in the generators and their inverses, i.e., $w=1$ where $w= x_{i_1}^{\epsilon_1}x_{i_2}^{\epsilon_2}\dots x_{i_k}^{\epsilon_k}, \epsilon_j=\pm 1$.
Each word $w$ will lift to $n$ words, $w_l$ for $l = 1, \dots, n$ with  

\[
    w_l = (\pathelement{x}_{i_1}^{j_1(l)})^{\epsilon_1}(\pathelement{x}_{i_2}^{j_2(l)})^{\epsilon_2}\dots (\pathelement{x}_{i_k}^{j_k(l)})^{\epsilon_k},
\]
and $j_1(l)=l$, where each $j_m(l)$ is determined inductively by making sure that the path $(\pathelement{x}_{i_2}^{j_2(l)})^{\epsilon_2}$
starts where $(\pathelement{x}_{i_1}^{j_1(l)})^{\epsilon_1}$ ends, $(\pathelement{x}_{i_3}^{j_3(l)})^{\epsilon_3}$
starts where $(\pathelement{x}_{i_2}^{j_2(l)})^{\epsilon_2}$ ends, etc. 
Here we repeatedly use the fact that $\pathelement{x}_i^j$ is the lift of $x_i$ that starts at $P_j$ and ends at $P_{\rho(i)(j)}$.
In general, if we lifted first $s$ letters in 
$w=x_{i_1}^{\epsilon_1}x_{i_2}^{\epsilon_2}\dots x_{i_k}^{\epsilon_k}$ starting at $l$
to get
$w_l(s) = (\pathelement{x}_{i_1}^{j_1(l)})^{\epsilon_1}(\pathelement{x}_{i_2}^{j_2(l)})^{\epsilon_2}\dots (\pathelement{x}_{i_s}^{j_s(l)})^{\epsilon_s}$
and the lift $w_l(s)$ ends at $P_a$, we continue inductively with 
$(\pathelement{x}_{i_{s+1}}^{j_{s+1}(l)})^{\epsilon_{s+1}}$ where now

$$i_{s+1}= a, \hskip.2in j_{s+1}=\rho(a)(s+1) \hskip.2in \text{if} \hskip.2in \epsilon_{s+1}=1$$   

$$i_{s+1}= \rho(a)^{-1}(j_{s+1}),  \hskip.2in j_{s+1}=a \hskip.2in \text{if} \hskip.2in\epsilon_{s+1}=-1$$ 

where the second case indicates that we go ``backwards'' when $\epsilon_j=-1$.  

After adding the claw relations, we can replace each path $\pathelement{x}_i^j$ with the loop $\loopelement{x}_i^j$ to obtain the relation 
\[
    w_l = (\loopelement{x}_{i_1}^{j_1(l)})^{\epsilon_1}(\loopelement{x}_{i_2}^{j_2(l)})^{\epsilon_2}\dots (\loopelement{x}_{i_k}^{j_k(l)})^{\epsilon_k}=1.
\]

Applying this process to each relation in \Cref{downstairs_tripod.eq}, and adding in the claw relations for each group, we get presentations for groups in the tripod for the unbranched cover $X'$.

\begin{equation}
    \label{upstairs_tripod_unbranched.eq}
    \begin{tikzcd}
        & \widetilde{G}_\alpha=\langle x_i^j \mid (w_1^\alpha)_l,\dots (w_b^\alpha)_l, c_k \rangle \\
        \widetilde{G} = \langle x_i^j \mid w_l, c_k \rangle \arrow[ru, two heads, "\iota_\alpha"]
        \arrow[r, two heads, "\iota_\beta"] \arrow[rd, two heads, "\iota_\gamma", swap] &
        \widetilde{G}_\beta = \langle x_i^j \mid (w_1^\beta)_l,\dots (w_b^\beta)_l, c_k \rangle \\
        & \widetilde{G}_\gamma = \langle x_i^j \mid (w_1^\gamma)_l,\dots (w_b^\gamma)_l, c_k \rangle \\
    \end{tikzcd}
\end{equation}
where indices $i,j$ for the generators range over all $i\in\{0,\dots, 2b-1\}$ and $j\in\{1,\dots n\}$; indices for the lifted relations range over all $l\in\{1,\dots,n\}$; and the index $k$ for the claw relations $c_k$ ranges over all $k\in\{2,\dots,n\}$.

(4) Finally, we produce presentations for groups in the tripod of the group trisection for the branched cover $X$ by adding to each of the presentations in \Cref{upstairs_tripod_unbranched.eq} relations of the form  $\mu=1$, which fill in each meridian $\mu$ of each point in the preimage of the point $p_i$, for all $i\in\{0,\dots, 2b-1\}$. 

We first write each such meridian $\mu$ as a concatenation of paths $\pathelement{x}_i^j$.  Let $(n_1,\dots,n_k)$ be a disjoint $k$-cycle in the permutation $\rho(x_i)$. Then $\mu=\pathelement{x}_i^{n_1}\pathelement{x}_i^{n_2}\dots\pathelement{x}_i^{n_k}$.

After adding the claw relations as in step (2), we can replace each $\hat{x}_i^j$ with $x_i^j$ to obtain the relation

\[
    \loopelement{x}_i^{n_1} \loopelement{x}_i^{n_2} \dots \loopelement{x}_i^{n_k}=1,
\]

which is just the lift of $(x_i)^k$ starting at $P_{n_1}$, according to the lifting process in step (3). 
We obtain one such relation for each $k$-cycle in the permutation $\rho(x_i)$, for all $i \in \{0, \dots, 2b-1\}$. 

\begin{example}
    \label{stevedore_RS_example.ex}
    We first present an example of the ordinary (non-trisected) Reidemeister-Schreier algorithm for a punctured surface group and a (somewhat complicated) coloring, to introduce the algorithm. This is the second map in our family of representations of the complement of the Stevedore disk double in \Cref{stevedore.sec}.  The full trisected version is presented, for a simpler coloring, in the next example.

    Consider the homomorphism
    \begin{align*}
        \rho \colon \pi_1(\sphere{2}-\{p_0,\dots,p_9\}, P)
        =
        \langle x_0, \dots, x_9 \mid x_0\dots x_9 \rangle
        & \rightarrow S_5 \\
        \rho(x_0) & = (1342) \\
        \rho(x_1) & = (1243) \\
        \rho(x_2)=\rho(x_4)=\rho(x_6)=\rho(x_8) & = (2453) \\
        \rho(x_3)=\rho(x_5)=\rho(x_7)=\rho(x_9) & = (2354).
    \end{align*}

    The path lifts
    $\{\gamma_2,\gamma_3, \gamma_4, \gamma_5\}=\{\pathelement{x}_1^1, \pathelement{x}_0^1, \pathelement{x}_1^1\pathelement{x}_1^2, \pathelement{x}_0^1\pathelement{x}_3^3\}$
    all begin at $P_1$ and end at $P_2$, $P_3$, $P_4$, and $P_5$ respectively, as shown in \Cref{claw.fig}.
    Therefore, the generators of the punctured surface $\pi_1(\Sigma_g',P_1)$, namely $x_i^j$ with $0\leq i \leq 9$ and $1\leq j\leq 5$, are expressed as products of the paths $\pathelement{x}_i^j$ as follows:
    
    \setlength{\tabcolsep}{10pt} 
    \renewcommand{\arraystretch}{1.5} 
    \begin{center}
        \begin{tabular}{c|c}
        \toprule
        \multicolumn{2}{c}{
            ${x}_0^1=\hat{x}_0^1(\hat{x}_0^1)^{-1}=1,
            {x}_1^1=\hat{x}_1^1(\hat{x}_1^1)^{-1}=1,
            {x}_i^1=\hat{x}_i^1 \text{ for } 2\leq i \leq 9$
        }\\
        \midrule
        $x_0^2=\hat{x}_1^1\hat{x}_0^2$&$x_0^3=\hat{x}_0^1\hat{x}_0^3(\hat{x}_1^2)^{-1}(\hat{x}_1^1)^{-1} $\\
        $x_1^2=\hat{x}_1^1\hat{x}_1^2(\hat{x}_1^2)^{-1}(\hat{x}_1^1)^{-1} =1$& $x_1^3= \hat{x}_0^1\hat{x}_1^3$\\
        $x_{2i}^2=\hat{x}_1^1\hat{x}_{2i}^2(\hat{x}_1^2)^{-1}(\hat{x}_1^1)^{-1}  \text{ for } 2\leq i \leq 4$& $x_{2i}^3=\hat{x}_0^1\hat{x}_{2i}^3 (\hat{x}_1^1)^{-1} \text{ for } 2\leq i \leq 4$ \\
         $x_{2i+1}^2=\hat{x}_1^1\hat{x}_{2i+1}^2 (\hat{x}_0^1)^{-1} \text{ for } 2\leq i \leq 4$&$x_{2i+1}^3=\hat{x}_0^1\hat{x}_{2i+1}^3(\hat{x}_3^3)^{-1}(\hat{x}_0^1)^{-1} \text{ for } 2\leq i \leq 4$\\
        \midrule
        $x_0^4=\hat{x}_1^1\hat{x}_1^2\hat{x}_0^4(\hat{x}_1^1)^{-1} $& $x_0^5=\hat{x}_0^1\hat{x}_3^3\hat{x}_0^5(\hat{x}_3^3)^{-1}(\hat{x}_0^1)^{-1}$\\
        $x_1^4=\hat{x}_1^1\hat{x}_1^2\hat{x}_1^4(\hat{x}_0^1)^{-1}$ & $x_1^5=\hat{x}_0^1\hat{x}_3^3\hat{x}_1^5(\hat{x}_3^3)^{-1}(\hat{x}_0^1)^{-1}$\\
        $x_{2i}^4=\hat{x}_1^1\hat{x}_1^2x_{2i}^4(\hat{x}_3^3)^{-1}(\hat{x}_0^1)^{-1} \text{ for } 2\leq i \leq 4$ & $x_{2i}^5=\hat{x}_0^1\hat{x}_3^3\hat{x}_{2i}^5(\hat{x}_0^1)^{-1} \text{ for } 2\leq i \leq 4$\\
        $x_{2i+1}^5=\hat{x}_0^1\hat{x}_3^3\hat{x}_{2i+1}^5(\hat{x}_1^2)^{-1}(\hat{x}_1^1)^{-1} \text{ for } 2\leq i \leq 4$ & $x_{2i+1}^5=\hat{x}_0^1\hat{x}_3^3\hat{x}_{2i+1}^5(\hat{x}_1^2)^{-1}(\hat{x}_1^1)^{-1} \text{ for } 2\leq i \leq 4$\\
        \bottomrule
        \end{tabular}
    \end{center}
    
\end{example}

This results in a set of generators for the fundamental group of the covering space of the 1-skeleton union the claw.
To get the fundamental group of the cover $\Sigma_g'$, we need to add relations.
The triangles we add to collapse the claw that sit above the path lifts
$\{\gamma_2,\gamma_3, \gamma_4, \gamma_5\}=\{\pathelement{x}_1^1, \pathelement{x}_0^1, \pathelement{x}_1^1\pathelement{x}_1^2,  \pathelement{x}_0^1\pathelement{x}_3^3\}$ introduce relations  $\{x_1^1=1, x_0^1=1, x_1^1x_1^2=1,  x_0^1x_3^3=1\}$ which we can clearly replace, after cancellation, by  $\{x_1^1=1, x_0^1=1, x_1^2=1,  x_3^3=1\}$; this is just the list of edges in the tree $T$ described in the proof of \Cref{RS.alg} part (2). 

To get the fundamental group of the punctured surface cover $\Sigma_g'$, it remains to add the relations corresponding to lifts of the single relation $w=x_0\dots x_9=1$ in $\pi_1(\sphere{2}-\{p_0,\dots,p_9\})$, starting at each of the 5 lifts $P_1$, $P_2$, $P_3$, $P_4$,and $P_5$ of the basepoint $P$ on $\sphere{2}$. Each such lift is a product of path-lifts $\pathelement{x}_i^j$ of the $x_i$.
Recall that $\pathelement{x}_i^j$ begins at $P_{\rho(x_i)}$ and ends at $P_{\rho(x_i)(j)}$.
Therefore, to automate the process of determining the lifts, we look at the consecutive action on the fiber over the basepoint $P$, namely the points $\{P_1, P_2, P_3, P_4,P_5\}$, of the permutations given by the map $\rho(x_i)$ for each $x_i$ in $w$.
The numbers $i$ in the columns below refer to the subscripts of the $P_i$ at the endpoints of the paths $\pathelement{x}_i^j$.

\begin{center}
    \begin{tikzcd}[row sep=0.12cm, column sep=1.00cm]
        1 \arrow[r, "\rho (x_{0})", bend left] \arrow[r,swap, "\pathelement{x}_0^1"]& 3 \arrow[r, "\rho(x_1)", bend left] \arrow[r,swap, "\pathelement{x}_1^3"]& 1 \arrow[r, "\rho(x_2)", bend left] \arrow[r,swap, "\pathelement{x}_2^1"]& 1 \arrow[r, "\rho(x_3)",bend left]\arrow[r,swap, "\pathelement{x}_3^1"] & 1 \arrow[r, "\rho(x_4)", bend left] \arrow[r,swap, "\pathelement{x}_4^1"]& 1 \arrow[r, "\rho(x_5)", bend left]\arrow[r,swap, "\pathelement{x}_5^1"] & 1 \arrow[r, "\rho(x_6)", bend left] \arrow[r,swap, "\pathelement{x}_6^1"]& 1 \arrow[r, "\rho(x_7)", bend left] \arrow[r,swap, "\pathelement{x}_7^1"]& 1 \arrow[r, "\rho(x_8)", bend left]\arrow[r,swap, "\pathelement{x}_8^1"]& 1 \arrow[r, "\rho(x_9)", bend left]\arrow[r,swap, "\pathelement{x}_9^1"] & 1 \\
        2 \arrow[r,swap, "\pathelement{x}_0^2"] & 1 \arrow[r,swap, "\pathelement{x}_1^1"]& 2 \arrow[r,swap, "\pathelement{x}_2^2"] & 4 \arrow[r,swap, "\pathelement{x}_3^4"] & 2 \arrow[r,swap, "\pathelement{x}_4^2"] & 4 \arrow[r,swap, "\pathelement{x}_5^4"] & 2 \arrow[r,swap, "\pathelement{x}_6^2"] & 4 \arrow[r,swap, "\pathelement{x}_7^4"] & 2 \arrow[r,swap, "\pathelement{x}_8^2"] & 4 \arrow[r,swap, "\pathelement{x}_9^4"] & 2 \\
        3 \arrow[r,swap, "\pathelement{x}_0^3"] & 4 \arrow[r,swap, "\pathelement{x}_1^4"]& 3 \arrow[r,swap, "\pathelement{x}_2^3"] & 2 \arrow[r,swap, "\pathelement{x}_3^2"] & 3 \arrow[r,swap, "\pathelement{x}_4^3"] & 2 \arrow[r,swap, "\pathelement{x}_5^2"] & 3 \arrow[r,swap, "\pathelement{x}_6^3"] & 2 \arrow[r,swap, "\pathelement{x}_7^2"] & 3\arrow[r,swap, "\pathelement{x}_8^3"] & 2\arrow[r,swap, "\pathelement{x}_9^2"] & 3\\
        4 \arrow[r,swap, "\pathelement{x}_0^4"] & 2 \arrow[r,swap, "\pathelement{x}_1^2"] & 4 \arrow[r,swap, "\pathelement{x}_2^4"]& 5 \arrow[r,swap, "\pathelement{x}_3^5"] & 4 \arrow[r,swap, "\pathelement{x}_4^4"] & 5 \arrow[r,swap, "\pathelement{x}_5^5"] & 4 \arrow[r,swap, "\pathelement{x}_6^4"] & 5 \arrow[r,swap, "\pathelement{x}_7^5"] & 4\arrow[r,swap, "\pathelement{x}_8^4"] & 5\arrow[r,swap, "\pathelement{x}_9^5"] & 4\\
        5 \arrow[r,swap, "\pathelement{x}_0^5"]& 5 \arrow[r,swap, "\pathelement{x}_1^5"]& 5 \arrow[r,swap, "\pathelement{x}_2^5"] & 3 \arrow[r,swap, "\pathelement{x}_3^3"] & 5 \arrow[r,swap, "\pathelement{x}_4^5"] & 3 \arrow[r,swap, "\pathelement{x}_5^3"] & 5 \arrow[r,swap, "\pathelement{x}_6^5"] & 3 \arrow[r,swap, "\pathelement{x}_7^3"] & 5  \arrow[r,swap, "\pathelement{x}_8^5"] & 3 \arrow[r,swap, "\pathelement{x}_9^3"] & 5\\
    \end{tikzcd}
\end{center}

Note that each $x_i$ in $w$ appears with positive exponent, so the consecutive action of $\rho$ matches the superscripts of the lifts in the relations $w_i$.  Note that since each lift of $w$ is a closed loop, the indices at the beginning and end of each line must be equal.

Therefore, following the process in \Cref{RS.alg} part (3), we can replace $\hat{x}_i^j$ with $x_i^j$ after adding the claw relations, and the relation $w=x_0\dots x_9$ lifts to the relations below: 
    		
$$w_1=x_0^1x_1^3x_2^1x_3^1x_4^1x_5^1x_6^1x_7^1x_8^1x_9^1$$ 
$$w_2=x_0^2x_1^1x_2^2x_3^4x_4^2x_5^4x_6^2x_7^4x_8^2x_9^4$$ 
$$w_3=x_0^3x_1^4x_2^3x_3^2x_4^3x_5^2x_6^3x_7^2x_8^3x_9^2$$ 
$$w_4=x_0^4x_1^2x_2^4x_3^5x_4^4x_5^5x_6^4x_7^5x_8^4x_9^5$$ 
$$w_5=x_0^5x_1^5x_2^5x_3^3x_4^5x_5^3x_6^5x_7^3x_8^5x_9^3.$$

To get the presentation for the branched cover, we add lifts of the relation $x_i^k=1$, for each disjoint $k$-cycle in the permutation $\rho(x_i)$, following step (4) of \Cref{RS.alg}. This corresponds to filling in meridians of the preimages of each of the points $p_0 \dots p_9$. Note that the number of preimages of $p_i$ is equal to the number of disjoint cycles in the permutation $\rho(x_i)$, which in this example is 2 for all $i$, so we get a total of 20 relations, listed below:

$$x_0^1x_0^3x_0^4x_0^2=1, x_0^5=1$$
$$x_1^1x_1^2x_1^4x_1^3=1, x_1^5=1$$
$$x_{2i}^1x_{2i}^2x_{2i}^4x_{2i}^5x_{2i}^3=1, x_{2i+1}^1=1 \text{ for }1\leq i \leq 4$$ $$x_{2i+1}^2x_{2i+1}^3x_{2i+1}^5x_{2i+1}^4=1, x_{2i+1}^1=1  \text{ for }1\leq i \leq 4.$$

Using Sage to simplify the presentation for $\pi_1(\Sigma,P_1)$ obtained by adding the relations $w_1,\dots, w_5$, the claw relations, and the 20 relations listed directly above, one can check that this is the group of a genus 11 surface, as noted in the trisection parameters for the $p=5$ case in \Cref{stevedore.tab}.

\begin{example}[The spun trefoil $\tau T(2,3)$]
    \label{spun_trefoil_group_trisection.ex}
    Let $K$ be the standard 4-bridge trisection of the spun trefoil knot with the coloring pictured in \Cref{fig:spun_trefoil_triplane_diagram_3_coloring}.
    We will compute the group trisection of the corresponding branched cover using \Cref{RS.alg}.\footnote{Details of the computation can be found in \texttt{spun\_trefoil\_detail.ipynb} of \cite{Ruppik_Knot_groups_2021}.}
    As before, we use $x_0\dots, x_7$ to denote the meridians of the punctures on the bridge sphere, labelled from left to right, to obtain the following presentation.
	
    \begin{align*}
        \pi_1(\sphere{2} - \{8 \textrm{ pts.}\})
        \cong
        \langle 
        x_{0}, x_{1}, \ldots, x_{7}
        \mid
        x_{0} x_{1} \cdots x_{7}
        \rangle
    \end{align*}
    
    The coloring in \Cref{fig:spun_trefoil_triplane_diagram_3_coloring} corresponds to the monodromy homomorphism:
    \begin{nalign}
        \label{eqn:rho_example_homomorphism}
        \rho \colon \pi_1(\sphere{2} - \{ p_{0}, p_{1}, \ldots, p_{7} \})
        & \rightarrow
        S_{3} \\
        x_{0}, x_{1} & \mapsto (1,2) \\
        x_{2}, x_{3}, x_{4}, x_{5}, x_{6}, x_{7} & \mapsto (1,3) \\
    \end{nalign}
    
    
    We now follow the steps in \Cref{RS.alg}.
    
    {\it Step 1:} The cover has 3 sheets. In the branched cover of the bridge sphere upstairs we denote our generators as
    \begin{equation*}
        \langle
        x_{0}^{1}, x_{1}^{1}, x_{2}^{1}, \ldots , x_{7}^{1}, 
        x_0^{2}, x_{1}^{2}, x_{2}^{2}, \ldots, 
        x_{0}^{3}, x_{1}^{3}, x_{2}^{3}, \ldots,
        x_{7}^{3}
        \rangle
    \end{equation*}

    {\it Step 2:} We see that a maximal tree in the lifted 1-skeleton is given by the edges $x_0^1, x_2^1$, so we use $x_0^1=1$ and $x_2^1=1$ as the claw relations.

    {\it Step 3:}  We first lift the relation for the bridge sphere.
    
    The lifts of the relation $w=x_0\dots x_7$ are: 
    $$ w_1=x_0^1x_1^2x_2^1x_3^3x_4^1x_5^3x_6^1x_7^3,  w_2=x_0^2x_1^1x_2^2x_3^2x_4^2x_5^2x_6^2x_7^2,  w_3=x_0^3x_1^3x_2^3x_3^1x_4^3x_5^1x_6^3x_7^1$$

    obtained by the method from \Cref{stevedore_RS_example.ex}, hence the presentation of the punctured surface upstairs is:

    \begin{equation}
	\label{unbranchedsurface_spun_trefoil_pi_1.eqn}
        \pi_1(\Sigma') = 
        \langle
        x_{0}^{1}, x_{1}^{1}, x_{2}^{1},
        \ldots
        x_{7}^{1}, x_0^{2},
        x_{1}^{2}, x_{2}^{2}, \ldots
        x_{0}^{3}, x_{1}^{3}, x_{2}^{3},
        \ldots,
        x_{7}^{3} \ \mid  
        x_0^1, x_2^1, w_1,w_2,w_3
        \rangle
    \end{equation}
    
   To obtain a presentation for the fundamental group of the unpunctured central surface in the branched cover, we add relations to close the small loops around the punctures, giving us the presentation:
    \begin{align}
        \label{surface_spun_trefoil_pi_1.eqn}
        \pi_1(\Sigma) =	\langle
            x_{0}^{1}, 
             \ldots  x_{7}^{1}, x_0^{2},
             \ldots
            x_{0}^{3}, 
            \ldots,
            x_{7}^{3}
            \mid
            & x_0^1, x_2^1, w_1, w_2, w_3, \\
            & x_0^1 x_0^2, x_0^3, x_1^1 x_1^2, x_1^3, x_2^1 x_2^3 , x_2^2, x_3^1x_3^3, x_3^2, \dots, x_7^1 x_7^3, x_7^2
        \rangle 
    \end{align}
    
    One can check that this simplifies to the presentation below, which is the fundamental group of a genus 2 surface:

    $$\pi_1(\Sigma) = \langle x_3^1, x_4^1, x_5^1, x_6^1 \mid (x_3^1)^{-1}x_4^1(x_5^1)^{-1}x_6^1x_3^1(x_4^1)^{-1}x_5^1(x_6^1)^{-1} \rangle.$$

    \vskip.05in
    
    Now we focus on the given tangles with boundary
    $\partial (\disk{3}, T_{\mu})
    = (\sphere{2}, \{ p_0, \ldots, p_{7} \})$.
    The map on fundamental groups induced from the boundary inclusion of the punctured sphere into the tangle complement is the quotient map given by introducing new relations, one for each tangle strand.
    Following the procedure for computing the group of a trivial tangle complement in \Cref{Wirtinger.sec}, we obtain the presentations

    \begin{equation*} 
        G_{\alpha}
        =
        \pi_1(\mathbb{D}^3 - T_{\alpha})
        = 
        \langle 
        x_{0}, x_{1}, \ldots, x_{7}
        \mid
        x_0 x_1 x_2 x_1 x_2 x_1^{-1} x_2^{-1} x_1^{-1},\\ \  x_1 x_2 x_1 x_2^{-1} x_1^{-1} x_7, \\ \  x_3x_6, \\ \  x_4x_5
        \rangle
    \end{equation*}

    \begin{equation*}	
        G_{\beta}
        =
        \pi_1(\mathbb{D}^3 - T_{\beta}) =
        \langle
            x_{0}, x_{1}, \ldots, x_{7}
        \mid
        x_0 x_1 x_2 x_1 x_2 x_1^{-1} x_2^{-1} x_1^{-1}, \\ \  x_1 x_2 x_1 x_2^{-1} x_1^{-1} x_5, \\  \  x_3x_4, \\ \  x_6x_7
        \rangle
    \end{equation*}
    
    \begin{equation*}	
        G_{\gamma}=\pi_1(\mathbb{D}^3 - T_{\gamma})=
            \langle
            x_{0}, x_{1}, \ldots, x_{7}
        \mid
        x_0 x_1 x_4 x_1 x_4 x_1^{-1} x_4^{-1} x_1^{-1},\\ \  x_1 x_4 x_1 x_4^{-1} x_1^{-1} x_7, \\  \ x_2x_3, \\ \  x_5x_6
            \rangle
    \end{equation*}

    Recall from \Cref{Wirtinger.sec} that the surface relation $x_0 \dots x_7 = 1$ is automatically implied by the relations for $G_\mu$ for each $\mu\in\{\alpha,\beta,\gamma\}$.

    Let us concentrate on $G_{\alpha}$ and carefully implement the algorithm for computing $\widetilde{G}_{\alpha}$, the fundamental group of the unbranched cover of $\disk{3} - T_\alpha$.
    
    We need to lift each of the 4 relations above, starting at each of the 3 lifts $P_1$, $P_2$, and $P_3$ of the basepoint $P$ on $\sphere{2}$.
    We will use the same process described in \Cref{stevedore_RS_example.ex}, though in this case, the $x_i$ appear in the original relations with both positive and negative exponents.

    We again look at the consecutive action of the corresponding permutations given by the map $\rho$ in \Cref{eqn:rho_example_homomorphism}.

    \begin{center}
        \begin{tikzcd}[row sep=0.12cm, column sep=1.00cm]
            1 \arrow[r, "\rho (x_{0})"] & 2 \arrow[r, "\rho(x_1)"] & 1 \arrow[r, "\rho(x_2)"] & 3 \arrow[r, "\rho(x_1)"] & 3 \arrow[r, "\rho(x_2)"] & 1 \arrow[r, "\rho(x_1^{-1})"] & 2 \arrow[r, "\rho(x_2^{-1})"] & 2 \arrow[r, "\rho(x_1^{-1})"] & 1 \\
            2 \arrow{r} & 1 \arrow{r} & 2 \arrow{r} & 2 \arrow{r} & 1 \arrow{r} & 3 \arrow{r} & 3 \arrow{r} & 1 \arrow{r} & 2 \\
            3 \arrow{r} & 3 \arrow{r} & 3 \arrow{r} & 1 \arrow{r} & 2 \arrow{r} & 3 \arrow{r} & 1 \arrow{r} & 3 \arrow{r} & 3
        \end{tikzcd}
    \end{center}
     
    Note that when we write the word, the algorithm (just from requiring the lift to be continuous) tells us that generators with power $+1$ get the superscript to the left of the arrow underneath it, while the generators with power $-1$ get the superscript to the right of the arrow underneath it.
    Note again that since this lift is a closed loop, the index at the beginning and end of each line has to be equal.

    The first relation 
    \begin{equation*} 
        x_0 x_1 x_2 x_1 x_2 x_1^{-1} x_2^{-1} x_1^{-1}
    \end{equation*}
    will therefore lift to the three relations
    \begin{equation*}
        x_0^1 x_1^2 x_2^1 x_1^3 x_2^3 (x_1^2)^{-1} (x_2^2)^{-1}(x_1^1)^{-1},
    \end{equation*}
    
    \begin{equation*}
        x_0^2 x_1^1 x_2^2 x_1^2 x_2^1 (x_1^3)^{-1} (x_2^1)^{-1}(x_1^2)^{-1},
    \end{equation*}
    
    \begin{equation*}
        x_0^3 x_1^3 x_2^3 x_1^1 x_2^2 (x_1^1)^{-1} (x_2^3)^{-1}(x_1^3)^{-1},
    \end{equation*}
    where the subscript indices are determined by the sequence of generators in the word and the corresponding permutations.

    The lifts of the remaining relations for each tangle complement are shown in \Cref{tab:spun_trefoil_relation_lifts}.
    
 \setlength{\tabcolsep}{10pt} 
    \renewcommand{\arraystretch}{1.25} 
    \begin{table}[]
        \centering
        \begin{tabular}{c|c|c}
            \toprule
            Tangle group & Original Relations & Lifted Relations \\
            \midrule
            \midrule
            $\pi_1(\disk{3}-T_\alpha)$ & $x_0 x_1 x_2 x_1 x_2 x_1^{-1} x_2^{-1} x_1^{-1}$ &$x_0^1 x_1^2 x_2^1 x_1^3 x_2^3 (x_1^2)^{-1} (x_2^2)^{-1}(x_1^1)^{-1}$\\
            &&$x_0^2 x_1^1 x_2^2 x_1^2 x_2^1 (x_1^3)^{-1} (x_2^1)^{-1}(x_1^2)^{-1}$\\
            &&$x_0^3 x_1^3 x_2^3 x_1^1 x_2^2 (x_1^1)^{-1} (x_2^3)^{-1}(x_1^3)^{-1}$\\
            \hline
            &$x_1 x_2 x_1 x_2^{-1} x_1^{-1} x_7$&$x_1^1 x_2^2 x_1^2(x_2^3)^{-1} (x_1^3)^{-1} x_7^3$ \\
            &&$x_1^2 x_2^1 x_1^3 (x_2^1){-1} (x_1^2){-1} x_7^2$\\
            &&$x_1^3 x_2^3 x_1^1 (x_2^2)^{-1} (x_1^1)^{-1} x_7^1$\\
            \hline
            &$x_3x_6$&$x_3^1x_6^3$\\
            & &$x_3^2x_6^2$\\
            &&$x_3^3x_6^1$\\
            \hline
            &$x_4x_5$&$x_4^1x_5^3$\\
            && $x_4^2x_5^2$\\
            &&$x_4^3x_5^1$ \\
            \midrule
            \midrule
            $\pi_1(\disk{3}-T_\beta)$ &$x_0 x_1 x_2 x_1 x_2 x_1^{-1} x_2^{-1} x_1^{-1}$&$x_0^1 x_1^2 x_2^1 x_1^3 x_2^3 (x_1^2)^{-1} (x_2^2)^{-1}(x_1^1)^{-1}$\\
            &&$x_0^2 x_1^1 x_2^2 x_1^2 x_2^1 (x_1^3)^{-1} (x_2^1)^{-1}(x_1^2)^{-1}$\\
            &&$x_0^3 x_1^3 x_2^3 x_1^1 x_2^2 (x_1^1)^{-1} (x_2^3)^{-1}(x_1^3)^{-1}$ \\
            \hline
            &  $x_1 x_2 x_1 x_2^{-1} x_1^{-1} x_5$ &$x_1^1 x_2^2 x_1^2(x_2^3)^{-1} (x_1^3)^{-1} x_5^3$\\
            &&$x_1^2 x_2^1 x_1^3 (x_2^1){-1} (x_1^2){-1} x_5^2$\\
            &&$x_1^3 x_2^3 x_1^1 (x_2^2)^{-1} (x_1^1)^{-1} x_5^1$\\ 
            \hline
            &  $x_3x_4$& $x_3^1x_4^3$\\
            &&$x_3^2x_4^2$ \\
            &&$x_3^3x_4^1$\\ 
            \hline
            &  $x_6x_7$&$x_6^1 x_7^3$\\
            && $x_6^2x_7^2$\\
            &&$x_6^3x_7^1$\\
            \midrule
            \midrule
            $\pi_1(\disk{3}-T_\gamma)$ &$x_0 x_1 x_4 x_1 x_4 x_1^{-1} x_4^{-1} x_1^{-1}$&$x_0^1 x_1^2 x_4^1 x_1^3 x_4^3 (x_1^2)^{-1} (x_4^2)^{-1}(x_1^1)^{-1}$\\
            &&$x_0^2 x_1^1 x_4^2 x_1^2 x_4^1 (x_1^3)^{-1} (x_4^1)^{-1}(x_1^2)^{-1}$\\
            &&$x_0^3 x_1^3 x_4^3 x_1^1 x_4^2 (x_1^1)^{-1} (x_4^3)^{-1}(x_1^3)^{-1}$\\
            \hline
            & $x_1 x_4 x_1 x_4^{-1} x_1^{-1} x_7$ &$x_1^1 x_4^2 x_1^2(x_4^3)^{-1} (x_1^3)^{-1} x_7^3$\\
            &&$x_1^2 x_4^1 x_1^3 (x_4^1){-1} (x_1^2){-1} x_7^2$\\
            &&$x_1^3 x_4^3 x_1^1 (x_4^2)^{-1} (x_1^1)^{-1} x_7^1$\\  
            \hline
            &$x_2x_3$&$x_2^1x_3^3$\\
            &&$x_2^2x_3^2$\\ 
            &&$x_2^3 x_3^1$\\
            \hline
            & $x_5x_6$&    $x_5^1 x_6^3$\\
            &&$x_5^2 x_6^2$\\
            &&$x_5^3 x_6^1$ \\
            \bottomrule
        \end{tabular}
        \caption{Lifts of Wirtinger relations for the tangle complements $\disk{3}-T_\mu$ for the spun trefoil in \Cref{fig:spun_trefoil_triplane_diagram_3_coloring}. }
        \label{tab:spun_trefoil_relation_lifts}
    \end{table}

    Of course, we also have the ``claw relations'', which in this case are just 
    $x_0^1=1$ and $x_2^1=1$.

    The group of the unbranched cover simplifies to  
    $$\widetilde{G}_{\alpha} = \langle x_1^1, x_3^1, x_4^1, x_5^1, x_6^1, x_0^2, x_1^2, x_2^2, x_3^2,x_4^2 \rangle.$$
    To get the group of the branched cover, we add in the meridians of the preimage of the branch set, which are the same as the ones for the central surface.  The group of the corresponding branched cover simplifies to
    $$\pi_1(H_\alpha)=\langle x_3^1, x_4^1 \rangle. $$

   Via the same process, the group of the unbranched cover for the $\beta$-tangle simplifies to

   $$\widetilde{G}_\beta = \langle x_1^1,x_3^1,x_4^1,x_6^1,x_7^1,x_0^2, x_1^2,x_2^2,x_3^2,x_6^2 \rangle,$$

   and the corresponding handlebody group simplifies to
   $$\pi_1(H_\beta)=\langle x_3^1, x_6^1\rangle.$$

    The group of the unbranched cover for the $\gamma$-tangle simplifies to
    
    $$\widetilde{G}_\gamma=\langle  x_1^1, x_3^1, x_5^1, x_6^1, x_7^1, x_0^2, x_2^2, x_4^2, x_5^2, x_7^3 \rangle, $$
    
    and the corresponding handlebody group is
    
    $$\pi_1(H_\gamma)=\langle x_4^1,x_5^1 \rangle. $$
    
    The group trisection for the branched cover of the spun trefoil is therefore obtained by taking push-outs of the following tripod:
    
    \begin{equation}
        \label{spun_trefoil_tripod.eq}
        \begin{tikzcd}
            & \pi_1(H_\alpha) = \langle x_3^1, x_4^1\rangle     \\
            \pi_1(\Sigma) = \langle x_3^1,x_4^1,x_5^1,x_6^1
            \mid
            (x_3^1)^{-1}x_4^1(x_5^1)^{-1}x_6^1x_3^1(x_4^1)^{-1}x_5^1(x_6^1)^{-1}
            \rangle
            \arrow[ru, two heads, "\iota_\alpha"]
            \arrow[r, two heads, "\iota_\beta"]
            \arrow[rd, two heads, "\iota_\gamma", swap] &
            \pi_1(H_\beta)=\langle x_3^1, x_6^1 \rangle  \\
            & \pi_1(H_\gamma)=\langle x_4^1,x_5^1 \rangle, \\         
        \end{tikzcd}
    \end{equation}
    
    and by examining the unsimplified relations for each tangle in \autoref{tab:spun_trefoil_relation_lifts}, one sees the inclusion maps are defined as follows:

    \begin{equation}
        \label{spun_trefoil_inclusions.eq}
        \begin{cases}
            \iota_\alpha(\loopelement{x}_3^1)&=\iota_\alpha(\loopelement{x}_6^1)=\loopelement{x}_3^1\quad \iota_\alpha(\loopelement{x}_4^1)=\iota_\alpha(\loopelement{x}_5^1)=\loopelement{x}_4^1\\
            \iota_\beta(\loopelement{x}_3^1)&=\iota_\beta(\loopelement{x}_4^1)=\loopelement{x}_3^1\quad\iota_\beta(\loopelement{x}_5^1)=0\quad \iota_\beta(\loopelement{x}_6^1)=\loopelement{x}_6^1\\
            \iota_\gamma(\loopelement{x}_5^1)&=\iota_\gamma(\loopelement{x}_6^1)=\loopelement{x}_5^1\quad\iota_\gamma(\loopelement{x}_3^1)=0\quad \iota_\gamma(\loopelement{x}_4^1)=\loopelement{x}_4^1\quad
        \end{cases}
    \end{equation}
    Note that the ranks of the surface and handlebody groups are consistent with the fact that the  branched cover of the sphere with given $\rho$ is a genus two surface, and branched covers over the tangles are genus two handlebodies. By computing pushouts, we see the pairwise unions of handlebodies are simply-connected, and hence the branched cover $X$ of $\sphere{4}$ along the spun trefoil is simply-connected, and the covering trisection parameters are therefore [2, [0, 0, 0]]. Note that an analysis of the trisection diagram of the cover by the method in \cite{cahn2017singular} shows that the cover is diffeomorphic to $\sphere{2}\times \sphere{2}$ \cite{blair2019note}.

\end{example}

\section{Computing Homology and the Intersection Form}
\label{homology:sec}

In this section, we will present a method for computing the homology groups and intersection form of a 4-manifold constructed by branched covering a colored bridge trisected surface.
Specifically, we will provide an algorithm that takes as input a colored tri-plane diagram and returns the homology and intersection form as its output.
This section can be understood as a bridge-trisected analog of the computations of homology groups of trisected manifolds in Feller-Klug-Schirmer-Zemke \cite{feller2018homology}.

\subsection{Computing homology groups from tri-plane data}

We begin with the computation of the homology groups of the branched cover, which combines \Cref{RS.alg} with the work of \cite{feller2018homology}.

\begin{algorithm} 
    \label{homology.alg}
    Let $\rho \colon \pi_1(\sphere{2} - \{2b\; pts.\}) \rightarrow S_n$ be a group homomorphism that factors to a colored group trisection corresponding to a colored bridge trisected surfaces $F \subset \sphere{4}$ as in \Cref{coloredgptrisection.def}.
    Let $X$ denote the corresponding branched cover.
    The homology groups $H_{\ast}(X; \mathbb{Z})$ can be computed as follows.

    \begin{enumerate}
        \item Abelianize the group trisection in \Cref{RS.alg} to obtain a basis for $H_1(\Sigma)$, $H_1(H_\mu)$ for $\mu\in\{\alpha,\beta,\gamma\}$, and $H_1(X)$, in terms of the generators $\loopelement{x}_i^j$ from \Cref{RS.alg}.
        
        \item Determine a basis for each Lagrangian $L_\epsilon=\text{ker}(\iota_\mu): H_1(\Sigma)\rightarrow H_1(H_\mu)$ in terms of the generators $\loopelement{x}_i^j$ from \Cref{RS.alg}.
        
        \item From \cite[Theorem 2.6]{feller2018homology}, we have 
        \[
            H_2 = \dfrac{L_\gamma \cap (L_\alpha+L_\beta)}{L_\gamma \cap L_\alpha +L_\gamma \cap L_\beta},
        \]
        
        \item From \cite[Corollary 2.4]{feller2018homology}, we have
        \[
            H_3= L_\alpha\cap L_\beta \cap L_\gamma
        \]
    \end{enumerate}
\end{algorithm}

\begin{example}\label{spun_trefoil_homology.ex}
    In this section, we detail the computation of the homology groups for the coloring of the trisected spun trefoil shown in \Cref{fig:spun_trefoil_triplane_diagram_3_coloring}.\footnote{Details of the computation can be found in \texttt{spun\_trefoil\_detail.ipynb} of \cite{Ruppik_Knot_groups_2021}.}

    Abelianizing the presentation for $\pi_1(\Sigma)$ in \Cref{spun_trefoil_group_trisection.ex}, we determine that
    \[
        H_1(\Sigma)
        =
        \mathbb{Z}[\loopelement{x}_3^1, \loopelement{x}_4^1, \loopelement{x}_5^1, \loopelement{x}_6^1],
    \]
    reflecting the fact that $\Sigma$ is a closed, orientable surface of genus 2. 
    
    Similarly, from the results of \Cref{spun_trefoil_group_trisection.ex}, the homology groups of the handlebodies are

    \begin{align*}
        H_1(H_\alpha) &= \mathbb{Z}[\loopelement{x}_3^1,\loopelement{x}_4^1] \\
        H_1(H_\beta) &= \mathbb{Z}[\loopelement{x}_3^1,\loopelement{x}_6^1] \\
        H_1(H_\gamma) &= \mathbb{Z}[\loopelement{x}_4^1,\loopelement{x}_5^1], \\
    \end{align*}
    
    with inclusion maps $\iota_\mu \colon H_1(\Sigma)\rightarrow H_1(H_\mu)$ defined on the above basis for $H_1(\Sigma)$ as in \Cref{spun_trefoil_inclusions.eq}.

    Next, we compute each Lagrangian $L_\mu$, which is the kernel of the map $\iota_\mu$:
    
    $$L_\alpha=\mathbb{Z}[\loopelement{x}_3^1-\loopelement{x}_6^1, \loopelement{x}_4^1-\loopelement{x}_5^1]$$
    $$L_\beta=\mathbb{Z}[\loopelement{x}_3^1-\loopelement{x}_4^1, \loopelement{x}_5^1]$$
    $$L_\gamma=\mathbb{Z}[\loopelement{x}_3^1, \loopelement{x}_5^1-\loopelement{x}_6^1]$$

    The ingredients for the computation of $H_2(X)$ are as follows:
    
    $$L_\alpha+L_\beta=\mathbb{Z}[\loopelement{x}_3^1,\loopelement{x}_4^1,\loopelement{x}_5^1,\loopelement{x}_6^1]$$
    $$L_\gamma\cap(L_\alpha+L_\beta)=\mathbb{Z}[\loopelement{x}_3^1,\loopelement{x}_5^1-\loopelement{x}_6^1]$$
    $$L_\gamma \cap L_\alpha= \mathbf{0}$$
    $$L_\gamma \cap L_\beta= \mathbf{0}$$
    
    Since
    \[
        H_2 = \dfrac{L_\gamma \cap (L_\alpha+L_\beta)}{L_\gamma \cap L_\alpha +L_\gamma \cap L_\beta},
    \]
    
    we have
    
    \[
        H_2(X)=\mathbb{Z}[\loopelement{x}_3^1,\loopelement{x}_5^1-\loopelement{x}_6^1].
    \]

\end{example}

\subsection{Computing the intersection form from tri-plane data.}

We next compute the intersection form on $H_2(X;\mathbb{Z})$ with respect to the basis of $H_2(X;\mathbb{Z})$ found in \Cref{homology.alg}.

From \cite[Thm. 2.6]{feller2018homology}, the intersection form on
\[
    H_2 = \dfrac{L_\gamma \cap (L_\alpha+L_\beta)}{L_\gamma \cap L_\alpha +L_\gamma \cap L_\beta},
\]
is given by
$\Phi \colon L_\gamma \cap (L_\alpha + L_\beta) \times L_\gamma \cap (L_\alpha + L_\beta) \rightarrow \mathbb{Z}$
given by
\[
    \Phi(x,y) = \langle x', y \rangle_\Sigma
\]
where $x'$ is any element of $L_\alpha$ which satisfies $x-x'\in L_\beta$, and $\langle \cdot , \cdot \rangle_\Sigma$ is the algebraic intersection form on $H_1(\Sigma;\mathbb{Z})$.

Therefore, to compute the intersection form on $H_2(X)$, it suffices to compute the intersection form on $H_1(\Sigma)$.  Viewing $x'$ and $y$ above as elements of $H_1(\Sigma)$, we can write each as concatenations of the path-lifts $\pathelement{x}_i^j$ of the generators $x_i$ of $\pi_1(\sphere{2}-\{2b\;pts\})$, as the corresponding loops $\loopelement{x}_i^j$ generate $H_1(\Sigma)$.  Two such paths $\pathelement{x}_i^j$ and $\pathelement{x}_k^l$ are disjoint outside small neighborhoods $U_1,\dots, U_n$ of the lifts $P_1,\dots,P_n$ of the basepoint $P \in \sphere{2}-\{2b\;pts\}$, unless $\pathelement{x}_i^j$ and $\pathelement{x}_k^l$ are either the same path or are path inverses of each other.  Therefore we can compute the algebraic intersection number of $x'$ and $y$ by examining how the path-lifts $\pathelement{x}_i^j$ of $x_i$ intersect each of the neighborhoods $U_i$ of $P_i$.  We record this information in a {\it configuration triple} (see \Cref{config_triple.def} and \Cref{linkpairneighborhood.fig}). We formalize this process and the computation of the intersection form on $H_2(X)$ in \Cref{intersectionform.alg} below.

\begin{algorithm}
    \label{intersectionform.alg}
    Let $\rho \colon \pi_1(\sphere{2}-\{2b\; pts.\})\rightarrow S_n$ be a group homomorphism that factors to a colored group trisection corresponding to a colored bridge trisected surfaces $F\subset \mathbb{S}^4$ as in \Cref{coloredgptrisection.def}.
    Let $X$ denote the corresponding branched cover.
    Let $x$ and $y$ be elements of $L_\gamma \cap (L_\alpha + L_\beta)$.
    Then the intersection number $\Phi(x,y)$ is given by the algorithm below.

    \begin{enumerate}
        \item Choose $x'\in L_\alpha$ such that $x-x'\in L_\beta$.
        \item Choose representatives of $x'$ and $y$ which are lifts of words in the generators $\{x_0,\dots, x_{2b-1}\}$ of $\pi_1(\sphere{2}-\{2b\;pts.\})$, so that each is a concatenation of path lifts $\pathelement{x}_i^j$ of the $x_i$ and their inverses:
        
        \[
            x'=[(\pathelement{x}_{a_1}^{b_1})^{\epsilon_1}\dots (\pathelement{x}_{a_k}^{b_k})^{\epsilon_k}]
        \]
        \[
            y=[(\pathelement{x}_{c_1}^{d_1})^{\delta_1}\dots (\pathelement{x}_{c_l}^{d_l})^{\delta_l}],
            \text{with }\epsilon_i ,\delta_i\in \{-1,1\}
        \]
        
        \item For each cyclically adjacent pair of paths $(\pathelement{x}_{a_s}^{b_s})^{\epsilon_s}(\pathelement{x}_{a_{s+1}}^{b_{s+1}})^{\epsilon_{s+1}}$  (respectively $(\pathelement{x}_{c_t}^{d_t})^{\delta_t}(\pathelement{x}_{c_{t+1}}^{d_{t+1}})^{\delta_{t+1}}$) in the above representative of $x'$ (respectively $y$), write down its corresponding \emph{configuration triple} $\tau(x',s)$ (respectively $\tau(y,t)$; see \Cref{config_triple.def} and \Cref{linkpairneighborhood.fig}.
        
        \item For each pair of configuration triples for the above representatives of $x'$ and $y$ with the same first coordinate $i$ (in other words, for each pair of consecutive paths passing through the same neighborhood $U_i$), compute the local intersection number of $x'$ and $y$ in $U_i$ according to \Cref{fig:int_form_4_endpts,fig:int_form_3_endpts,fig:int_form_2_endpts}.
        
        \item Compute the total intersection number of $x'$ and $y$ by adding the local intersection numbers above.
    \end{enumerate}
\end{algorithm}

\begin{proof}
Let $x$, $y\in H_2(X)$.  Our goal is to compute $\Phi(x,y)$.
Following step (1), we choose $x'\in L_\alpha$ such that $x-x'\in L_\beta$.  
Following step (2), write (representatives of) $x'$ and $y$ as lifts of words in the generators $\{x_0,\dots, x_{2b-1}\}$ of $\pi_1(\sphere{2}-\{2b \textrm{ pts.} \})$, so that each is a concatenation of path lifts $\pathelement{x}_i^j$ of the $x_i$ and their inverses:

\[
    x' = [(\pathelement{x}_{a_1}^{b_1})^{\epsilon_1} \dots (\pathelement{x}_{a_k}^{b_k})^{\epsilon_k}] \text{ with }\epsilon_i \in \{-1,1\}
\]
\[
    y = [(\pathelement{x}_{c_1}^{d_1})^{\delta_1} \dots (\pathelement{x}_{c_l}^{d_l})^{\delta_l}]
    \text{ with }\delta_i\in \{-1,1\}.
\]
Now we show that steps (3)-(5) compute $\langle x',y \rangle_\Sigma$.
The approach is to count intersection points of (representatives of) these classes inside a lift of a neighborhood $U$ of the basepoint $P$ on $\sphere{2}$. 

First choose an evenly covered neighborhood $U$ of $P$, with $U$ homeomorphic to an open disk, where its lift $\widetilde{U} = U_1 \sqcup \dots \sqcup U_n$, with $U_i$ a neighborhood of $P_i$.

\begin{figure}[htbp]
    \begin{tikzpicture}
        \node[anchor=south west,inner sep=0] at (0,0){\includegraphics[width=10cm]{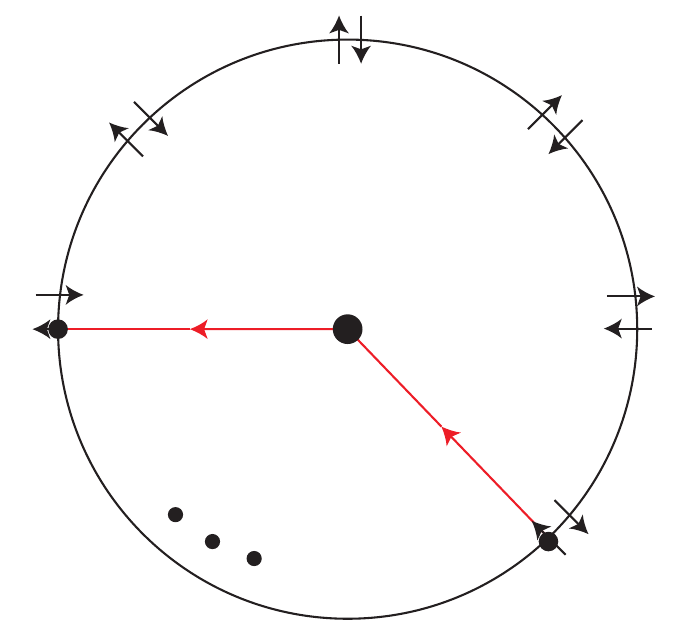}};
        \node at (5, 4) {\textcolor{black}{$P_{i}$}};
        \node at (1, 1) {\textcolor{black}{$U_{i}$}};
        \node at (4.5, 9.3) {\textcolor{violet}{$\pathelement{x}_0^i$}};
        \node at (6.3, 9.3) {\textcolor{violet}{$\pathelement{x}_0^{\rho(x_0)^{-1}(i)}$}};
        \node at (4.5, 8.3) {\textcolor{blue}{$0$}};
        \node at (5.6, 8.3) {\textcolor{blue}{$1$}};
        \node at (7.7, 8.4) {\textcolor{violet}{$\pathelement{x}_1^i$}};
        \node at (9.3, 7.4) {\textcolor{violet}{$\pathelement{x}_1^{\rho(x_1)^{-1}(i)}$}};
        \node at (7.3, 7.6) {\textcolor{blue}{$2$}};
        \node at (8.2, 6.7) {\textcolor{blue}{$3$}};
        \node at (9.6, 5.4) {\textcolor{violet}{$\pathelement{x}_2^i$}};
        \node at (10.2, 4.3) {\textcolor{violet}{$\pathelement{x}_2^{\rho(x_2)^{-1}(i)}$}};
        \node at (8.7, 5.4) {\textcolor{blue}{$4$}};
        \node at (8.7, 4.3) {\textcolor{blue}{$5$}};
        \node at (1, 7.3) {\textcolor{violet}{$\pathelement{x}_{2b-1}^i$}};
        \node at (2, 8.4) {\textcolor{violet}{$\pathelement{x}_{2b-1}^{\rho(x_{2b-1})^{-1}(i)}$}};
        \node at (3, 7.6) {\textcolor{blue}{$4b-1$}};
        \node at (2.4, 6.7) {\textcolor{blue}{$4b-2$}};
        \node at (7.5, 1.6){\textcolor{blue}{$j$}};
        \node at (1.4, 4.3){\textcolor{blue}{$k$}};
        \node at (9,.5) {Entry location $j$};
        \node at (-1,4.3) {Exit location $k$};
        \node at (5,5.2) {Fiber $i$};
        \node at (7.4, 3) {\textcolor{red}{$(\pathelement{x}_{a_s}^{b_s})^{\epsilon_s}$}};
        \node at (3,5) {\textcolor{red}{$(\pathelement{x}_{a_{s+1}}^{b_{s+1}})^{\epsilon_{s+1}}$}};
    \end{tikzpicture}
    \caption{
        A pair of consecutive paths $(\pathelement{x}_{a_s}^{b_s})^{\epsilon_s} (\pathelement{x}_{a_{s+1}}^{b_{s+1}})^{\epsilon_{s+1}}$ occuring in a choice of representative of $x'$, with configuration triple $\tau(x',s)=(i,j,k).$ Entry and exit locations are labelled clockwise around $\partial U_i$ in blue.
        \label{linkpairneighborhood.fig}
    }
\end{figure}
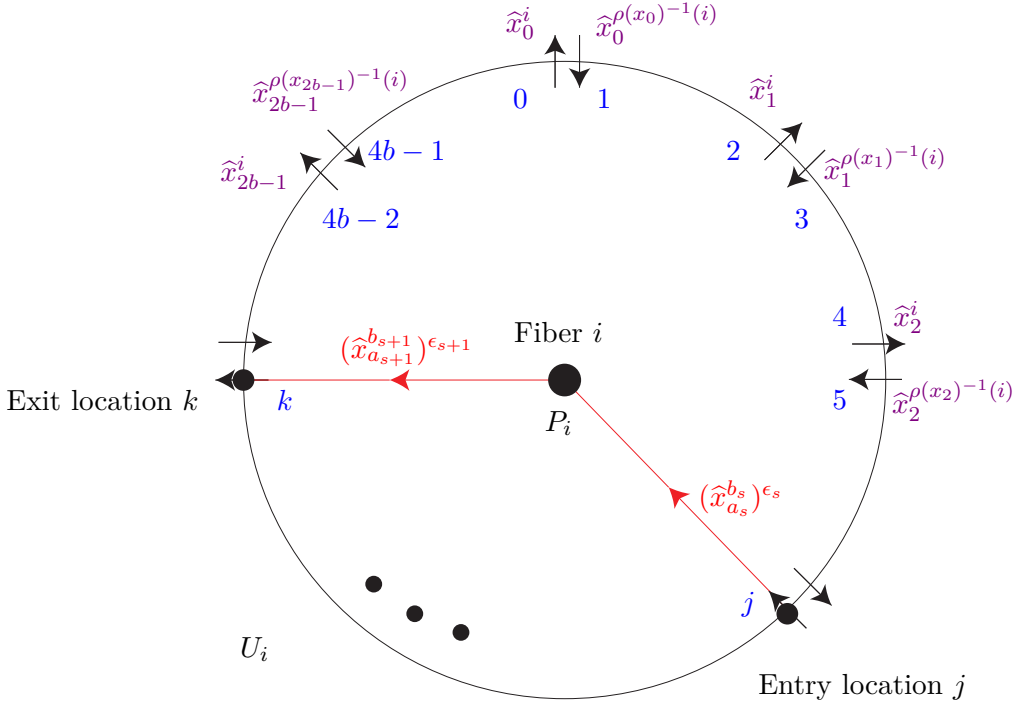

\begin{definition}
    \label{config_triple.def}
    For each pair of consecutive paths $(\pathelement{x}_{a_s}^{b_s})^{\epsilon_s}(\pathelement{x}_{a_{s+1}}^{b_{s+1}})^{\epsilon_{s+1}}$ in $x'$ (and similarly for $y$), we write down a corresponding \emph{configuration triple} $\tau(x',s)=(i,j,k)$ where $(i,j,k)$ are defined below.
    Loosely speaking, this triple encodes which neighborhood $U_i$ this pair of paths passes through, as well as the positions $j$ and $k$ on the boundary of $U_i$ where pair of paths enters and exits $U_i$, respectively.  
    See \Cref{linkpairneighborhood.fig}. The precise definitions of $i$, $j$, and $k$ are as follows:
    \begin{enumerate}
        \item The {\it fiber number} $i\in\{1,\dots,n\}$, which is the final point $P_i$ of $(\pathelement{x}_{a_s}^{b_s})^{\epsilon_s}$, and is equal to the initial point of $(\pathelement{x}_{a_{s+1}}^{b_{s+1}})^{\epsilon_{s+1}}$.  Explicitly, we have

        \[ i=\begin{cases}
            \rho(x_{a_s})(b_s)=b_{s+1} &\text{ if } \epsilon_s = 1 \text{ and }  \epsilon_{s+1} = 1\\
            b_s=\rho(x_{a_{s+1}})(b_{s+1}) &\text{ if } \epsilon_s = -1 \text{ and }  \epsilon_{s+1} = -1\\
            \rho(x_{a_s})(b_s)=\rho(x_{a_{s+1}})(b_{s+1})  &\text{ if } \epsilon_s = 1 \text{ and }  \epsilon_{s+1} = -1\\
            b_s=b_{s+1} &\text{ if } \epsilon_s = -1 \text{ and }  \epsilon_{s+1} = 1.\\
            
        \end{cases} \]

        \item The \emph{entry location} $j\in \{0,\dots, 4b-1\}$, which denotes the position on the boundary of the neighborhood $U_i$ of $P_i$ where the path $(\pathelement{x}_{a_s}^{b_s})^{\epsilon_s}$ enters $U_i$.
        Explicitly,
        
        \[
            j
            =
            \begin{cases}
                2 a_s +1 &\text{ if }\epsilon_s=1\\
                2 a_s &\text{ if }\epsilon_s=-1\\
            \end{cases}
        \]
        
        \item The \emph{exit location} $k\in \{0,\dots, 4b-1\}$, which denotes the position on the boundary of the neighborhood $U_i$ of $P_i$ where the path $(\pathelement{x}_{a_{s+1}}^{b_{s+1}})^{\epsilon_{s+1}}$ exits $U_i$.
        Explicitly,
        
        \[
            k
            =
            \begin{cases}
                2 a_{s+1} &\text{ if }\epsilon_s=1\\
                2 a_{s+1} + 1 &\text{ if }\epsilon_s=-1\\
            \end{cases}
        \]
        
    \end{enumerate}
\end{definition}

We then write down the configuration triple $\tau(y,t)$ for each pair of consecutive paths $(\pathelement{x}_{c_t}^{d_t})^{\epsilon_t}(\pathelement{x}_{c_{t+1}}^{d_{t+1}})^{\epsilon_{t+1}}$ in $y$ using the same process above.

Since the surface $\Sigma$ is orientable, we can choose a \emph{positive push-off} $(x')^+$, by pushing this curve slightly to the right, relative to the orientation of the curve.
Then, we choose a positive push-off $(y)^+$, by pushing this curve slightly to the right relative to its orientation, but by a smaller radius than $(x')^+$.

Since the loops $x_i$ are disjoint except at $P$, the push-offs can be chosen so that $(x')^+$ and $(y)^+$ intersect transversely in double points in the neighborhoods $U_i$, and so that outside the $U_i$, $(x')^+$ and $(y)^+$ are disjoint. 

After taking push-offs, intersection points between $x'$ and $y$ occur when $\tau(x',s)$ and $\tau(y,t)$ have the same fiber number, or in other words, when the consecutive pairs $(\pathelement{x}_{a_s}^{b_s})^{\epsilon_s}(\pathelement{x}_{a_{s+1}}^{b_{s+1}})^{\epsilon_{s+1}}$ and $(\pathelement{x}_{c_t}^{d_t})^{\epsilon_t}(\pathelement{x}_{c_{t+1}}^{d_{t+1}})^{\epsilon_{t+1}}$ pass through the same $U_i$.  One example is shown in \Cref{intersection_in_U.fig}, in the case where $\epsilon_s=\epsilon_t=1$.

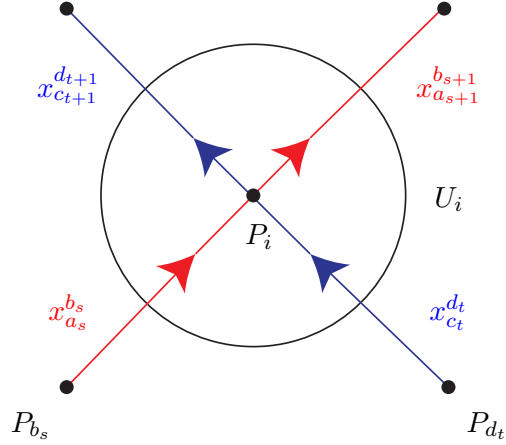
\begin{figure}
   \begin{tikzpicture}
        \node[anchor=south west,inner sep=0] at (0,0){\includegraphics[width=6cm]{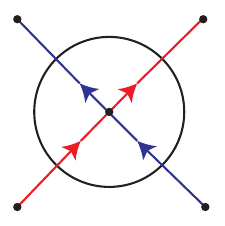}};
        \node at (3.0, 2.5) {\textcolor{black}{$P_{i}$}};
        \node at (.5, 1.5) {\textcolor{red}{$x_{a_s}^{b_s}$}};
        \node at (5.5, 4.5) {\textcolor{red}{$x_{a_{s+1}}^{b_{s+1}}$}};
        \node at (.5, 4.5) {\textcolor{blue}{$x_{c_{t+1}}^{d_{t+1}}$}};
        \node at (5.5, 1.5) {\textcolor{blue}{$x_{c_t}^{d_t}$}};
        \node at (5.5, 3) {\textcolor{black}{$U_{i}$}};
        \node at (0,0) {\textcolor{black}{$P_{b_s}$}};
        \node at (6,0) {\textcolor{black}{$P_{d_t}$}};
    \end{tikzpicture}
    \caption{
        Two pairs of consecutive paths $\pathelement{x}_{a_s}^{b_s} \pathelement{x}_{a_{s+1}}^{b_{s+1}}$ and $\pathelement{x}_{c_t}^{d_t} \pathelement{x}_{c_{t+1}}^{d_{t+1}}$ intersecting in the neighborhood $U_i$ of $P_i$.
        \label{intersection_in_U.fig}
    }
\end{figure}
 
In general, these pairs of consecutive paths can intersect $U_i$ in 4, 3, 2, or 1 distinct endpoints, as shown in the top rows of \Cref{fig:int_form_4_endpts,fig:int_form_3_endpts,fig:int_form_2_endpts}.
In each case, we compute the contribution to the algebraic intersection number by taking push-offs as shown in the bottom rows of the same figures, and these figures should be viewed as local models for the intersections between push-offs. The local intersection numbers for each possible configuration are shown in \Cref{fig:int_form_4_endpts,fig:int_form_3_endpts,fig:int_form_2_endpts} as well, and can therefore be read off from the configuration triples.\end{proof}

\begin{figure}
 \includegraphics[width=\textwidth]{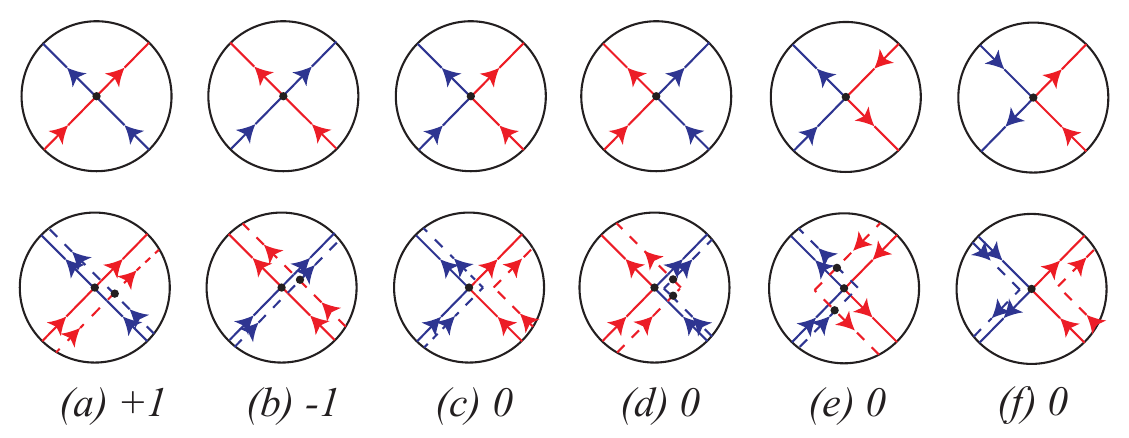}
    \caption{
        Local intersection numbers of the push-offs $(x')^+$ (red) and $(y)^+$ (blue) with $U_i$, in the case where $x'$ and $y$ intersect $\partial U_i$ in 4 distinct endpoints.
        \label{fig:int_form_4_endpts}
    }
\end{figure}

\begin{figure}
   \includegraphics[width=\textwidth]{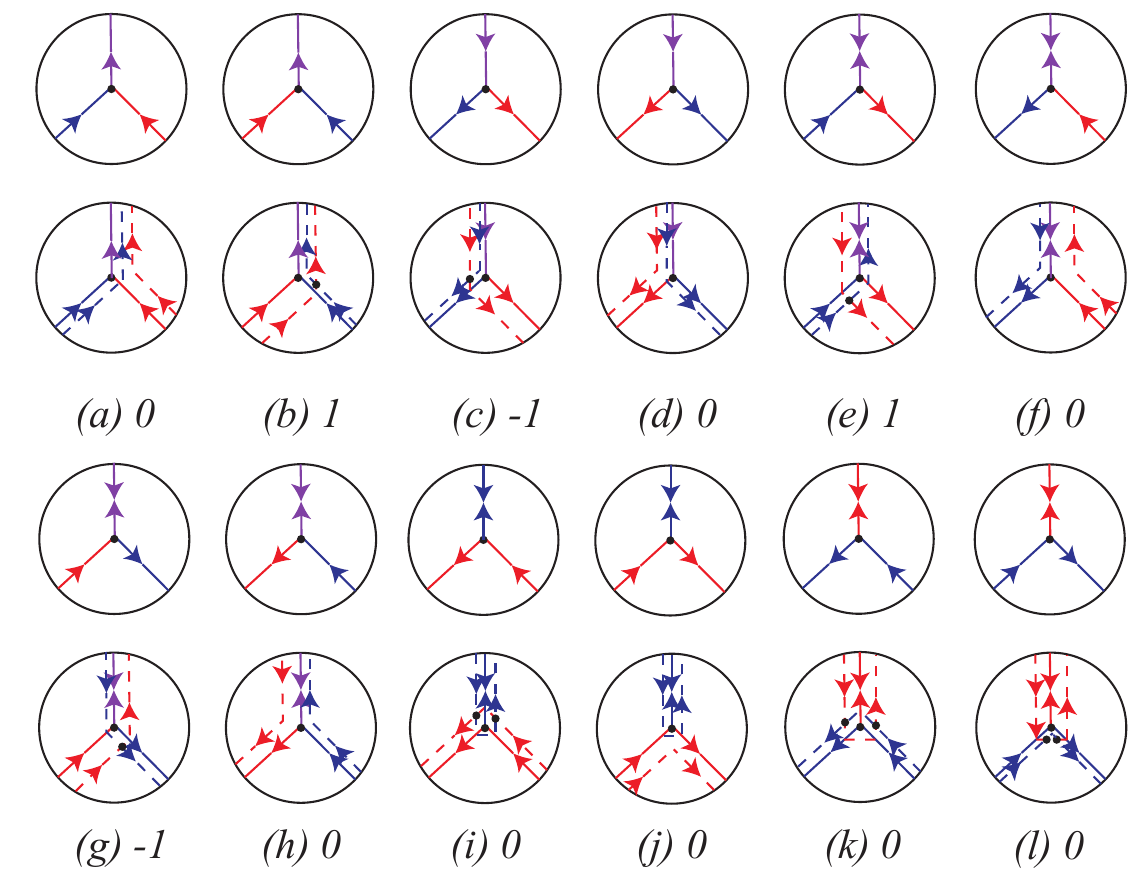}
    \caption{
        Local intersection numbers of the push-offs $(x')^+$ (red) and $(y)^+$ (blue) with $U_i$, in the case where $x'$ and $y$ intersect $\partial U_i$ in 3 distinct endpoints.
        \label{fig:int_form_3_endpts}
    }
\end{figure}

\begin{figure}
\includegraphics[width=4in]{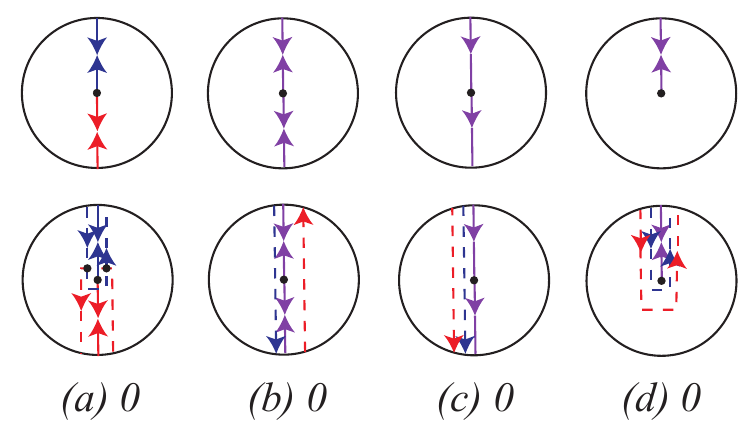}
    \caption{
        Local intersection numbers of the push-offs $(x')^+$ (red) and $(y)^+$ (blue) with $U_i$, in the case where $x'$ and $y$ intersect $\partial U_i$ in 2 or 1 distinct endpoints.
        \label{fig:int_form_2_endpts}
    }
\end{figure}

\begin{remark}
    The form $\langle \cdot, \cdot \rangle_\Sigma$ is skew-symmetric, though this is not obvious from \Cref{fig:int_form_4_endpts,fig:int_form_3_endpts,fig:int_form_2_endpts}.
    While interchanging red and blue in the individual local pictures in \Cref{fig:int_form_3_endpts,fig:int_form_2_endpts} does not obviously produce a skew-symmetric form as in \Cref{fig:int_form_4_endpts}, the sum over all the local pictures involved is skew-symmetric.
\end{remark}

\begin{example}\label{spun_trefoil_int_form.ex}
    We now use \Cref{intersectionform.alg} to compute the intersection form on the branched cover of the spun trefoil with respect to our basis $\{\loopelement{x}_3^1,\loopelement{x}_5^1-\loopelement{x}_6^1\}$ of $H_2(X)$ from \Cref{spun_trefoil_homology.ex}.  Recall that 
    \[
        \Phi(x,y) = \langle x', y \rangle_\Sigma
    \]
    where $x'$ is any element of $L_\alpha$ which satisfies $x-x'\in L_\beta$.
    
    For $x=\loopelement{x}_3^1$, we write $x=(\loopelement{x}_4^1-\loopelement{x}_5^1)+(\loopelement{x}_3^1-\loopelement{x}_4^1+\loopelement{x}_5^1),$ where $\loopelement{x}_4^1-\loopelement{x}_5^1\in L_\alpha$ and $\loopelement{x}_3^1-\loopelement{x}_4^1+\loopelement{x}_5^1 \in L_\beta$.
    Then we take $x'=\loopelement{x}_4^1-\loopelement{x}_5^1.$
    
    Similarly, for $x=\loopelement{x}_5^1-\loopelement{x}_6^1$, we write $x=((\loopelement{x}_3^1-\loopelement{x}_6^1)-(\loopelement{x}_4^1-\loopelement{x}_5^1))-(\loopelement{x}_3^1-\loopelement{x}_4^1)$ where $(\loopelement{x}_3^1-\loopelement{x}_6^1)-(\loopelement{x}_4^1-\loopelement{x}_5^1)\in L_\alpha$ and $\loopelement{x}_3^1-\loopelement{x}_4^1\in L_\beta$.
    Then we take $x'=(\loopelement{x}_3^1-\loopelement{x}_6^1)-(\loopelement{x}_4^1-\loopelement{x}_5^1).$
    
    In what follows, we compute the intersection numbers of these loops, to obtain the following representing matrix of the intersection form:
    
    \begin{center}
        \begin{tabular}{c||c|c}
            &$\loopelement{x}_3^1$&$\loopelement{x}_5^1-\loopelement{x}_6^1$\\
            \midrule
            \midrule
            $\loopelement{x}_4^1-\loopelement{x}_5^1$& $0$ & $-1$ \\
            \midrule
            $\loopelement{x}_3^1-\loopelement{x}_6^1+\loopelement{x}_5^1 - \loopelement{x}_4^1$ & $-1$ & $0$ \\
        \end{tabular}
    \end{center}  

    Because this is a relatively small example, we could of course draw the above loops and compute the above intersection numbers by hand. However, our goal here is to give a worked example of \Cref{intersectionform.alg}, which we use to compute (by computer) the intersection form for the families of larger examples  in \Cref{sec:twist_spun}, \Cref{stevedore.sec}, \Cref{sec:suciu}, and \Cref{xi.sec}.  
    
    Recall that the steps for computing $\langle x',y\rangle_\Sigma$ are as follows:
    
    \begin{enumerate}
        \item Write $x'$ and $y$ as products of the paths $\pathelement{x}_i^j$.  
        \item For each cyclically adjacent pair of paths $\pathelement{x}_i^j$ in $x'$ (respectively $y$), write down its corresponding configuration triple.
        \item For each pair of configuration triples of $x'$ and $y$ with the same first coordinate $i$ (in other words, corresponding to the same neighborhood $U_i$ of $P_i$), compute the local intersection number in $U_i$ using \Cref{fig:int_form_4_endpts,fig:int_form_3_endpts,fig:int_form_2_endpts} 
        \item Compute the total intersection number by adding the results of (3).
    \end{enumerate}
    
    We carry out these steps for the upper-left entry of the intersection form $\langle \loopelement{x}_4^1(\loopelement{x}_5^1)^{-1}, \loopelement{x}_3^1\rangle_\Sigma = 0$, switching back to multiplicative notation for the purpose of specifying a cyclic order on the paths $\pathelement{x}_i^j$ in step (2).  
    
    {\it Intersection number } $\langle \loopelement{x}_4^1(\loopelement{x}_5^1)^{-1}, \loopelement{x}_3^1\rangle_\Sigma = 0$:
    
    \begin{enumerate}
        \item Rewriting using the claw map, $\loopelement{x}_4^1(\loopelement{x}_5^1)^{-1}=\pathelement{x}_4^1(\pathelement{x}_5^1)^{-1}$, and $\loopelement{x}_3^1=\pathelement{x}_3^1(\pathelement{x}_2^1)^{-1}.$ 
        \item Configuration triples for cyclically adjacent paths in $\pathelement{x}_4^1(\pathelement{x}_5^1)^{-1}$:
        $$(\pathelement{x}_4^1,(\pathelement{x}_5^1)^{-1})\mapsto (3,9,11) $$
        $$((\pathelement{x}_5^1)^{-1},\pathelement{x}_4^1)\mapsto (1, 10, 8)$$
         Configuration triples for cyclically adjacent paths in  $\pathelement{x}_3^1(\pathelement{x}_2^1)^{-1}$:
        
        $$(\pathelement{x}_3^1,(\pathelement{x}_2^1)^{-1})\mapsto (3,7,5)$$
        $$((\pathelement{x}_2^1)^{-1},\pathelement{x}_3^1)\mapsto (1,4,6)$$
        \item The pair of triples $((3,9,11),(3,7,5))$ contributes 0 to the intersection number, as does the pair of triples $((1, 10, 8),(1,4,6))$.
        \item It follows that $\langle \loopelement{x}_4^1(\loopelement{x}_5^1)^{-1}, \loopelement{x}_3^1\rangle_\Sigma = 0$.
    \end{enumerate}
    The configuration triples and local intersection numbers for the other three entries are given in ~\Cref{tab:spun_trefoil_intersections}, completing the computation of the intersection form for the branched cover of the spun trefoil.

 \setlength{\tabcolsep}{10pt} 
    \renewcommand{\arraystretch}{1.5} 
    \begin{table}[]
        \centering
        \begin{tabular}{c|c|c|c}
        \toprule
        Consecutive Paths & Consecutive Paths &  Local Intersection & Total Int. \\
        and Triples for $x'$ & and Triples for $y$ & Numbers & Number\\
        \midrule
        \midrule
        $x'=\pathelement{x}_4^1(\pathelement{x}_5^1)^{-1}$&$y=\pathelement{x}_3^1(\pathelement{x}_2^1)^{-1}$ &&0\\
        \hline
        $(\pathelement{x}_4^1,(\pathelement{x}_5^1)^{-1}) \mapsto (3,9,11)$&$(\pathelement{x}_3^1,(\pathelement{x}_2^1)^{-1})\mapsto (3,7,5)$&$(3,9,11)\cdot(3,7,5)\mapsto 0$&\\
        $((\pathelement{x}_5^1)^{-1},\pathelement{x}_4^1)\mapsto (1, 10, 8)$&$((\pathelement{x}_2^1)^{-1},\pathelement{x}_3^1)\mapsto (1,4,6)$&$(1,10,8)\cdot(1,4,6)\mapsto 0$& \\
        \hline
        $x'=\pathelement{x}_4^1(\pathelement{x}_5^1)^{-1}$&$y=\pathelement{x}_5^1(\pathelement{x}_6^1)^{-1}$&&$-1$\\
        \hline
         $(\pathelement{x}_4^1,(\pathelement{x}_5^1)^{-1})\mapsto (3,9,11)$&$(\pathelement{x}_5^1,(\pathelement{x}_6^1)^{-1})\mapsto (3,11,13)$
        &$(3,9,11)\cdot(3,11,13)\mapsto -1$&\\
        $((\pathelement{x}_5^1)^{-1},\pathelement{x}_4^1)\mapsto (1,10,8)$&$((\pathelement{x}_6^1)^{-1},\pathelement{x}_5^1)\mapsto (1,12,10)$&$(1,10,8)\cdot(1,12,10)\mapsto 0$&\\
        \hline
        $x'=\pathelement{x}_3^1(\pathelement{x}_6^1)^{-1}\pathelement{x}_5^1(\pathelement{x}_4^1)^{-1}$&$y=\pathelement{x}_3^1(\pathelement{x}_2^1)^{-1}$&&$-1$\\
        \hline
        $(\pathelement{x}_3^1,(\pathelement{x}_6^1)^{-1})\mapsto (3,7,13)$&$(\pathelement{x}_3^1,(\pathelement{x}_2^1)^{-1})\mapsto (3,7,5)$
        &$(3,7,13)\cdot (3,7,5)\mapsto -1$ &\\
        $((\pathelement{x}_6^1)^{-1},\pathelement{x}_5^1)\mapsto (1,12,10)$&$((\pathelement{x}_2^1)^{-1},\pathelement{x}_3^1)\mapsto (1,4,6)$&$(3,11,9)\cdot(3,7,5)\mapsto 0$&\\
        $(\pathelement{x}_5^1,(\pathelement{x}_4^1)^{-1})\mapsto (3,11,9)$&&$(1,12,10)\cdot(1,4,6)\mapsto 0$&\\
        $((\pathelement{x}_4^1)^{-1},\pathelement{x}_3^1)\mapsto  (1,8,6)$&&$(1,8,6)\cdot(1,4,6)\mapsto 0$&\\
        \hline
        $x'=\pathelement{x}_3^1(\pathelement{x}_6^1)^{-1}\pathelement{x}_5^1({x}_4^1)^{-1}$&$y=\pathelement{x}_5^1(\pathelement{x}_6^1)^{-1}$&&$0$\\
        \hline
        $(\pathelement{x}_3^1,(\pathelement{x}_6^1)^{-1})\mapsto (3,7,13)$&$(\pathelement{x}_5^1,(\pathelement{x}_6^1)^{-1})\mapsto (3,11,13)$
        &$(3,7,13)\cdot (3,11,13)\mapsto 0$&\\
        $((\pathelement{x}_6^1)^{-1},\pathelement{x}_5^1)\mapsto (1,12,10)$&$(({x}_6^1)^{-1},\pathelement{x}_5^1)\mapsto (1,12,10)$& $(1,12,10)\cdot (1,12,10)\mapsto 0$&\\
        $(\pathelement{x}_5^1,(\pathelement{x}_4^1)^{-1})\mapsto (3,11,9)$&& $(3,11,9)\cdot(3,11,13)\mapsto 0$&\\
        $((\pathelement{x}_4^1)^{-1},\pathelement{x}_3^1)\mapsto  (1,8,6)$&&$(1,8,6)\cdot(1,12,10)\mapsto 0$&\\
        \bottomrule
        \end{tabular}
        \caption{
            Configuration triples and local intersection numbers for the spun trefoil.
            \label{tab:spun_trefoil_intersections}
        }
    \end{table}
   \end{example}

\section{Branched Covers of Twist-Spun Knots}
\label{sec:twist_spun}
We begin by applying our algorithms to various 
 branched covers of twist-spun knots in $\sphere{4}$.  In many (but not all) cases, the resulting branched cover is itself a spun manifold.  We first discuss the problem of when a branched cover of $(\sphere{3},K)$ extends to a branched cover of its $k$-twist spin.

\subsection{Branched covers of twist-spun knots}

Following \cite{meier2018trisections}, the $k$-spin $\mathcal{S}^k(\sphere{3},K)$ is given by
\[
    \mathcal{S}^k(\sphere{3},K)
    =
    ((\sphere{3},K)^\circ\times \sphere{1})
    \bigcup_{\tau^k}
    (\sphere{2}\times \disk{2},\{\mathfrak{n},\mathfrak{s}\} \times \disk{2})
\]
where $(\sphere{3},K)^\circ$ is the 3-ball-tangle-pair obtained from the pair $(\sphere{3},K)$ by removing a pair $(\disk{3},\disk{1})$, and $\tau^k$ is the $k^{\text{th}}$ power of the Gluck twist $\tau \colon \sphere{2} \times \sphere{1} \rightarrow \sphere{2}\times \sphere{1}$.  

\begin{proposition} 
    \label{extendtospin.prop}
    Suppose $f \colon (M,L) \rightarrow (\sphere{3},K)$ is an $n$-fold branched cover with corresponding monodromy homomorphism $\rho \colon \pi_1(\sphere{3}-K)\rightarrow S_n$.
    Then $f$ extends to a corresponding branched cover of the $k$-twist spin $\tau^kK$ of $K$ in $\sphere{4}$ if and only if $\rho$ is invariant under conjugation by $\rho(\mu_0)^k$, where $\mu_0$ is the meridian of $K$ about the arc adjacent to $\mathfrak{n}$.
    In particular, such an extension exists for cyclic covers, and when $k$ is even, for dihedral and simple covers.
\end{proposition}

\begin{proof}
    Let $\mu_0$ be a meridian of $K$ about the arc adjacent to $\mathfrak{n}$.
    A presentation for $\pi_1(\sphere{4}-\tau^k K)$ can be obtained from a presentation of $\pi_1(\sphere{3}-K)$ by adding relations of the form $\mu_0^{k} x \mu_0^{-k}=x$ for all $x\in\pi_1(\sphere{3}-K)$ \cite[Ex. 5, p.98]{rolfsen2003knots}.
    Then $\rho$ extends to the $k$-twist spin if and only if $\rho(\mu_0)^k\rho(x)\rho(\mu_0)^{-k}=\rho(x)$ for all $x\in \pi_1(\sphere{3}-K)$. 
\end{proof}

\subsection{Spun manifold-knot pairs as branched covers}

Let $(M,K)$ be a pair consisting of a closed, connected 3-manifold $M$ and a knot $K$ in $M$.
Meier \cite{meier2018trisections} introduced the notion of the $k$-twist spin of the pair $(M,K)$, defined as follows:

\[
    \mathcal{S}^k(M,K)
    =
    ((M,K)^{\circ}\times \sphere{1})
    \bigcup_{\tau^{k}}
    (\sphere{2}\times \disk{2},\{\mathfrak{n},\mathfrak{s}\}\times \disk{2}).
\]

It is well-known that the spin $\mathcal{S}^k(M)$ of the manifold $M$ depends only on the parity of $k$, but as noted in \cite{meier2018trisections}, the isotopy class of the resulting spun knot $\tau^kK$ in $\mathcal{S}^k(M)$ could vary with $k$ when the parity of $k$ is fixed.

Now suppose $f \colon (M,L) \rightarrow (\sphere{3}, K)$ is an $n$-fold branched cover with corresponding monodromy homomorphism $\rho \colon \pi_1(\sphere{3}-K)\rightarrow S_n$, satisfying the criteria of \Cref{extendtospin.prop}, so that the cover $f$ extends to a cover $\mathcal{S}(f) \colon (X, F)\rightarrow (\sphere{4},\tau^k K)$.
It is natural to ask when $(X,F)$ is itself a spun manifold $\mathcal{S}^l(M,L)$ for some twist parameter $l$.
A necessary condition is that the pre-image $L$ of $K$ be a knot, and hence that the covering is cyclic.
The twist parameter in the cover can be determined from the twist parameter $k$ of the branching set and the degree of the cover $n$, as follows.

\begin{proposition}
    \label{coverisspun.prop}
    Let $f \colon (M,L) \rightarrow (\sphere{3},K)$ be an $n$-fold cyclic cover.
    Let $(\sphere{4},\tau^k K)$ denote the $k$-twist spin $\mathcal{S}^k(\sphere{3}, K)$.
    Let $\widehat{f}$ be the extension $\widehat{f} \colon (X,F)\rightarrow (\sphere{4},\tau^kK)$ guaranteed by \Cref{extendtospin.prop}.
    Then if $n$ divides $k$, $(X,F) = \mathcal{S}^{k/n}(M, L)$.
\end{proposition}

\begin{proof}
    We construct the cover of the decomposition 
    \[
        \mathcal{S}^k(\sphere{3},K)
        =
        ((\sphere{3},K)^\circ\times \sphere{1})
        \bigcup_{\tau^k}
        (\sphere{2}\times \disk{2},\{\mathfrak{n},\mathfrak{s}\}\times \disk{2}).
    \]
    Let $R_{2\pi/n} \colon \sphere{2}\rightarrow \sphere{2}$ denote the $2\pi/n$-rotation of $\sphere{2}$ about an axis through the poles ${\mathfrak{n},\mathfrak{s}}$, and let $r_n \colon (\widetilde{\sphere{2}},\{\widetilde{\mathfrak{n}},\widetilde{\mathfrak{s}}\})\rightarrow (\sphere{2},\{\mathfrak{n},\mathfrak{s}\})$ be the branched covering map corresponding to the action of $R_{2\pi/n}$. The map $r_n$ extends to a branched cover 
    \[
        r_n \times \operatorname{Id} \colon
        (\widetilde{\sphere{2}}\times \disk{2},\{\widetilde{\mathfrak{n}},\widetilde{\mathfrak{s}}\}\times \disk{2})
        \rightarrow
        (\sphere{2}\times \disk{2}, \{\mathfrak{n},\mathfrak{s}\}).
    \]

    By restricting $f$, we have a cyclic branched cover $f^\circ: (M,L)^\circ\rightarrow (\sphere{3},K)^\circ$, which is equal to $r_n$ on its boundary, and which extends to 

    \[
        f^\circ \times \text{Id} \colon
        (M,L)^\circ \times \sphere{1} \rightarrow (\sphere{3},K)^\circ \times \sphere{1}
    \]

    It remains to check that
    $\tau^k \colon (\sphere{2}\times \sphere{1}, \{\mathfrak{n},\mathfrak{s}\}\times \sphere{1})\rightarrow (\sphere{2}\times \sphere{1}, \{\mathfrak{n},\mathfrak{s}\}\times \sphere{1})$
    lifts to
    $\widetilde{\tau^k} \colon (\widetilde{\sphere{2}}\times \sphere{1}, \{\widetilde{\mathfrak{n}},\widetilde{\mathfrak{s}}\}\times \sphere{1}) \rightarrow (\widetilde{\sphere{2}}\times \sphere{1}, \{\widetilde{\mathfrak{n}},\widetilde{\mathfrak{s}}\}\times \sphere{1})$
    if and only if $n$ divides $k$, in which case $\widetilde{\tau^k} = \tau^{k/n}$.  

    Consider the restriction
    $\tau^k \colon (\sphere{2}-\{\mathfrak{n},\mathfrak{s}\})\times \sphere{1}\rightarrow (\sphere{2}-\{\mathfrak{n},\mathfrak{s}\})\times \sphere{1}$ of $\tau^k$
    to the complement of the branching set.
    Note that
    $(\tau^k)_* \colon \pi_1((\sphere{2}-\{\mathfrak{n},\mathfrak{s}\})\times \sphere{1})\rightarrow \pi_1((\sphere{2}-\{\mathfrak{n},\mathfrak{s}\})\times \sphere{1})$
    is given by the map $(x,y) \mapsto (kx,y)$ on $\mathbb{Z} \times \mathbb{Z} \rightarrow \mathbb{Z} \times \mathbb{Z}$.  

    The map $(r_n \times \operatorname{Id})_* \colon \pi_1( (\widetilde{\sphere{2}}-\{\widetilde{\mathfrak{n}},\widetilde{\mathfrak{s}}\})\times \sphere{1})\rightarrow \pi_1((\sphere{2}-\{\mathfrak{n},\mathfrak{s}\})\times \sphere{1})$ is given by $(x,y)\mapsto (nx,y)$.  As $k\mathbb{Z}\times \mathbb{Z}\leq n\mathbb{Z}\times \mathbb{Z} $ if and only if $n$ divides $k$, the result follows.
\end{proof}

There are many examples of representations $\rho \colon \pi_1(\sphere{3}-K) \rightarrow S_n$ corresponding to non-cyclic branched covers $f \colon (M,L)\rightarrow (\sphere{3},K)$, such that $\rho$, and hence $f$, extends to the spin $\mathcal{S}^k(\sphere{3},K)$.
A non-cyclic branched cover of a spun knot $\mathcal{S}^k(\sphere{3},K)$, while not strictly a spun manifold itself, may still be obtained from $(M,L)^\circ\times \sphere{1}$ by gluing in copies of $(\sphere{2}\times \disk{2}, \{\mathfrak{n},\mathfrak{s}\})$, and it would be interesting to analyze this structure further.

In the following subsections, we compute invariants of both cyclic and non-cyclic branched covers of spun knots.

\subsection{Twist-Spun torus knots}

Using the recipe for twist spun knots from \cite{meier2017bridgeS4}, we can construct the tri-plane diagram for a $k$-twist spin of a $(2, n)$-torus knot $\Twist^{k} T(2, n)$ in \Cref{fig:l_twist_spin_T_2_n}.
Here the number of full-twists $k \in \ZZ$ can be chosen arbitrarily and $n \in \ZZ$ is odd.  The Wirtinger generators of the tri-plane diagram are labelled $\{x_0,\dots, x_7\}$, from left to right as usual.

\begin{figure}
    \centering
    \begin{tikzpicture}
        \node[anchor=south west,inner sep=0] at (0,0){\includegraphics[width=16cm]{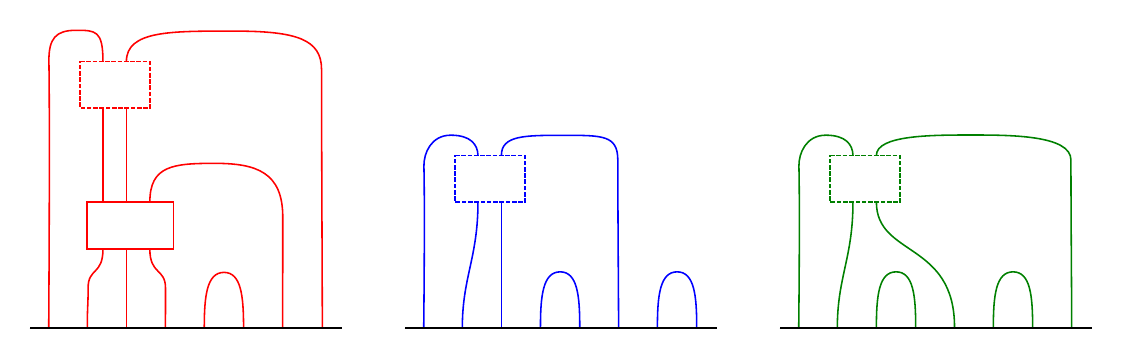}};
        \node at (1.85, 1.9) {\color{red} \scriptsize $k$ full};
        \node at (1.65, 3.9) {\color{red} \scriptsize $n$ half};
        \node at (7.0, 2.55) {\color{blue} \scriptsize $n$ half};
        \node at (12.35, 2.55) {\color{myGreen} \scriptsize $n$ half};
    \end{tikzpicture}
    \caption[Tri-plane diagram of a 4-bridge trisection of a $k$-twist spin of the $(2, n)$-torus knot,  $\Twist^{k} T(2, n)$]{
        Tri-plane diagram of a 4-bridge trisection of an $k$-twist spin of the $(2, n)$ torus knot, $\Twist^{k} T(2, n)$, with $k \in \ZZ$ arbitrary and $n \in \ZZ$ not divisible by $2$.
        The number of twisting $k$ determines the number of full-twists in the red $\alpha$-tangle, and the second parameter $n$ of the torus knot the number of half-twists near the top of the three tangles.  
        \label{fig:l_twist_spin_T_2_n}
    }
\end{figure}
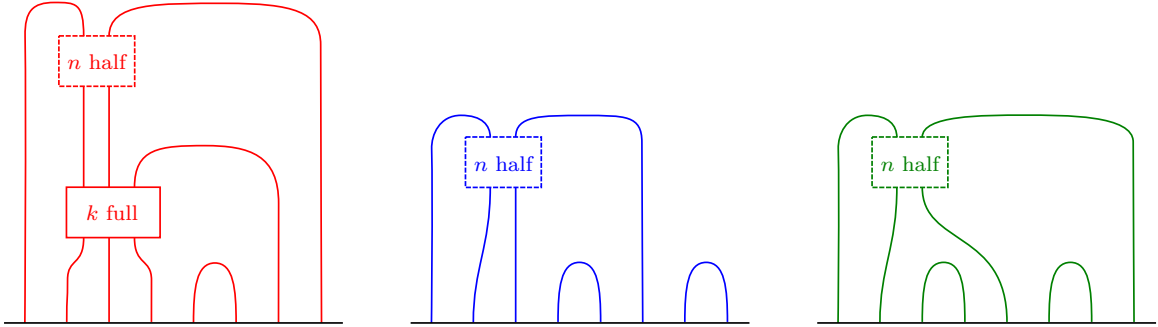

\begin{figure}
    \centering
    \begin{tikzpicture}
    \node[anchor=south west,inner sep=0] at (0,0){\includegraphics[width=16cm]{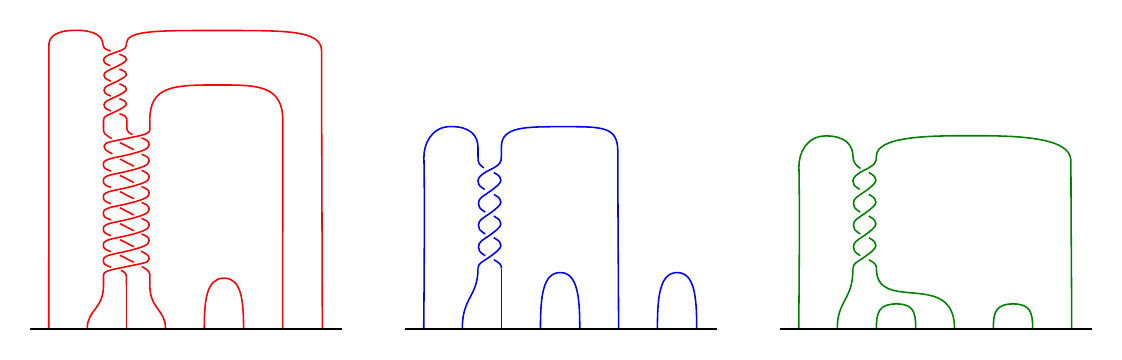}};
    \end{tikzpicture}
    \caption[Tri-plane diagram of a 4-bridge trisection of the $3$-twist spin of the $(2, 5)$ torus knot, $\Twist^{3} T_{2, 5}$]{
        Tri-plane diagram of a 4-bridge trisection of the $3$-twist spin of the $(2, 5)$ torus knot, $\Twist^{3} T_{2, 5}$.
        This is the case $l = 3$, $n = 5$ of \Cref{fig:l_twist_spin_T_2_n}.
        \label{fig:3_twist_spin_T_2_5}
    }
\end{figure}

We now construct dihedral and cyclic covers of these spun knots.  

\subsection{Dihedral covers of twist-spun torus knots.}

The irregular $p$-fold dihedral branched cover of $(\sphere{4},F)$ is the branched cover corresponding to the homomorphism $\rho:\pi_1(\sphere{4}-F)\twoheadrightarrow D_{2p}$ which sends a meridian to a reflection through a vertex of a regular $p$-gon.  When such a homomorphism exists, $F\subset \sphere{4}$ is called (Fox) $p$-colorable.

It follows from \Cref{extendtospin.prop} that the 2-knot $\tau^k(2,p)$ is p-colorable if and only if $k$ is even.   The explicit coloring we use in our computations maps generators $x_0$ and $x_1$ to the reflection $R_1$ through vertex 1 of the regular $p$-gon (with vertices labelled cyclically $1,\dots, p$), and maps all other generators to the reflection $R_2$ through vertex 2.

One can also detect colorability of 2-knots via the determinant. 
For 2-knots, we will use the definition of the determinant from \cite{joseph20190concordance}.
Let $\Delta(\mathcal{K}) \subseteq \ZZ[t^{\pm 1}]$ 
be the Alexander ideal of the surface knot $\mathcal{K}$.
Then the \emph{determinant} of $\mathcal{K}$ is defined to be the integer
$\Delta(\mathcal{K})|_{t = -1}$, the non-negative generator
of the principal ideal $\{ f(1) \mid f(t) \in \Delta(\mathcal{K}) \} \subseteq \ZZ$.

The determinant of a $k$-twist spin
$\tau^{k} K$ of a classical knot $K$
is the same as $\det(K)$ for an even number
of twists $k$, and $1$ if $k$ is odd.
This implies that a nontrivial Fox $p$-coloring
of the classical knot $k$ extends to a non-trivial
$p$-coloring of the $k$-twist spin $\tau^{k} K$
if and only if the amount of twisting $k$ is even. 


\begin{table}
    \begin{center}
        \renewcommand{\arraystretch}{1.15}
        \footnotesize
        \begin{tabular}{c|cc|c|c|ccc}
            \toprule
            $p$&$k$ & Trisection& $\pi_1$& $H_2$ & Intersection & $\sigma$& even/\\
            &&Parameters&&&Form&&odd\\
            \midrule
            \midrule
            $3$&$0$&$(2;0,0,0)$&$1$&$\ZZ[x_3^1,x_5^1-x_6^1]$&$\begin{pmatrix}0&-1\\-1&0\end{pmatrix}$&$0$&even\\
            \hline
            $3$&$2$&$(2;0,0,0)$&$1$&$\ZZ[x_3^1,x_5^1-x_6^1]$&$\begin{pmatrix}0&-1\\-1&-3\end{pmatrix}$&$0$&odd\\
            \hline
            $3$&$4$&$(2;0,0,0)$&$1$&$\ZZ[x_3^1,x_5^1-x_6^1]$&$\begin{pmatrix}0&-1\\-1&-6\end{pmatrix}$&$0$&even\\
            \hline
            $3$&$6$&$(2;0,0,0)$&$1$&$\ZZ[x_3^1,x_5^1-x_6^1]$&$\begin{pmatrix}0&-1\\-1&-9\end{pmatrix}$&$0$&odd\\
            \midrule
            \midrule
            $5$&$0$&$(4;0,0,0)$&$1$&$\ZZ[x_3^1,x_5^1-x_6^1,x_3^4,x_5^4-x_6^4]$&$\begin{pmatrix}0&-1&0&0\\-1&0&0&0\\0&0&0&-1\\0&0&-1&0\end{pmatrix}$&$0$&even\\
            \hline
            $5$&$2$&$(4;0,0,0)$&$1$&$\ZZ[x_3^1,x_5^1-x_6^1,x_3^4,x_5^4-x_6^4]$&$\begin{pmatrix}0&-1&0&0\\-1&-3&0&1\\0&0&0&-1\\0&1&-1&-2\end{pmatrix}$&$0$&odd\\
            \hline
            $5$&$4$&$(4;0,0,0)$&$1$&$\ZZ[x_3^1,x_5^1-x_6^1,x_3^4,x_5^4-x_6^4]$&$\begin{pmatrix}0&-1&0&0\\-1&-6&0&2\\0&0&0&-1\\0&2&-1&-4\end{pmatrix}$&$0$&even\\
            \hline
            $5$&$6$&$(4;0,0,0)$&$1$&$\ZZ[x_3^1,x_5^1-x_6^1,x_3^4,x_5^4-x_6^4]$&$\begin{pmatrix}0&-1&0&0\\-1&-9&0&3\\0&0&0&-1\\0&3&-1&-6\end{pmatrix}$&$0$&odd\\
            \midrule
            \midrule
            $7$&$0$&$(6;0,0,0)$&$1$&$\ZZ[x_3^1,x_5^1-x_6^1,x_3^4,x_5^4-x_6^4,x_3^5,x_5^5-x_6^5]$&$\begin{pmatrix}0&-1&0&0&0&0\\-1&0&0&0&0&0\\0&0&0&-1&0&0\\0&0&-1&0&0&0\\0&0&0&0&0&-1\\0&0&0&0&-1&0\end{pmatrix}$&$0$&even\\
            \hline
            $7$&$2$&$(6;0,0,0)$&$1$&$\ZZ[x_3^1,x_5^1-x_6^1,x_3^4,x_5^4-x_6^4,x_3^5,x_5^5-x_6^5]$&$\begin{pmatrix}0&-1&0&0&0&0\\-1&-3&0&0&0&-1\\0&0&0&-1&0&0\\0&0&-1&-2&0&-1\\0&0&0&0&0&-1\\0&-1&0&-1&-1&-2\end{pmatrix}$&$0$&odd\\
            \hline
            $7$&$4$&$(6;0,0,0)$&$1$&$\ZZ[x_3^1,x_5^1-x_6^1,x_3^4,x_5^4-x_6^4,x_3^5,x_5^5-x_6^5]$&$\begin{pmatrix}0&-1&0&0&0&0\\-1&-6&0&0&0&-2\\0&0&0&-1&0&0\\0&0&-1&-4&0&-2\\0&0&0&0&0&-1\\0&-2&0&-2&-1&-4\end{pmatrix}$&$0$&even\\
            \hline
            $7$&$6$&$(6;0,0,0)$&$1$&$\ZZ[x_3^1,x_5^1-x_6^1,x_3^4,x_5^4-x_6^4,x_3^5,x_5^5-x_6^5]$&$\begin{pmatrix}0&-1&0&0&0&0\\-1&-9&0&0&0&-3\\0&0&0&-1&0&0\\0&0&-1&-6&0&-3\\0&0&0&0&0&-1\\0&-3&0&-3&-1&-6\end{pmatrix}$&$0$&odd\\
            \bottomrule
        \end{tabular}
    \end{center}
    \caption{
        Invariants of the $p$-fold irregular dihedral cover of the $k$-twist spin of the $T(2,p)$ torus knot for even values of $k$.
        \label{spun_torus_dihedral.tab}
    }
\end{table}

As evidenced by \Cref{spun_torus_dihedral.tab}\footnote{Details of this computation can be found in \texttt{spun\_2\_bridge\_torus\_knots\_dihedral\_covers.ipynb} of \cite{Ruppik_Knot_groups_2021}.}, the $p$-fold irregular dihedral cover of $\sphere{4}$ branched along $\tau^{k}(2,p)$ is homeomorphic to, and likely diffeomorphic to, $\cs^{(p-1)/2} \sphere{2}\times \sphere{2}$ if $k$ is congruent to 0 mod 4, and $\cs^{(p-1)/2}(\mathbb{C}P^2 \cs \overline{\mathbb{C}P^2})$ if $k$ is congruent to 2 mod 4.
One can easily verify this by drawing the trisection diagram explicitly following the procedure in \cite{cahn2017singular,cahn2020dihedral} for small values of $p$, and the general case, while more involved, can likely be seen via the same methods. 


\subsection{Cyclic covers of twist-spun torus knots}

Next, we compute invariants of the $n$-fold cyclic branched cover of $\sphere{4}$, branched along the $k$-twist spin of a $T(2,p)$ torus knot.\footnote{Details of this computation can be found in \texttt{spun\_2\_bridge\_torus\_knots\_cyclic\_covers.ipynb} of \cite{Ruppik_Knot_groups_2021}.}
As noted at the beginning of this section, $n$-fold cyclic covers of $k$-twist spun knots are themselves spun manifolds, provided the twisting parameter $k$ is divisible by the degree $n$ of the cover.

We first investigate the cyclic covers of the ordinary (0-twist) spun trefoil.
The results of the computation are shown in \Cref{spun_trefoil_cyclic_covers.tab}, and see the same 6-fold periodicity in the homology groups as in the computation for branched covers of $\sphere{3}$ along the classical trefoil \cite[p.~150]{rolfsen2003knots}.
When $n$ is a multiple of 6, the rank of $H_2$ is nontrivial, and we compute the intersection form as well.

Unlike the 0-twist case, for the 2-twist spin of the trefoil, $\tau^{2} T(2,3)$, $\pi_{1}(\sphere{4} - \tau^{2} T(2,3))$ is not a classical knot group, because its Alexander ideal is not principal. However, we see simpler order 2 periodicity in $\pi_1$, $H_1$ and $H_2$ when we consider the 2-twist spin as opposed to the 0-twist spin, even as we vary $p$. The computation of invariants for the $2$-twist spin of the $T(2,p)$ torus knot, taking cyclic branched covers of degree up to $n=6$, is shown in \Cref{spun_torus_cyclic_covers.tab}.

\begin{table}
    \scriptsize
    \begin{center}
    \begin{tabular}{c|c|ccc|c|c|ccc}
        \toprule
        $p$&$k$& $n$ & Trisection parameters& $|\pi_1$|& $H_1$&$H_2$&Intersection form&$\sigma$&even/odd \\
        \midrule
        \midrule
        $3$&$0$&$2$&$(3;1,1,1)$&3& $\mathbb{Z}_3$&$\mathbb{Z}_3$&$()$&0&even\\
        \midrule
        $3$&$0$&$3$&$(6;2,2,2)$&8 $(Q_8)$& $\mathbb{Z}_2^2$&$\mathbb{Z}_2^2$&$()$&0&even\\
        \midrule
        $3$&$0$&$4$&$(9;3,3,3)$&24 $(SL_2(\mathbb{F}_3))$& $\mathbb{Z}_3$&$\mathbb{Z}_3$&$()$&0&even\\
        \midrule
        $3$&$0$&$5$&$(12;4,4,4)$&120 $(SL_2(\mathbb{F}_5))$ & 1 & 1&$()$ &0&even\\
        \midrule
        $3$&$0$&$6$&$(15;5,5,5)$&$\infty$ & $\mathbb{Z}^2$&$\mathbb{Z}^4$&$\begin{pmatrix}-10&9&1&0\\9&-2&0&-1\\1&0&0&0\\0&-1&0&0\end{pmatrix}$&0&even\\
        \midrule
        $3$&$0$&$7$&$(18;6,6,6)$&& 1&1&$()$&0&even\\
        \midrule
        $3$&$0$&$8$&$(21;7,7,7)$&& $\mathbb{Z}_3$&$\mathbb{Z}_3$&$()$&0&even\\
        \midrule
        $3$&$0$&$9$&$(24;8,8,8)$&& $\mathbb{Z}_2^2$&$\mathbb{Z}_2^2$&$()$&0&even\\
        \midrule
        $3$&$0$&$10$&$(27;9,9,9)$&& $\mathbb{Z}_3$&$\mathbb{Z}_3$&$()$&0&even\\
        \midrule
        $3$&$0$&$11$&$(30;10,10,10)$&& 1&1&$()$&0&even\\
        \midrule
        $3$&$0$&$12$&$(33;11,11,11)$&& $\mathbb{Z}^2$&$\mathbb{Z}^4$&$\begin{pmatrix}-10&5&14&0\\5&-6&-11&-1\\14&-11&-24&-1\\0&-1&-1&0\end{pmatrix}$&0&even\\
        \midrule
        $3$&$0$&$13$&$(36;12,12,12)$&& 1&1&$()$&0&even\\
        \midrule
        $3$&$0$&$14$&$(39;13,13,13)$&& $\mathbb{Z}_3$&$\mathbb{Z}_3$&$()$&0&even\\
        \bottomrule
    \end{tabular}
    \end{center}
    
    \vspace{0.5mm}
    
    \begin{flushleft}Generators of $H_2$ for intersection form computations:
    
    \vspace{0.5mm}

    \begin{center}
    \begin{tabular}{c|l}
         \toprule
         $n$ & $H_2$ generators  \\
         \midrule
         \midrule
         $6$& (1) $x_2^1+x_2^2+x_2^3+2x_3^3+x_4^3+5x_7^3+4x_0^4+3x_2^4+4x_4^4+x_7^4+x_2^5-x_3^5+2x_4^5$, \\
         & (2) $x_3^1+x_2^2-6x_3^3-2x_4^3-10x_7^3-8x_0^4-6x_2^4-8x_4^4+x_7^4+3x_0^5-x_4^5$, \\
         & (3) $x_5^2+x_6^3-x_3^4-x_7^4-x_0^5-x_2^5-x_4^5$,\\
         & (4) $x_5^1+x_6^2$ \\
         \midrule
         $12$
         & (1) $x_2^1+x_2^2+x_2^3+x_4^3+x_2^4+2x_{4}^{4}+x_2^5+2x_{4}^{5}-x_7^5+2x_{0}^{6}+x_2^6+x_4^6-2x_{7}^{6}+2x_{0}^{7}+x_2^7+4x_{0}^{8}+x_2^8$\\
         & $\phantom{(1)} +3x_{0}^{9}+x_2^9+x_4^9+2x_{7}^{9}+2x_{0}^{10}+x_2^{10}+2x_{4}^{10}+4x_{7}^{10}+3x_{0}^{11}+x_2^{11}-x_3^{11}+2x_{4}^{11}$\\
         &(2) $x_{3}^{1}+x_{2}^{2}-2x_{4}^{3}-2x_{4}^{4}-x_{7}^{4}-x_{4}^{5}+x_{7}^{5}-3x_{0}^{6}+x_{7}^{6}-3x_{0}^{7}-3x_{0}^{8}-x_{4}^{8}-2x_{0}^{9}-2x_{4}^{9}-6x_{7}^{9}-8x_{0}^{10}$\\
         & $\phantom{(2)} -2x_{4}^{10}+3x_{7}^{10}-x_{0}^{11}-x_{4}^{11}$\\
         &(3) $x_{4}^{1}-2x_{4}^{3}-3x_{4}^{4}+x_{7}^{4}-2x_{4}^{5}+3x_{7}^{5}-3x_{0}^{6}+4x_{7}^{6}-3x_{0}^{7}+x_{4}^{7}-6x_{0}^{8}-4x_{0}^{9}-2x_{4}^{9}-5x_{7}^{9}-5x_{0}^{10}$\\
         & $\phantom{(3)} -3x_{4}^{10}-4x_{7}^{10}-4x_{0}^{11}-2x_{4}^{11}$\\
         &(4) $x_{5}^{1}+x_{6}^{2}$\\
         \bottomrule
    \end{tabular}
    \end{center}
    
    \end{flushleft}
    \caption{Invariants of the $n$-fold cyclic cover of the 0-spun trefoil $T(2,3)$.\label{spun_trefoil_cyclic_covers.tab}}
\end{table}

\begin{table}
    \begin{center}
        \renewcommand{\arraystretch}{1.2}
        \begin{tabular}{c|c|cc|c|c|c}
            \toprule
            $p$ & $k$ & $n$ & Trisection parameters& $\pi_1$& $H_1$&$H_2$ \\
            \midrule
            \midrule
            $3$ & $2$ & $2$ & $(3;1,1,1)$ & $\langle x_2 \mid x_2^3\rangle$& $\mathbb{Z}_3$ & $\mathbb{Z}_3$\\
            \hline
            $3$&$2$&$3$&$(6;2,2,2)$&1& 1&1\\
            \hline
            $3$&$2$&$4$&$(9;3,3,3)$& $\langle x_2 \mid x_2^3\rangle$ & $\mathbb{Z}_3$&$\mathbb{Z}_3$\\
            \hline
            $3$&$2$&$5$&$(12;4,4,4)$&1& 1&1\\
            \hline
            $3$&$2$&$6$&$(15;5,5,5)$& $\langle x_2 \mid x_2^3\rangle$ & $\mathbb{Z}_3$&$\mathbb{Z}_3$\\
            \midrule
            \midrule
            $5$&$2$&$2$&$(3;1,1,1)$& $\langle x_2 \mid x_2^5\rangle$ & $\mathbb{Z}_5$&$\mathbb{Z}_5$\\
            \hline
            $5$&$2$&$3$&$(6;2,2,2)$&1& 1&1\\
            \hline
            $5$&$2$&$4$&$(9;3,3,3)$& $\langle x_2 \mid x_2^5\rangle$ & $\mathbb{Z}_5$&$\mathbb{Z}_5$\\
            \hline
            $5$&$2$&$5$&$(12;4,4,4)$&1& 1&1\\
            \hline
            $5$&$2$&$6$&$(15;5,5,5)$& $\langle x_2 \mid x_2^5\rangle$ & $\mathbb{Z}_5$&$\mathbb{Z}_5$\\
            \midrule
            \midrule
            $7$&$2$&$2$&$(3;1,1,1)$& $\langle x_2 \mid x_2^7\rangle$ & $\mathbb{Z}_7$&$\mathbb{Z}_7$\\
            \hline
            $7$&$2$&$3$&$(6;2,2,2)$&1& 1&1\\
            \hline
            $7$&$2$&$4$&$(9;3,3,3)$& $\langle x_2 \mid x_2^7\rangle$ & $\mathbb{Z}_7$&$\mathbb{Z}_7$\\
            \hline
            $7$&$2$&$5$&$(12;4,4,4)$&1& 1&1\\
            \hline
            $7$&$2$&$6$&$(15;5,5,5)$& $\langle x_2 \mid x_2^7\rangle$ & $\mathbb{Z}_7$&$\mathbb{Z}_7$\\
            \bottomrule
        \end{tabular}
    \end{center}
    \caption{
        Invariants of the $n$-fold cyclic cover of the $2$-twist spin of the $T(2,p)$ torus knot.
        \label{spun_torus_cyclic_covers.tab}
    }
\end{table}

\section{An infinite tower of simply-connected irregular covers of the Stevedore disk double}
\label{stevedore.sec}
	
Let $\Delta \subset \disk{4}$ be the standard 2 minima ribbon disk for the Stevedore knot $6_1$. The double $K = \Delta \cup_{\partial} \Delta  \subset \disk{4} \cup_{\partial} \disk{4} \cong \sphere{4}$ is a ribbon 2-knot with group $\pi_{1}(\sphere{4} - K) \cong \langle a, b \mid bab^{-1} = a^{2} \rangle \cong \ZZ \rtimes \ZZ[\frac{1}{2}]$, the Baumslag-Solitar group $BS(1,2)$.
Note that $\pi_{1}(\sphere{4} - K)$ is not a classical knot group, because its Alexander polynomial $\Delta_{K}(t) = 2t - 1$ is not symmetric under $t \mapsto t^{-1}$.

A tri-plane diagram of the disk double $K$ is given in \Cref{fig:stevedore_triplane}, together with a choice of Wirtinger generators. The red tangle, together with the black banded tangle, form a banded unlink diagram for $K$, and the green and blue tangles are obtained from the black tangle by resolving the bands.  After simplifying the Writinger relations, we get $\pi_1(\mathbb{S}^4-K,P)=\langle x_0,x_2 \mid x_0x_2x_0^{-1}x_2^{-1}x_0x_2^{-1}\rangle$.
An isomorphism with the presentation above is given by $a\mapsto x_2x_0^{-1}$ and $b\mapsto x_2$.

The Baumslag-Solitar groups $BS(1,m)$ are non-co-Hopfian, i.e., each such group contains a proper subgroup (and hence infinitely many) isomorphic to itself.  We use this to construct an infinite tower of trisected irregular branched covers of the Stevedore disk double, the first of which is determined by the standard 3-coloring of the disk.

For each odd $n$ we construct a map $\rho_n \colon BS(1,2)\rightarrow S_n$, with image isomorphic to a semi-direct product $\ZZ_n\rtimes\ZZ_m$, where $m$ is the multiplicative order of 2 mod n, as follows.
Let $\sigma=(1\;2\;\dots\;n)\in S_n$.
Since $n$ is odd, $\sigma^2$ is also an $n$-cycle.
Because $\sigma$ and $\sigma^2$ have the same cycle type, there exists a unique $\tau$ such that $\tau(1)=1$ and $\tau ^{-1} \circ \sigma\circ \tau=\sigma^2$.
(Note that we compose permutations right to left, so we left-compose by $\tau^{-1}$ and right-compose by $\tau$, in contrast with the group relation.)
Then we define $\rho_n$ on the generators of $BS(1,2)$ by $a\mapsto \sigma$ and $b\mapsto \tau$.  We work out $\rho_n$ explicitly for odd $n\leq 17$ in \Cref{stevedore_hom.tab}.

Let $X_K$ denote the $n$-fold unbranched cover of ${\mathbb{S}^4-K}$ corresponding to $\rho_n$.
Let $f_n \colon (X_K,P_1) \rightarrow (\mathbb{S}^4-K,P) $ be the corresponding covering map.
Then $\pi_1(X_K,P_1)\cong\operatorname{Im }(f_n)_*$.
Since $\rho(a^n)(1)=\operatorname{Id}(1)=1$ and $\rho(b)(1)=1$, both $a^n$ and $b$ lift to elements of $\pi_1(X_K,P_1)$, call them $A$ and $B$. 

We now show that $\pi_1(X_K,P_1)\cong \langle A,B \mid BAB^{-1} = A^2 \rangle \cong BS(1,2)$, with $(f_n)_*(A)=a^n$ and $(f_n)_*(B)=b$, using the Reidemeister-Schreier algorithm.
This implies that $BS(1,2)$ is non-co-Hopfian, and also implies that $BS(1,2)$ arises as the group of the preimage of the branch locus in each corresponding $n$-fold unbranched cover $X_K'$.
We then show that the corresponding unbranched covers $X_K'$ are simply-connected. 

\begin{proposition}
    The $n$-fold branched cover $X_K'$ of the Stevedore disk double corresponding to $\rho_n$ is simply-connected.
    Furthermore, the group of the corresponding unbranched cover is isomorphic to $BS(1,2)$.
\end{proposition}

\begin{proof}
    We first use the Reidemeister-Schreier algorithm to compute $\pi_1(X_K,P_1)$.
    d
    We first construct $\rho_n$ explicitly as follows.
    First, $\rho_n(a)=\sigma=(1 ,2 ,\dots, n)$, so $\sigma^2=(1, 3, \dots , n, 2 , 4, \dots , n-1)$.
    Next we construct $\rho_n(b)=\tau$ such that $\tau(1)=1$ and $\tau ^{-1} \circ \sigma\circ \tau=\sigma^2$.
    To do this, assemble a permutation in 2-row notation by placing $\sigma$ above $\sigma^2$.
    The resulting permutation below, with $k=(n+3)/2$, is $\tau^{-1}$:
    \[
        \tau^{-1}
        =
        \begin{pmatrix}
            1 &2 &3 & \dots &k-1  &k &k+1& k+2&\dots &n \\
            1 &3 &5 & \dots &n    &2 &4  & 6  &\dots &n-1 
        \end{pmatrix}.
    \]
    
    A priori, the generators of $\pi_1(X_K,P_1)$ are $a^1$,$a^2$, ... , $a^n$, $b^1$, $b^2$, ... , $b^n$.
    Since $\rho(a)=\sigma=(1,2,\dots,n)$, the words $\{a^1, a^1a^2, a^1a^2a^3,\dots, a^1a^2\dots a^{n-1}\}$ satisfy the Schreier condition.
    It follows from the corresponding ``claw'' relations that
    $$a^1=a^2=\dots=a^{n-1}=1.$$
    
    Set $A=a^n$.
    
    We next write down the $n$ lifts $r_1, \dots, r_n$ of the relation $r=bab^{-1}a^{-1}a^{-1}=1.$ Note that superscripts of $a$ and $b$ below denote generator labels, not exponents.  We obtain:
    \begin{align}
    1=r_1&=b^1a^1(b^3)^{-1}(a^2)^{-1}(a^1)^{-1}& \label{bs_reln1.eq}\\
    1=r_i&=b^ia^{k+s}(b^{i+2})^{-1}(a^{i+1})^{-1}(a^i)^{-1} &\text{ for }2\leq i \leq n-3, \text{ }i \text{ even, }s=i/2-1\label{bs_reln2.eq}\\
    1=r_i&=b^ia^{i-1}(b^{i+2})^{-1}(a^{i+1})^{-1}(a^i)^{-1} &\text{ for } 3\leq i\leq n-2, \text{ }i  \text{ odd}\label{bs_reln3.eq}\\
    1=r_{n-1}&=b^{n-1}a^n (b^1)^{-1}(a^n)^{-1}(a^{n-1})^{-1}&\label{bs_reln4.eq}\\
    1=r_n&=b^na^{k-1}(b^2)^{-1}(a^1)^{-1}(a^7)^{-1}&\label{bs_reln5.eq}
    \end{align}
    
    From the relations in \Cref{bs_reln1.eq} and \Cref{bs_reln3.eq}, and applying the relation $a^i=1$ for $1\leq i\leq n-1$, we have:
    
    \[
        b^1=b^3=\dots=b^{n}.
    \]
    
    From the relation \Cref{bs_reln2.eq}, and applying the relation $a^i=1$ for $1\leq i\leq n-1$, we have:
    \[
        b^2 = b^4 = \dots = b^{n-1}.
    \]
    
    Set $B=b^2$ and $C=b^1$.
    
    Similarly, from relation \ref{bs_reln4.eq}, we have:
    $$1=BAC^{-1}A^{-1},$$
    and from relation \ref{bs_reln5.eq}, we have:
    $$1=CB^{-1}A^{-1}.$$
    
    Combining the last two relations gives $BAB^{-1}=A^2$.  Therefore $\pi_1(X_K,P_1) = \langle A, B \mid BAB^{-1} = A^2 \rangle \cong BS(1,2)$. 
    
    Next, we show that the corresponding branched cover is simply connected.
    For each $k$ cycle of $\rho_n(x_i)$, where $x_i$ is a Wirtinger generator shown in \Cref{fig:stevedore_triplane}, we lift the relation $(x_i)^k=1$.
    In this case, the two 1-cycles corresponding to the fixed points $\rho_n(x_0)(n)=\rho_n(a^{-1}x_2)(n)= \tau\circ \sigma^{-1}(n)=n$ and $\rho_n(x_2)(1)=\tau(1)=1$ are sufficient.
    The lift of the corresponding relations $(x_0)^1=1$ and $(x_2)^1=1$ are $(x_0^n)^1=1$ and $(x_2^1)^1=1$ respectively.
    Note $x_0^n=(a^{n-1})^{-1}x_2^{n-1}=b^2=B$ by relation \ref{bs_reln2.eq} and the claw relation; in addition, $x_2^1=b^1=C$.  Since $B=C=1$ and $AB=C$, we have $A=B=1$.
    Therefore $\pi_1(X_K')=1.$
\end{proof}

We compute $\rho_n$ for odd $p \leq 17$, and use our algorithm to compute the fundamental group, trisection parameters, homology groups, and intersection form in each case. This is summarized in \Cref{stevedore.tab}.\footnote{Details of this computation can be found in \texttt{stevedore\_group\_trisection\_other\_homs-intersection\_form.ipynb} of \cite{Ruppik_Knot_groups_2021}.}
We can again identify the covers as sums of $\mathbb{C}P^2 \cs \overline{\mathbb{C}P}^2$ or $\sphere{2} \times \sphere{2}$ up to homeomorphism, for $n\leq 17$, and wonder whether the covers corresponding to $\rho_n$ are diffeomorphic to sums of $\mathbb{C}P^2 \cs \overline{\mathbb{C}P}^2$ or $\sphere{2}\times \sphere{2}$ in general.

\begin{figure}[htbp]
	\centering
	\begin{tikzpicture}
		\node[anchor=south west,inner sep=0] at (0,0){\includegraphics[width=12cm]{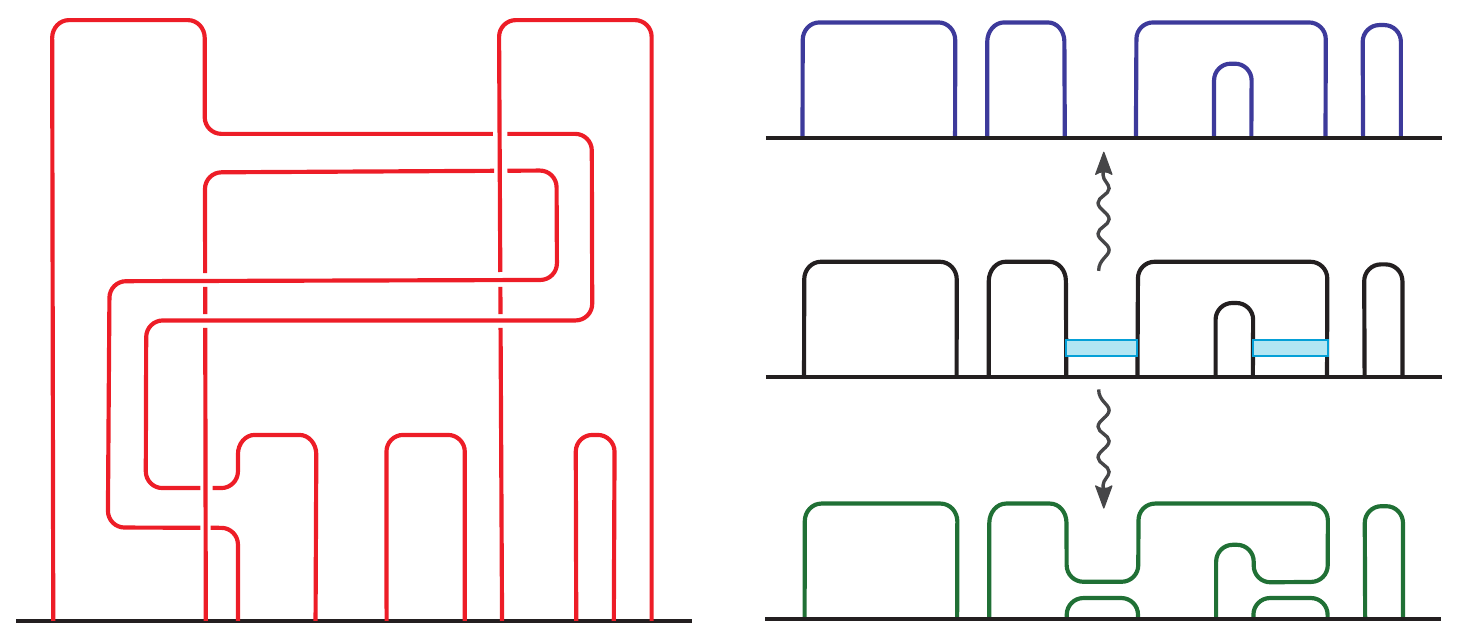}};
	\end{tikzpicture}
	\caption{
		Tri-plane diagram of the Stevedore disk double. Wirtinger generators are labelled $\{x_0,\dots,x_9\}$ from left to right along the bridge sphere.
		\label{fig:stevedore_triplane}
	}	
\end{figure}

\begin{table}[htbp]
    \begin{center}
        \renewcommand{\arraystretch}{1.1}
        \begin{tabular}{c|c|c|c}
            \toprule
            $n$& $\rho_n(a)$&$\rho_n(b)=\rho_n(x_2)$&$\rho_n(x_0)$\\
            \midrule
            \midrule
            $3$& $(1,2,3)$&$(2,3)$&$(1,2)$\\
            \hline
            $5$&$(1,2,3,4,5)$&$(2,4,5,3)$&$(1,3,4,2)$\\
            \hline
            $7$&$(1,2,3,4,5,6,7)$&$(2,5,3)(4,6,7)$&$(1,4,2)(3,5,6)$\\
            \hline
            $9$&$(1,2,3,4,5,6,7,8,9)$&$(2,6,8,9,5,3)(4,7)$&$(1,5,7,8,4,2)(3,6)$\\
            \hline
            $11$&$(1,2,3,4,5,6,7,8,9,10,11)$&$(2,7,4,8,10,11,6,9,5,3)$&$(1,6,3,7,9,10,5,8,4,2)$\\ 
            \hline
            $13$ 
            & $(1,2,3,4,5,6,7,$
            & $(2,8,11,6,10,12,13,7,4,9,5,3)$
            & $(1,7,10,5,9,11,12,$
            \\
            & $8,9,10,11,12,13)$ 
            & 
            & $6,3,8,4,2)$
            \\
            \hline
            $15$&$(1,2,3,4,5,6,7,8,9,$&$(2,9,5,3)(4,10,13,7)$&$(1,8,4,2)(3,9,12,6)$\\
            &$10,11,12,13,14,15)$&$(6,11)(8,12,14,15)$&$(5,10)(7,11,13,14)$\\
            \hline
            $17$&$(1,2,3,4,5,6,7,8,9,10,11,$&$(2,10,14,16,17,9,5,3)$&$(1,9,13,15,16,8,4,2)$\\
            &$12,13,14,15,16,17)$&$(4,11,6,12,15,8,13,7)$&$(3,10,5,11,14,7,12,6)$\\
            \bottomrule
        \end{tabular}
    \end{center}
    \caption{
        Images of $a, b, x_0$ and $x_2$ under $\rho_n$, where $\rho_n \colon BS(1,2)\rightarrow S_n$ is the homomorphism defined in \Cref{stevedore.sec}. 
        \label{stevedore_hom.tab}
    }
\end{table}

\begin{table}[htbp]
    \begin{center}
        \renewcommand{\arraystretch}{1.3}
        \begin{tabular}{c|c|c|c|ccc}
            \toprule
            $n$ & Trisection  &$\pi_1$  & $\text{rk}(H_2)$  &Intersection  &Parity & $\sigma$ \\
            &Parameters& &&Form&& \\
            \midrule
            \midrule
            $3$&$(3; 0, 0, 1)$&1&2&$\begin{pmatrix}1&-2\\-2&3\end{pmatrix}$&odd&0\\
            \hline
            $5$&$(11;2,2,5)$&1&2&$\begin{pmatrix}15&-1\\-1&0\end{pmatrix}$&odd&0\\
            \hline
            $7$&$(14;2,2,6)$&1&4&$\begin{pmatrix}4&5&1&-6\\5&6&1&-6\\ 1&1&0&-1\\-6&-6&-1&10\end{pmatrix}$&even&0\\
            \hline
            $9$&$(22;4,4,10)$&1&4&$\begin{pmatrix} 0&-1 & 0&  1\\ -1& 9&  0& -18\\ 0 &  0  &0& -1\\ 1& -18& -1&43\end{pmatrix}$&odd&0\\
            \hline
            $11$&$(35;8,8,17)$&1&2&$\begin{pmatrix}0&1\\1&187\end{pmatrix}$&odd&0\\
            \hline
            $13$&$(43;10,10,21)$&1&2&$\begin{pmatrix}0&1\\1&355\end{pmatrix}$&odd&0\\
            \hline
            $15$&$(36;6,6,16)$&1&8&$\begin{pmatrix}
            7 & -7 &  1 & -1 & 15 & -3 &  3 &  0\\
            -7 & 17 & -1 &  0& -23&   7&  -7 &  0\\
            1&  -1 &  0&  0 &  2 & -1 &  1 &  0\\
            -1 &  0 &  0 &  0&  -1 &  0&   0 &  0\\
            15& -23&   2 & -1 & 38& -10&  10&   0\\
            -3 &  7 & -1 &  0& -10&   5 & -4 & -2\\
            3&  -7 &  1 &  0 & 10 & -4 &  9 & -5\\
            0 &  0 &  0&   0&   0&  -2 & -5&   8
            \end{pmatrix}$&odd&0\\
            \hline
            $17$&$(54;12,12,26)$&1&4&$\begin{pmatrix}
            0&  0& 1 & 0\\
            0 & 0 & 2 & 1\\
            1 & 2 &67& 26\\
            0 & 1 &26 &30
            \end{pmatrix}$&odd&0\\
            \bottomrule
        \end{tabular}
    \end{center}
    \caption{
        Invariants of $n$-fold irregular covers of the Stevedore disk double for odd $n\leq 17$.
        Generators of $H_2$ corresponding to the intersection forms above can be found in \Cref{stevedore_H2.tab}.
        \label{stevedore.tab}
    }
\end{table}

\begin{table}[htbp]
    \begin{center}
        \scriptsize
        \renewcommand{\arraystretch}{1.3}
        \begin{tabular}{c|l}
            \toprule
            $n$&  $H_2$ generators \\
            \midrule
            \midrule
            $3$&(1) $x_{3}^{2}-x_{4}^{2}$\\
            &(2) $x_{5}^{2}$\\
            \hline
            $5$ & (1) $x_{2}^{2}-3x_{5}^{2}+x_{4}^{3}-3x_{5}^{3}+3x_{2}^{4}-x_{3}^{4}+x_{4}^{4}-2x_{5}^{4}$\\
            &(2) $x_{3}^{2}+x_{5}^{2}+x_{5}^{3}-x_{2}^{4}+x_{3}^{4}-x_{4}^{4}+x_{5}^{4}$\\
            \hline
            $7$ & (1) $x_{4}^{2}-x_{3}^{3}+x_{4}^{4}-x_{3}^{6}+2x_{4}^{6}$\\
            & (2) $x_{5}^{2}+x_{5}^{3}+x_{4}^{4}-x_{3}^{6}+2x_{4}^{6}$\\
            & (3) $x_{6}^{2}+x_{6}^{3}-x_{7}^{3}$\\
            & (4) $x_{2}^{3}-x_{5}^{3}-x_{4}^{4}+2x_{3}^{6}-3x_{4}^{6}$\\
            \hline
            $9$ & (1) $x_{4}^{2}+x_{3}^{6}+x_{5}^{6}$\\
            & (2) $x_{5}^{2}-x_{2}^{5}+x_{5}^{5}+x_{5}^{8}$\\
            & (3) $-x_{6}^{4}+x_{7}^{4}$\\
            & (4) $x_{2}^{3}+2x_{5}^{3}-x_{3}^{4}+x_{4}^{4}+3x_{2}^{5}-x_{5}^{5}+x_{5}^{6}-2x_{5}^{8}$\\
            \hline
            $11$ 
            & (1) $x_{4}^{2}+x_{3}^{7}+x_{5}^{7}+2x_{6}^{8}+2x_{7}^{10}$\\
            & (2) $4x_{4}^{3}-8x_{5}^{3}-2x_{4}^{4}-7x_{5}^{4}-8x_{2}^{5}-5x_{4}^{5}-6x_{5}^{5}-3x_{4}^{6}-7x_{2}^{7}+5x_{3}^{7}-5x_{4}^{7}+x_{5}^{7}$ \\
            & \phantom{(2)} $+3x_{3}^{8}-4x_{4}^{8}+2x_{3}^{9}-5x_{4}^{9}+x_{3}^{10}-5x_{4}^{10}$\\
            \hline
            $13$ & (1) $x_{6}^{8}-2x_{6}^{10}+x_{7}^{11}-2x_{7}^{12}$\\
            & (2) $5x_{4}^{3}-10x_{5}^{3}-2x_{4}^{4}-9x_{5}^{4}-10x_{2}^{5}-6x_{4}^{5}-8x_{5}^{5}-3x_{4}^{6}-7x_{5}^{6}-9x_{2}^{7}-6x_{4}^{7}+6x_{3}^{8}$\\
            & \phantom{(2)} $+x_{5}^{8}-8x_{2}^{9}+4x_{3}^{9}-6x_{4}^{9}+3x_{3}^{10}-5x_{4}^{10}+6x_{3}^{11}-6x_{4}^{11}+4x_{5}^{11}+x_{3}^{12}-6x_{4}^{12}$\\
            \hline
            $15$ 
            & (1) $x_{4}^{2}-x_{3}^{5}+x_{4}^{5}-x_{5}^{6}+x_{4}^{8}-x_{3}^{14}+x_{4}^{14}$\\
            & (2) $x_{5}^{2}+x_{5}^{3}-x_{5}^{5}+x_{5}^{6}-x_{4}^{8}-x_{5}^{10}+x_{3}^{14}-x_{4}^{14}$\\
            & (3) $x_{6}^{2}+x_{6}^{3}+x_{6}^{5}-x_{7}^{5}$\\
            & (4) $x_{7}^{2}+x_{6}^{3}$\\
            & (5) $x_{3}^{3}+x_{4}^{5}+2x_{5}^{5}-2x_{5}^{6}+2x_{4}^{8}+x_{5}^{10}-2x_{3}^{14}+2x_{4}^{14}$\\
            & (6) $x_{3}^{4}-x_{5}^{5}+x_{4}^{7}-x_{4}^{8}-x_{5}^{10}-x_{3}^{12}$\\
            & (7) $x_{5}^{4}+x_{5}^{5}+x_{5}^{7}+x_{4}^{8}-x_{2}^{10}+x_{5}^{10}-x_{3}^{12}-2x_{3}^{14}+2x_{4}^{14}-x_{5}^{14}$\\
            & (8) $x_{2}^{7}-x_{5}^{7}+x_{2}^{10}+2x_{3}^{12}+2x_{3}^{14}-2x_{4}^{14}+x_{5}^{14}$\\
            \hline
            $17$ 
            & (1) $x_{6}^{14}+x_{7}^{16}$\\
            & (2) $x_{5}^{2}-x_{2}^{5}-x_{4}^{5}-x_{2}^{9}-x_{4}^{9}+x_{5}^{9}+x_{3}^{10}+x_{5}^{10}+x_{6}^{10}-x_{2}^{14}+x_{3}^{14}-x_{4}^{14}$ \\
            & \phantom{(2)} $+x_{5}^{14}-x_{6}^{14}+x_{7}^{14}-x_{2}^{16}-x_{4}^{16}-x_{7}^{16}$ \\
            & (3) $x_{2}^{3}+4x_{4}^{4}+x_{2}^{5}+3x_{4}^{5}-x_{5}^{5}+2x_{4}^{6}-3x_{5}^{6}+3x_{4}^{7}-x_{5}^{7}-x_{3}^{8}+3x_{4}^{9}-x_{5}^{9}-3x_{3}^{10}-4x_{5}^{10}+3x_{3}^{11}$ \\
            & \phantom{(3)} $+x_{4}^{11}+x_{3}^{12}+x_{4}^{12}-x_{2}^{13}-x_{3}^{13}+x_{4}^{13}+2x_{2}^{14}-3x_{3}^{14}+3x_{4}^{14}-x_{5}^{14}+x_{2}^{16}+5x_{4}^{16}+x_{5}^{16}$\\
            & (4) $x_{4}^{3}+x_{4}^{4}+2x_{4}^{5}-3x_{5}^{6}+2x_{4}^{7}-2x_{5}^{7}+x_{3}^{8}-x_{4}^{8}+2x_{4}^{9}-2x_{3}^{10}-x_{5}^{10}+2x_{3}^{11}-x_{4}^{11}+x_{3}^{12}$ \\
            & \phantom{(4)} $-x_{4}^{12}-2x_{2}^{13}-x_{4}^{13}+3x_{2}^{14}-2x_{3}^{14}+2x_{4}^{14}+x_{5}^{14}+3x_{5}^{16}$\\
            \bottomrule
        \end{tabular}
    \end{center}
    \caption{
        Generators of $H_2$ of the branched cover of the Stevedore disk double corresponding to intersection form computations in \Cref{stevedore.tab} for odd $n\leq 17$.
        \label{stevedore_H2.tab}
    }
\end{table}

\section{Suciu's knots}
\label{sec:suciu}

Suciu \cite{suciu1985ribbon} constructed an infinite family of ribbon 2-knots $R_{k} \colon \sphere{2} \hookrightarrow \sphere{4}$ in the 4-sphere (indexed by the natural numbers $k \in \mathbb{N}_{\ge 1}$), each with $\pi_{1}(\sphere{4} - R_{k})$ isomorphic to the group of the trefoil knot $T_{2, 3} \colon \sphere{1} \hookrightarrow \sphere{3}$, and which can be distinguished by the action of $\pi_1$ on $\pi_2(\sphere{4} - R_{k})$. 

The fibration of the trefoil complement by punctured tori induces the following presentation of $\pi_{1}(\sphere{3} - T_{2, 3})$ as an HNN-extension of the free group $\langle u, v \rangle$ of rank 2 of the fiber by the cyclic group $\langle t \rangle$ whose generator is a meridian to the knot.
\[
    \langle
    t, u, v \mid 
    t u t^{-1} = v,
    t v t^{-1} = u^{-1} v
    \rangle
\]
In particular, $R_{1}$ is the spun trefoil knot with meridian $t$.
The other knots appear via modifications of the monodromy, and they have different meridians $t_{k} = u^{k} t$.

Kanenobu and Sumi later showed that the $R_k$ can also be distinguished by the fundamental groups of their double branched covers \cite{kanenobu2020suciu}.
This motivates the question of which other branched covers distinguish the Suciu knots.  

A $7$-bridge trisection of each $R_k$, derived from Suciu's ribbon presentation \cite[Fig.~6]{suciu1985ribbon}, is given in \cite{ruppik2022thesis}.
A bridge-position banded unlink diagram for $R_k$ is shown in \Cref{fig:Suciu_R_3_banded_bridge_position}.
In our figures, we illustrate the banded bridge splitting and resulting bridge trisection for $k=3$; increasing this parameter $k \in \mathbb{N}$ amounts to adding additional linking of the orange fusion band with the minima $c$ and $d$.
The three tri-plane diagrams resulting from the banded bridge splitting are pictured in \Cref{fig:Suciu_R_3_bridge_trisection}.

\begin{figure}
	\centering
	\begin{tikzpicture}
		\node[anchor=south west,inner sep=0] at (0,0){\includegraphics[width=12cm]{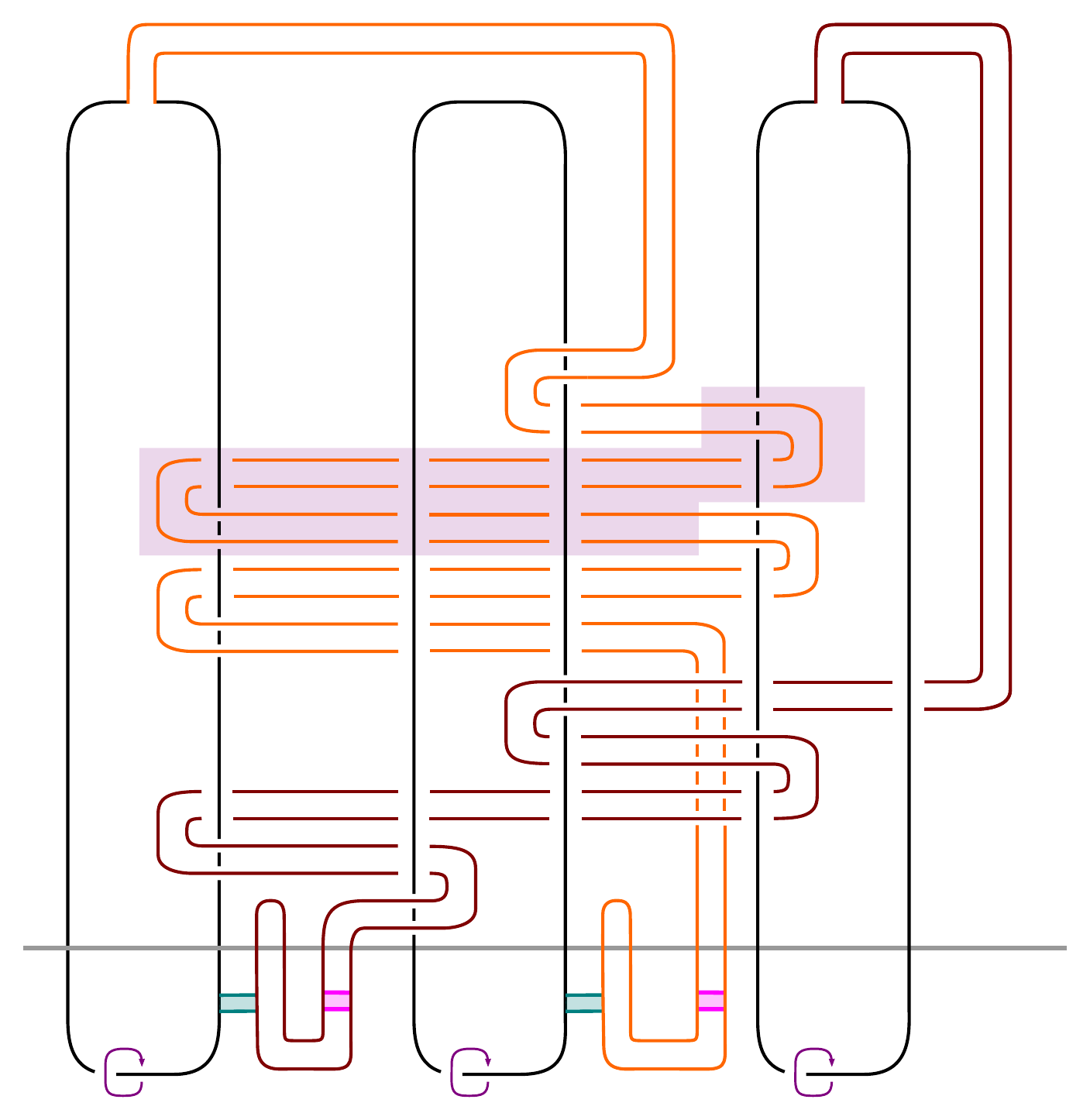}};
		\node at (1.0, 0.1) {\color{myPatriarch} $d$};
		\node at (4.8, 0.1) {\color{myPatriarch} $t_{k}$};
		\node at (8.6, 0.1) {\color{myPatriarch} $c$};
	\end{tikzpicture}
	\caption[Banded bridge position of Suciu's knot $R_{3}$]{
		A banded bridge splitting of Suciu's 2-knot $R_{k} \colon \sphere{2} \hookrightarrow \sphere{4}$ for $k = 3$ \cite{suciu1985ribbon}.
		The lightly shaded box indicates which part of the orange fusion band needs to be duplicated to increase the parameter $k$.
		\label{fig:Suciu_R_3_banded_bridge_position}
	}	
\end{figure}
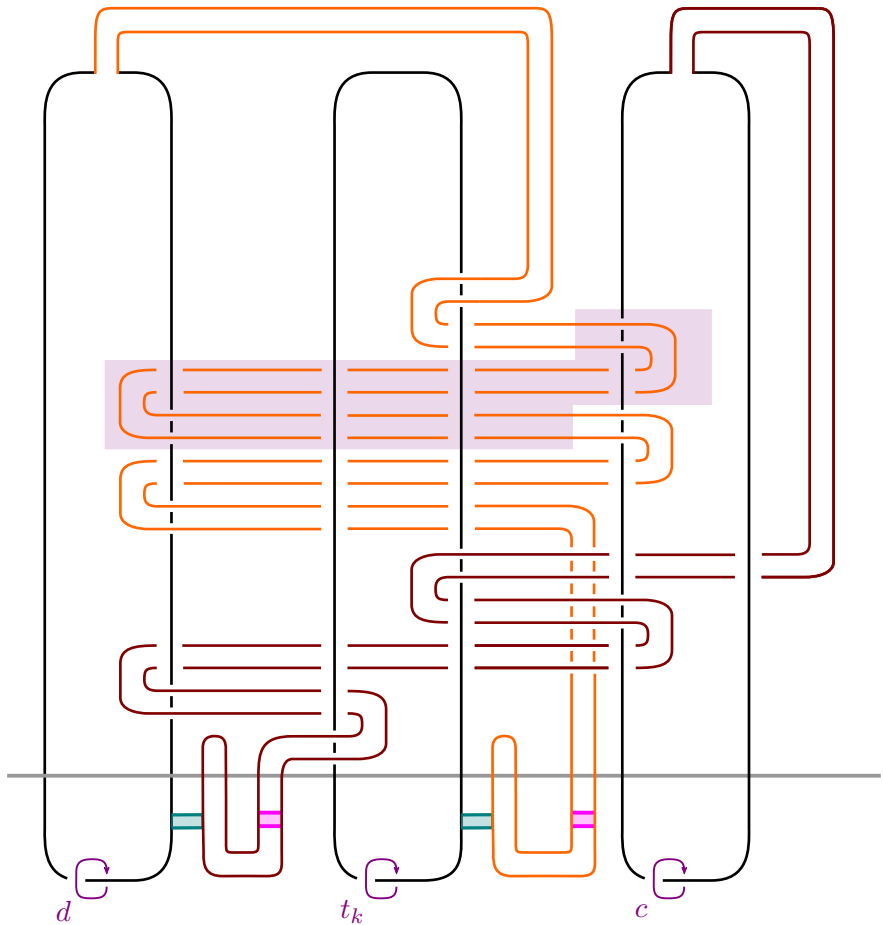

\begin{figure}[htbp]
	\centering
	\begin{tikzpicture}
		\node[anchor=south west,inner sep=0] at (0,0){\includegraphics[width=16cm]{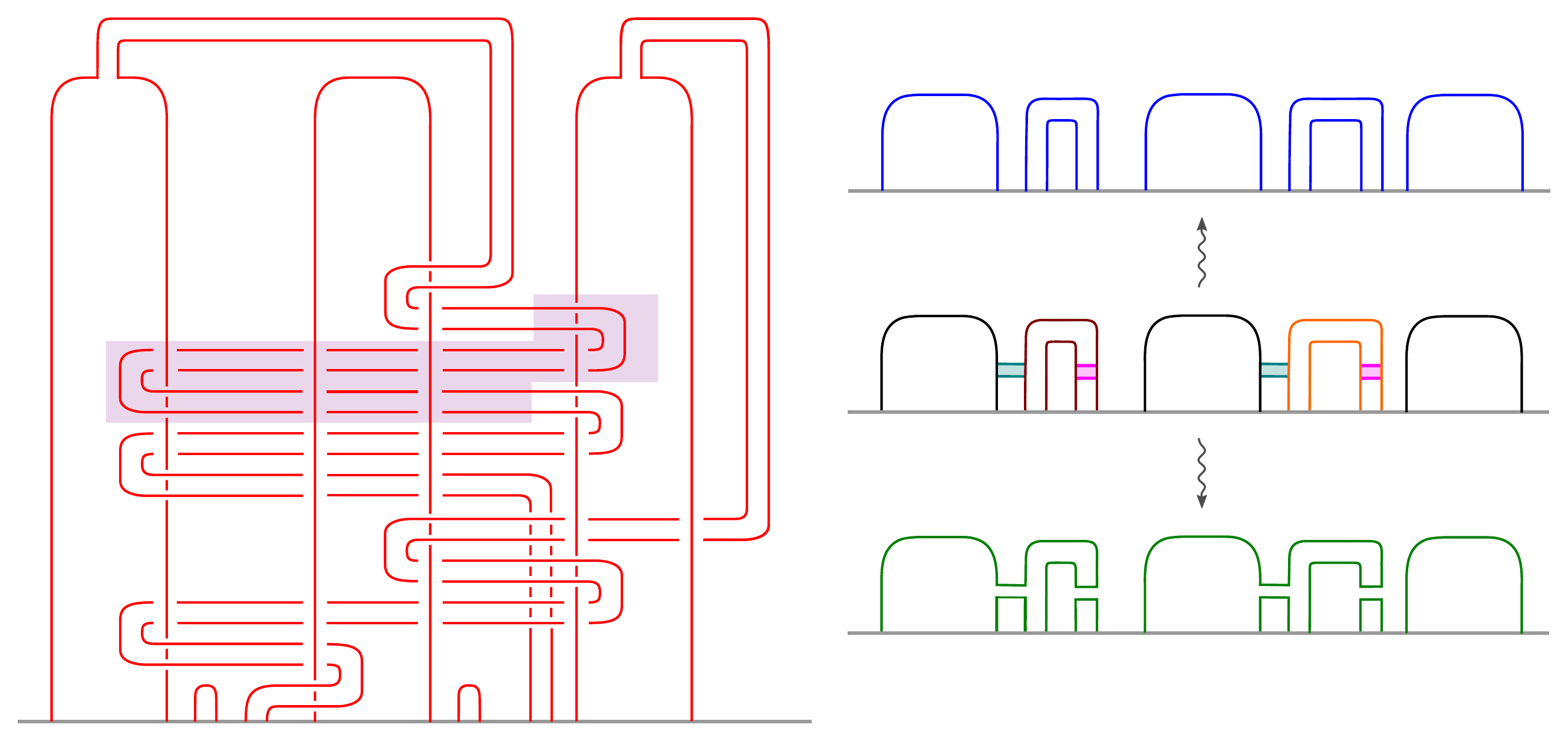}};
	\end{tikzpicture}
	\caption[tri-plane diagram of 7-bridge trisection of Suciu's knot $R_{3}$]{
		The three trivial tangles of a 7-bridge trisection of Suciu's knot $R_{k}$, for $k = 3$. 

		\label{fig:Suciu_R_3_bridge_trisection}
	}
\end{figure}

\begin{table}[htbp]
    \begin{center}
    \scriptsize
    \renewcommand{\arraystretch}{1.4}
        \begin{tabular}{c|c|c|c|c|c|c}
            \toprule
            $n$ &$k$ & Trisection  & $\pi_1$ & $|\pi_1|$ & $H_1$ & $H_2$ \\
            &&Parameters&&&& \\
            \midrule
            \midrule
            2&1& $(6;2,2,2)$& $\langle x_{12}^1 \mid (x_{12}^1)^3 \rangle$&3&$\mathbb{Z}_3$&$\mathbb{Z}_3$\\
            \hline
            2&2&$(6;2,2,2)$&$\langle x_6^1, x_{12}^1 \mid x_{12}^1x_6^1 (x_{12}^1)^{-2}x_6^1, (x_6^1)^{-1}x_{12}^1x_6^1 x_{12}^1 (x_6^1)^{-1} \rangle$ &24 & $\mathbb{Z}_3$&$\mathbb{Z}_3$ \\
            \hline
            2&3&$(6;2,2,2)$&$\langle x_6^1, x_{12}^1 \mid (x_6^1)^{-1} (x_{12}^1)^2 x_6^1 (x_{12}^1)^2 (x_6^1)^{-1},$&360&$\mathbb{Z}_3$&$\mathbb{Z}_3$\\
            &&& $x_{12}^1 (x_6^1)^{-1} x_{12}^1 (x_6^1)^2 x_{12}^1 (x_6^1)^{-1} \rangle$&&&\\
            \midrule
            \midrule
            3&1&$(12;4,4,4)$&$\langle x_{12}^1, x_{13}^1 \mid (x_{12}^1)^{-1} x_{13}^1 (x_{12}^1)^{-1} (x_{13}^1)^{-1}, (x_{13}^1)^{-1} (x_{12}^1)^2 (x_{13}^1)^{-1} \rangle$&8&$\mathbb{Z}_2^2$&$\mathbb{Z}_2^2$\\
            \hline
            3&2&$(12;4,4,4)$&$\langle x_6^1, x_7^1, x_{12}^1 \mid (x_{12}^1)^{-1} (x_7^1)^{-1} x_6^1 x_{12}^1 (x_6^1)^{-1} (x_7^1)^{-1} x_6^1 (x_{12}^1)^{-1}, $&960&$\mathbb{Z}_2^2$&$\mathbb{Z}_2^2$\\
            &&&$(x_7^1)^{-1} x_6^1 (x_{12}^1)^{-1}(x_6^1)^{-1} x_{12}^1 (x_6^1)^{-1} x_7^1 (x_{12}^1)^{-1},$&&&\\
            &&& $ (x_{12}^1)^{-1} x_6^1 x_{12}^1 ((x_7^1)^{-1} x_6^1)^2 (x_{12}^1)^{-1} x_7^1 x_{12}^1 (x_6^1)^{-1} x_7^1 \rangle $ &&&\\
            \hline
            3&3&$(12;4,4,4)$&$\langle x_6^1, x_{12}^1, x_{13}^1 \mid (x_6^1)^{-1} (x_{13}^1)^2((x_6^1)^{-1} (x_{13}^1)^{-1} x_6^1 (x_{12}^1)^{-1})^2, $&&$\mathbb{Z}_2^2$&$\mathbb{Z}_2^2$\\
            &&& $(x_{12}^1)^{-1} x_{13}^1 (x_6^1)^{-1} (x_{13}^1)^2 x_6^1 (x_{12}^1)^{-1} (x_6^1)^{-1} (x_{13}^1)^{-3} x_6^1, $  &&&\\ 
            &&& $x_{12}^1 (x_6^1)^{-1}x_{13}^1 (x_6^1)^2 (x_{12}^1)^{-2} (x_6^1)^{-1} x_{13}^1 x_6^1 (x_{12}^1)^3 (x_6^1)^{-1} \rangle $&&&\\
            \midrule
            \midrule
            4&1&$(18;6,6,6)$&$\langle x_{12}^1, x_{12}^2 \mid x_{12}^1 (x_{12}^2)^{-2} x_{12}^1 x_{12}^2, x_{12}^1 x_{12}^2 (x_{12}^1)^{-2} x_{12}^2 \rangle$&24&$\mathbb{Z}_3$&$\mathbb{Z}_3$\\
            \hline
            4&2&$(18;6,6,6)$&$\langle x_6^1, x_{13}^1 \mid (x_6^1)^{-1} (x_{13}^1)^2 (x_6^1) x_{13}^1 (x_6^1)^{-2} (x_{13}^1)^{-2} x_6^1 (x_{13}^1)^{-1} (x_6^1)^{-1} x_{13}^1, $&&$\mathbb{Z}_3$&$\mathbb{Z}_3$\\
            &&& $x_6^1 (x_{13}^1)^{-1} (x_6^1)^{-1}((x_6^1)^{-1} x_{13}^1)^2 x_{13}^1 (x_6^1)^{-1} (x_{13}^1)^{-1}((x_{13}^1)^{-1} x_6^1)^2 \rangle$  &&&\\
            \hline
            4&3&$(18;6,6,6)$&$\langle x_6^1, x_{12}^1, x_{13}^1 \mid ((x_6^1)^{-1}(x_{13}^1)^{-1} x_6^1 (x_{12}^1)^{-1})^2 (x_6^1)^{-1} (x_{13}^1)^2, $&&$\mathbb{Z}_3$&$\mathbb{Z}_3$\\
            &&&$(x_{13}^1)^{-1} x_{12}^1 (x_6^1)^{-1} (x_{13}^1)^3 x_6^1 (x_{12}^1)^{-2} (x_6^1)^{-1} x_{13}^1 x_6^1 (x_{12}^1)^3 (x_6^1)^{-1} (x_{13}^1)^{-2} x_6^1,$ &&&\\
            &&&$ x_{12}^1 (x_{12}^1 (x_6^1)^{-1})^2 x_{13}^1 (x_6^1)^2 (x_{12}^1)^{-2}  $&&&\\
            &&&  $\cdot(x_6^1)^{-1} (x_{13}^1) x_6^1 (x_{13}^1)^{-3} x_6^1 (x_{12}^1)^{-2} (x_6^1)^{-1} (x_{13}^1)^2 (x_{12}^1) (x_6^1)^{-1} x_{13}^1 x_6^1 x_{12}^1 \rangle$   &&&\\
            \midrule
            \midrule
            5&1&$(24;8,8,8)$&$ \langle x_{12}^1, x_{12}^2 \mid (x_{12}^1)^{-1} (x_{12}^2)^{-1} x_{12}^1 x_{12}^2 x_{12}^1 (x_{12}^2)^{-1},$&120&1&1\\
            &&& $x_{12}^1 (x_{12}^2)^{-1} x_{12}^1 
            x_{12}^2 (x_{12}^1)^{-2} x_{12}^2 \rangle$ &&&\\
            \hline
            5&2&$(24;8,8,8)$&$\langle x_{12}^3, x_6^4 \mid (x_{12}^3 x_6^4)^2 x_6^4 (x_{12}^3)^{-1} x_6^4 (x_{12}^3)^2 (x_6^4)^{-1}((x_6^4)^{-1} (x_{12}^3)^{-1})^2$&&1&1\\
            &&& $\cdot (x_{12}^3)^{-1} (x_6^4)^{-1} x_{12}^3, $ &&&\\
            &&& $((x_{12}^3)^{-1} (x_6^4)^{-1})^2 x_{12}^3 x_6^4 (x_{12}^3)^2 $ &&&\\
            &&& $\cdot(x_6^4)^{-1}((x_6^4)^{-1} (x_{12}^3)^{-1})^2 x_6^4 x_{12}^3 (x_6^4)^2 (x_{12}^3)^{-1} \rangle$  &&&\\
            \midrule
            \midrule
            6&1&$(30;10,10,10)$&$\langle x_{12}^3, x_{12}^4 \mid (x_{12}^3)^{-1} x_{12}^4 (x_{12}^3)^2 x_{12}^4 (x_{12}^3)^{-1} (x_{12}^4)^{-2}, $&$\infty$&$\mathbb{Z}^2$&$\mathbb{Z}^4$ \\
            &&& $x_{12}^3 (x_{12}^4)^{-1} (x_{12}^3)^{-1} x_{12}^4 (x_{12}^3)^{-1} (x_{12}^4)^{-1} x_{12}^3 x_{12}^4 \rangle$ &&&\\
            \bottomrule
        \end{tabular}
    \end{center}
    \caption{
        Invariants of the $n$-fold cyclic cover of the Suciu knot $R_k$.
        \label{suciu.tab}
    }
\end{table}

In \Cref{suciu.tab}, we compute invariants of the $n$-fold cyclic cover $\Sigma_n(R_k)$. \footnote{Details of this computation can be found in \texttt{suciu\_group\_trisection.ipynb} of \cite{Ruppik_Knot_groups_2021}.} When $n=2$, our computations recover the results of Kanenobu-Sumi \cite{kanenobu2020suciu}, who observe that for $k = 1, 2, 3$ the groups $\pi_{1}(\Sigma_{2}(R_{k}))$ are finite:
$\pi_{1}(\Sigma_{2}(R_{1})) \cong \ZZ / 3$ is cyclic of order $3$,
$\pi_{1}(\Sigma_{2}(R_{2}))$ is the binary tetrahedral group of order $24$ and
$\pi_{1}(\Sigma_{2}(R_{3}))$ is the direct product $\ZZ/3 \times \operatorname{Ico}^{*}$ of a cyclic group with the binary icosahedral group $\operatorname{Ico}^{*}$ which has order $120$.
We do not know whether $\pi_{1}(\Sigma_{2}(R_{k}))$ is finite or infinite for $k \ge 4$.  

\section{Bisections of Ribbon Surfaces and Application to Computing Ribbon Obstructions}
\label{xi.sec}

\subsection{Singular tri-plane diagrams and bisections of ribbon surfaces in the 4-ball}

Singular tri-plane diagrams, introduced in \cite{cahn2017singular}, are a generalization of tri-plane diagrams where one relaxes the condition that the one or more pairwise boundary unions of the trivial tangles are unlinks.  Depending on the application, one can either obtain a singular surface from such a diagram by coning off the entire link in a sector, or if the link is split, by coning off each individual split component.
In this section, we consider singular tri-plane diagrams $(A,B,C)$ such that $A\cup\overline{B}$ is a knot $K$, and $B\cup \overline{C}$ and $C\cup \overline{A}$ are unlinks.
The corresponding bridge-trisected surface $F$ in $\sphere{4}$ has one singularity, the cone on $K$.  By removing the sector containing the cone point, one obtains a {\it bisection} of $\disk{4}$ and a properly embedded, {\it bridge-bisected}, ribbon surface $F'$ with boundary $K$.
Moreover, it follows from the arguments in \cite{meier2017bridgeS4} that any ribbon surface in $\disk{4}$ can be bisected, and represented by a tri-plane diagram.

Now suppose we have a coloring $\rho: \pi_1(\disk{4}-F')\rightarrow S_n$, factored as a colored bridge bisection, and described by permutation labelling of the tri-plane diagram $(A,B,C)$, analogous to \Cref{coloredgptrisection.def} and \Cref{colored_tri_plane.def}.
Taking branched covers, we obtain a bisection of a 4-manifold $X'$ with boundary $M_\rho$, the branched cover of $\sphere{3}$ along $K$ corresponding to $\rho$.
When $M_\rho$ is homeomorphic to $\sphere{3}$, we can cone off the boundary of $F'$ to recover the bridge-trisected singular surface $F\subset \sphere{4}$, and also cone off $M_\rho$ upstairs, to obtain a trisection of $X$ with a 4-ball in one sector.
We can then use our algorithm to compute the corresponding group trisection, homology, and intersection form of $X$, from which we can recover the group bisection, homology, and intersection form of $X'$.
In particular, $\sigma(X)=\sigma(X')$, and we use this to compute Kjuchukova's ribbon obstruction for knots with $p$-fold dihedral cover $\sphere{3}$.  

More generally, when $\partial X'=M_\rho$ is an integer homology sphere, we have $H_2(X',\partial X')\simeq H_2(X')$. Thus \Cref{homology.alg} and \Cref{intersectionform.alg} can be used to calculate the second homology and intersection form of the branched cover of $\disk{4}$ along a such a ribbon surface $F'$.  

\subsection{Computing ribbon obstructions}

Kjuchukova \cite{kjuchukova2018dihedral} defined an invariant $\Xi_p(K,\rho)$ of a knot together with a choice of Fox $p$-coloring $\rho$ of $K$.
For each $p$, some value of $\Xi_p(K,\rho)$ (taken over all $p$-colorings $\rho$ of $K$) must fall in a bounded range when $K$ is homotopy-ribbon, meaning that $\Xi_p$ can be used to test potential counter-examples to the Slice-Ribbon Conjecture \cite{cahn2017singular, geske2021signatures}.  As noted above, the algorithms in this paper give, as a byproduct, a new algorithm for computing the invariant $\Xi_p$.

Viro \cite{viro1984signature} gave a formula for the signature of a branched cover of a 4-manifold in terms of the local branching indices and the self-intersection number of the branch set.  It follows from his result that the signature of a branched cover of $\sphere{4}$ with locally flat, orientable branching set must be zero; this is the case for the examples we have considered so far in \Cref{sec:twist_spun,stevedore.sec,sec:suciu}.  

In contrast, the invariant $\Xi_p(K,\rho)$ measures the {\it signature defect}, or the change in the signature of a $p$-fold dihedral branched cover resulting from the presence of a singularity, namely, the cone on the knot $K$, on the branching set.  

Precisely, let $f \colon X \rightarrow \sphere{4}$ be a $p$-fold dihedral cover with branching set $F$ and induced $p$-coloring $\rho$, and assume that $F$ is locally flat aside from one singular point $z$, modeled on the cone on the knot $K$.
Let $X'$ be obtained from $X$ by deleting a small 4-ball neighborhood of $z$. Then

\[
    \Xi_p(K,\rho) = \dfrac{p-1}{4}e(F) - \sigma(X',\partial X'),
\]

where $e(F)$ denotes the self-intersection number of $F$ and $\sigma(X',\partial X')$ denotes the Novikov signature of a manifold with boundary \cite{geske2021signatures}.
When $X$ is a manifold (equivalently, when the $p$-fold dihedral cover of $K$ is $\sphere{3}$) 

\[
    \Xi_p(K,\rho)
    =
    \dfrac{p-1}{4}e(F) - \sigma(X),
\]
where $\sigma(X)$ is the ordinary signature of $X$ \cite{kjuchukova2018dihedral}.
Therefore, given a representation of $F$ as a {\it singular} tri-plane diagram, one can determine the value of $\Xi_p$ using the algorithm for computing $\sigma(X)$ in this paper when $e(F)=0$ (e.g., when $F$ is orientable, which is the case in our examples).
An algorithm for computing $e(F)$ for a smooth surface from a tri-plane diagram is given in \cite{joseph2021bridgetrisectionclassical}; adapting this for surfaces with singularities would allow one to compute $\Xi_p$ in the general case, using the techniques of this paper to compute the signature term.

Note that we write $\Xi_p(K)$ instead of $\Xi_p(K,\rho)$ when $K$ has a single $p$-coloring up to equivalence. 

We construct a 2-parameter family of knots $K_{s,t}$ such that $K_{s,t}$ is $p$-colorable whenever $p$ divides $t$, and show that for small values of $s$ and $p$, $\Xi_p(K_{s,p})=\frac{s(p-1)}{2}$ in \Cref{tab:xi}. \footnote{Details of this computation can be found in \texttt{singular\_surface\_family.ipynb} of \cite{Ruppik_Knot_groups_2021}.}
This is the first explicit computation of the $\Xi_p$ invariant for $p>3$.
We note that the algorithm in this paper can only be used to compute $\Xi_p(K,\rho)$ when one has an explicit surface in $\disk{4}$ bounding $K$ over which $\rho$ extends; however, when one has such a surface, this algorithm is easier to use than the algorithm in \cite{cahn2018computing}, as it is fully automated.

\begin{figure}[htbp]
    \centering
    \includegraphics[width=6in]{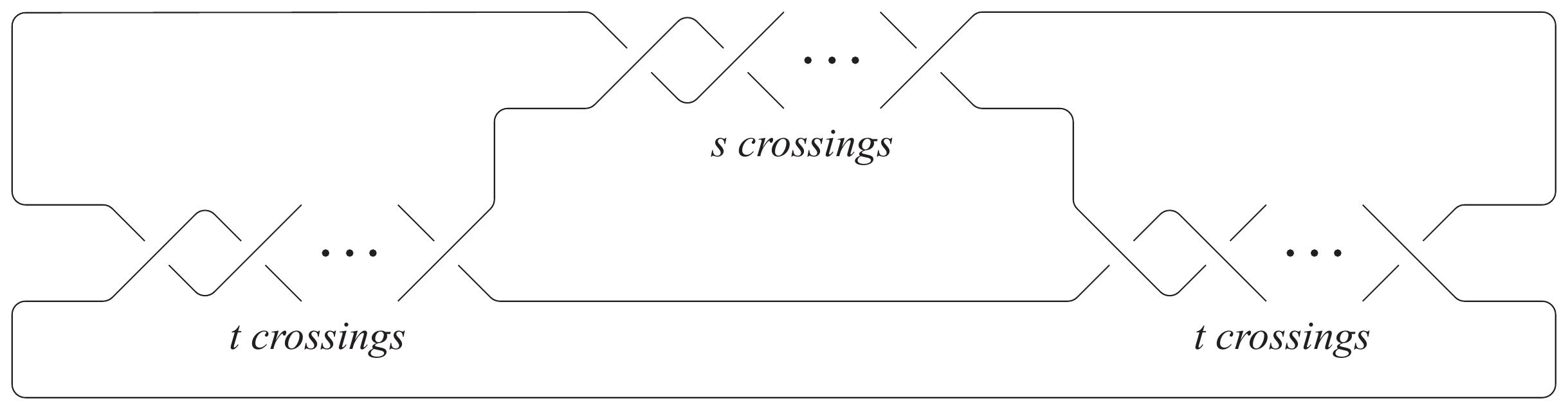}
    \caption{
        \label{singularityfamily.fig}
        A two-parameter family of knots $K_{s,t}$, with $s,t$ odd, such that $\Xi_p(K_{s,p})=\frac{s(p-1)}{2}$.
    }   
\end{figure}

\begin{figure}[htbp]
    \centering
    \includegraphics[width=7in]{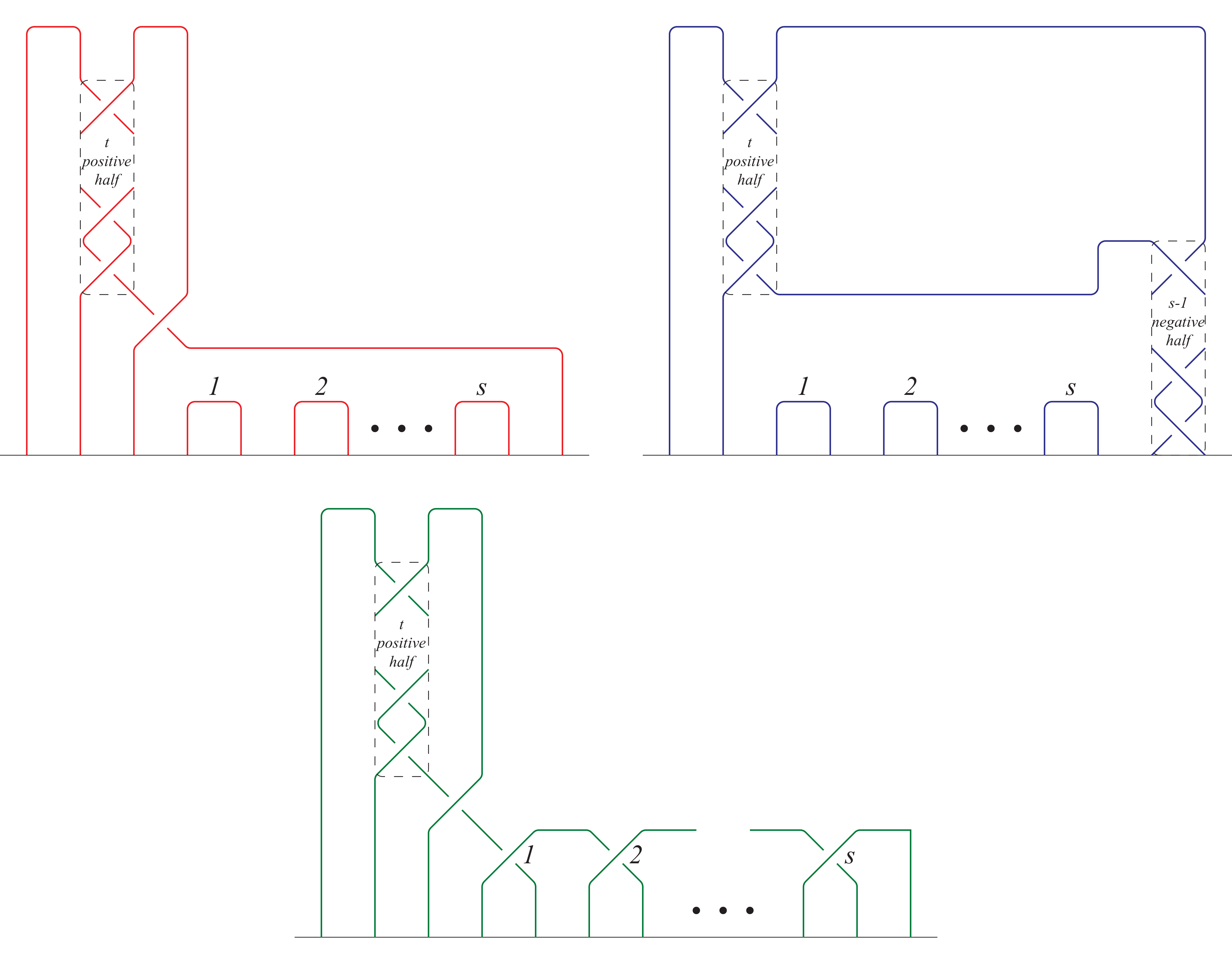}
    \caption{
        \label{singular_tri_plane.fig}
        A two-parameter family tri-plane diagrams $(A_{s,t},B_{s,t},C_{s,t})$ for a singular, orientable surface of genus $\frac{s-1}{2}$, with the property that $A_{s,t} \cup \overline{B}_{s,t}$ is the mirror image of the knot $K_{s,t}$, and $B_{s,t} \cup \overline{C}_{s,t}$ and $C_{s,t} \cup \overline{A}_{s,t}$ are 2-component unlinks.
    }
\end{figure}

A knot $K$ is strongly $p$-admissible if it has a $p$-coloring $\rho$ which extends over a surface in the 4-ball, and if the branched cover of $\sphere{3}$ along $K$ determined by $\rho$ is homeomorphic to $\sphere{3}$.
We note that when $p=3$, the knots $K_{s,3}$ were shown in \cite{cahn2020dihedral} to achieve the full range of values of $\Xi_3$ on strongly 3-admissible knots, namely the set of odd integers.
We wonder whether this is the case for other $p$.

\begin{question} Suppose that $K$ is a strongly $p$-admissible knot with respect to the coloring $\rho$.  Must $\Xi_p(K,\rho)$ be an odd multiple of $\frac{p-1}{2}$?
\end{question}

\begin{table}[htbp]
    \centering
    \renewcommand{\arraystretch}{1.2}
    \begin{tabular}{c|c|c|c|c|c|c|c|c}
        \toprule
        $s$ & $p$ & Trisection & $\pi_1$ & $\operatorname{rk} H_2$ & Intersection & Parity & $\sigma$ & $\Xi_p(K_{s,p})$ \\
        &&Parameters&&&Form&&&\\
        \midrule
        \midrule
        1&3&$(1;0,0,0)$&1&1&$\begin{pmatrix}-1\end{pmatrix}$&odd&-1&1\\
        \hline
        3&3&$(3;0,0,0)$&1&3&
        $\begin{pmatrix}
            -6&-11&-3\\-11&-21&-6\\-3&-6&-2
        \end{pmatrix}$
        &odd&-3&3\\
        \hline
        5&3&$(5;0,0,0)$&1&5&
        $\begin{pmatrix}
            -86 &-171 & -43&  -22&  -12\\
            -171& -341&  -86 & -44&  -24\\
            -43  &-86 & -22 & -11 &  -6\\
            -22 & -44 & -11 &  -6 &  -3\\
            -12 & -24  & -6  & -3 &  -2\\
        \end{pmatrix}$ & odd & -5 & 5 \\
        \midrule
        \midrule
        1&5&$(2;0,0,0)$&1&2&$\bigoplus_2 \begin{pmatrix}-1
        \end{pmatrix}$&odd&-2&2\\
        \hline
        3&5&$(6;0,0,0)$&1&6&
        $\bigoplus_2\begin{pmatrix}
              -6 &-11 & -3 \\
        -11 & -21 & -6 \\
         -3 & -6 & -2 \\
        \end{pmatrix}$&odd&-6&6\\
        \hline
        5&5&$(10;0,0,0)$&1&10&$\bigoplus_2\begin{pmatrix}
             -86& -171 & -43 & -22 & -12  \\ 
        -171& -341 & -86 & -44& -24   \\
         -43 & -86 & -22 & -11 &  -6  \\ 
         -22  &-44 & -11 &  -6 &  -3   \\
         -12  &-24 &  -6 &  -3 &  -2    
        \end{pmatrix}$&odd&-10&10\\
        \midrule
        \midrule
        1&7&$(3;0,0,0)$&1&3&$\bigoplus_3\begin{pmatrix}
            -1
        \end{pmatrix}$&odd&-3&3\\
        \hline
        3&7&$(9;0,0,0)$&1&9&$\bigoplus_3\begin{pmatrix}
        -6& -11&  -3  \\
        -11 &-21 & -6  \\
         -3 & -6 & -2   \\
        \end{pmatrix}$&odd&-9&9\\
        \hline
        5&7&$(15;0,0,0)$&1&15&$\bigoplus_3\begin{pmatrix}
             -86& -171 & -43 & -22 & -12  \\ 
        -171& -341 & -86 & -44& -24   \\
         -43 & -86 & -22 & -11 &  -6  \\ 
         -22  &-44 & -11 &  -6 &  -3   \\
         -12  &-24 &  -6 &  -3 &  -2    
        \end{pmatrix}$&odd&-15&15\\
        \midrule
        \midrule
        1&9&$(4;0,0,0)$&1&4&$\bigoplus_4\begin{pmatrix}
            -1
        \end{pmatrix}$&odd&-4&4\\
        \hline
        3&9&$(12;0,0,0)$&1&12&$\bigoplus_4\begin{pmatrix}
        -6& -11&  -3  \\
        -11 &-21 & -6  \\
         -3 & -6 & -2   \\
        \end{pmatrix}$&odd&-12&12\\
        \hline
        5&9&$(20;0,0,0)$&1&20&$\bigoplus_4\begin{pmatrix}
             -86& -171 & -43 & -22 & -12  \\ 
        -171& -341 & -86 & -44& -24   \\
         -43 & -86 & -22 & -11 &  -6  \\ 
         -22  &-44 & -11 &  -6 &  -3   \\
         -12  &-24 &  -6 &  -3 &  -2    
        \end{pmatrix}$&odd &-20&20\\
        \bottomrule
    \end{tabular}
    \caption{
        Invariants of $p$-fold irregular dihedral branched covers of $\sphere{4}$, whose branching sets have one cone singularity on the knot $K_{s,p}$.
        \label{tab:xi}
    }
\end{table}

\newpage
\printbibliography

\end{document}

%% file: commands.tex
\newcommand{\cs}{\mathop \#} 
\newcommand{\bcs}{\mathop \natural} 

\newcommand{\sphere}[1]{\mathbb{S}^{#1}}
\newcommand{\disk}[1]{\mathbb{D}^{#1}}

\newcommand{\ZZ}{\mathbb{Z}}

\newcommand{\Twist}{\tau}

\makeatletter
\g@addto@macro\@floatboxreset\centering
\makeatother

\allowdisplaybreaks[4]

\addto\extrasamerican{%
}

\newtheorem{proposition}{Proposition}[section]

\theoremstyle{definition}

\newtheorem{definition}{Definition}[section]
\newtheorem{example}{Example}[section]
\newtheorem{remark}{Remark}[section]
\newtheorem{question}{Question}[section]

\makeatletter
\let\c@conjecture=\c@theorem
\let\c@corollary=\c@theorem
\let\c@proposition=\c@theorem
\let\c@lemma=\c@theorem
\let\c@definition=\c@theorem
\let\c@example=\c@theorem
\let\c@remark=\c@theorem
\let\c@equation\c@theorem
\makeatother

\def\makeautorefname#1#2{\expandafter\def\csname#1autorefname\endcsname{#2}}
\makeautorefname{theorem}{Theorem}%
\makeautorefname{conjecture}{Conjecture}%
\makeautorefname{corollary}{Corollary}%
\makeautorefname{proposition}{Proposition}%
\makeautorefname{lemma}{Lemma}%
\makeautorefname{definition}{Definition}%
\makeautorefname{example}{Example}%
\makeautorefname{remark}{Remark}%
\makeautorefname{question}{Question}

\numberwithin{equation}{section}

\newcommand*{\Alphabet}{ABCDEFGHIJKLMNOPQRSTUVWXYZ1234567890}
\newcommand*{\alphabet}{abcdefghijklmnopqrstuvwxyz1234567890}
\newlength\fcaph
\newlength\fdesc
\newlength\factualfontsize
\numberwithin{equation}{section}
\newcounter{commentcounter}

\usepackage[dvipsnames]{xcolor}
\definecolor{coverred}{RGB}{113, 11, 21}
\definecolor{coveryellow}{RGB}{212,175,55}

\definecolor{lightBlueSlice}{rgb}{0, 0.87, 0.87}
\definecolor{myPatriarch}{rgb}{0.5, 0.0, 0.5}
\definecolor{myFuchsia}{RGB}{255, 0, 255} 
\definecolor{myDarkOrange}{RGB}{212, 85, 0}
\definecolor{myBrown}{RGB}{145, 111, 111}

\definecolor{myGreen}{rgb}{0.0, 0.5, 0.0}
\definecolor{myGraphRed}{RGB}{255, 85, 85}
\definecolor{myGraphGreen}{RGB}{90, 160, 44}

\definecolor{maroon(html/css)}{rgb}{0.5, 0.0, 0.0}
\definecolor{aqua}{rgb}{0.0, 1.0, 1.0}